\documentclass[a4paper,reqno,11pt]{amsart}
\usepackage{amssymb,amsmath,amsfonts,amsthm}
\usepackage{mathrsfs}

\numberwithin{equation}{section}
\usepackage[left=1 in, right=1 in,top=1 in, bottom=1 in]{geometry}

\providecommand{\abs}[1]{\left\vert#1\right\vert}
\providecommand{\norm}[1]{\left\Vert#1\right\Vert}
\providecommand{\pnorm}[2]{\left\Vert#1\right\Vert_{L^{#2}}}

\providecommand{\Rn}[1]{\mathbb{R}^{#1}}

\providecommand{\br}[1]{\left\langle #1 \right\rangle}
\providecommand{\snorm}[2]{\left\Vert#1\right\Vert_{H^{#2}}}
\providecommand{\snormspace}[3]{\left\Vert#1\right\Vert_{H^{#2}({#3})}}
\providecommand{\sd}[1]{\mathcal{D}_{#1}}
\providecommand{\se}[1]{\mathcal{E}_{#1}}
\providecommand{\sdb}[1]{\bar{\mathcal{D}}_{#1}}
\providecommand{\seb}[1]{\bar{\mathcal{E}}_{#1}}
\providecommand{\snorm}[2]{\left\Vert#1\right\Vert_{H^{#2}}}
\providecommand{\snormspace}[3]{\left\Vert#1\right\Vert_{H^{#2}({#3})}}
\providecommand{\ns}[1]{\norm{#1}^2}

\providecommand{\ip}[2]{\left(#1,#2\right)}
\providecommand{\iph}[3]{\left(#1,#2\right)_{\mathcal{H}^{#3}}}
\providecommand{\hn}[2]{\norm{#1}_{\mathcal{H}^{#2}}}

\def\nab{\nabla}
\def\al{\alpha}
\def\dt{\partial_t}

\def\hal{\frac{1}{2}}
\def\ep{\varepsilon}

\def\ls{\lesssim}
\def\p{\partial}
\def\pa{\partial^\alpha}
\def\sg{\mathbb{D}}

\def\da{\Delta_{\mathcal{A}}}
\def\naba{\nab_{\mathcal{A}}}
\def\diva{\diverge_{\mathcal{A}}}
\def\Sa{S_{\mathcal{A}}}
\def\Hsig{{_0}H^1_\sigma(\Omega)}
\def\H1{{_0}H^1(\Omega)}
\def\dis{\displaystyle}

\def\a{\mathcal{A}}
\def\f{\mathcal{F}_{2N}}
\def\g{\mathcal{G}_{2N}}
\def\h{\mathcal{H}}

\def\fj1{\mathcal{J}^{-1}}
\def\k{\mathcal{K}}
\def\n{\mathcal{N}}

\def\x{\mathcal{X}}
\def\y{\mathcal{Y}}
\def\z{\mathcal{Z}}

\def\rest{\hskip 1pt{\hbox to 10.8pt{\hfill\vrule height 7pt width 0.4pt depth 0pt\hbox{\vrule height 0.4pt
width 7.6pt depth 0pt}\hfill}}}

\def\evalu{\hskip 1pt{\hbox to 2pt{\hfill \vrule height -6pt width 0.4pt depth0pt}}}

\DeclareMathOperator{\diverge}{div}
\DeclareMathOperator{\supp}{supp}

\newtheorem{lem}{Lemma}[section]

\newtheorem{prop}[lem]{Proposition}
\newtheorem{thm}[lem]{Theorem}
\newtheorem{remark}[lem]{Remark}

\title[Viscous surface waves]{Zero surface tension limit of viscous surface waves}

\author{Zhong Tan}
\address{
School of Mathematical Sciences\\
Xiamen University\\
Xiamen, Fujian 361005, China}
\email[Z. Tan]{ztan85@163.com}

\author{Yanjin Wang}
\address{
School of Mathematical Sciences\\
Xiamen University\\
Xiamen, Fujian 361005, China}
\email[Y. J. Wang]{yanjin$\_$wang@xmu.edu.cn}

\thanks{This research was supported by the National Natural Science Foundation of China (No. 11201389, 11271305) and the Natural Science Foundation of Fujian Province of China (No. 2012J05011).}



\begin{document}
\begin{abstract}
We consider the free boundary problem for a layer of viscous, incompressible fluid in a uniform gravitational field, lying above a rigid bottom and below the atmosphere. For the ``semi-small" initial data, we prove the zero surface tension limit of the problem within a local time interval. The unique local strong solution with surface tension is constructed as the limit of a sequence of approximate solutions to a special parabolic regularization. For the small initial data, we prove the global-in-time zero surface tension limit of the problem.
\end{abstract}

\maketitle
\setcounter{tocdepth}{2}



\section{Introduction}

\subsection{Formulation of the equations in Eulerian coordinates}

We will follow Guo and Tice \cite{GT_lwp,GT_per,GT_inf} to formulate our problem.
We consider a viscous, incompressible fluid evolving in a moving domain
\begin{equation}
\Omega(t) = \{ y \in \Sigma \times \Rn{} \;\vert\; -b(y_1,y_2) < y_3 < \eta(y_1,y_2,t)\}.
\end{equation}
Here we assume that either $\Sigma = \Rn{2}$, or $\Sigma = (L_1 \mathbb{T}) \times (L_2 \mathbb{T})$ for $\mathbb{T} = \Rn{} / \mathbb{Z}$ the usual $1-$torus and $L_1, L_2 >0$ the periodicity lengths. The lower boundary of $\Omega(t)$, denoted by $\Sigma_b$, is assumed to be rigid and given, but the upper boundary, denoted by $\Sigma(t)$, is a free surface that is the graph of the unknown function $\eta: \Sigma\times\Rn{+} \to \Rn{}$. We assume that
\begin{equation}
\begin{cases}
0 < b_-\le b \in C^\infty(\Sigma)  & \text{if } \Sigma = (L_1 \mathbb{T}) \times (L_2 \mathbb{T}) \\
0 < b_-\le b \in C^\infty(\Sigma) \text{ and }b\rightarrow b_\infty>0 \text{ as }|(y_1,y_2)|\rightarrow\infty & \text{if } \Sigma = \Rn{2}
\end{cases}
\end{equation}
for some constants $b_-,b_\infty$.
The fluid is described by its velocity and pressure functions, which are given for each $t\ge0$ by $ u(\cdot,t):\Omega(t) \to \Rn{3} $ and $ p(\cdot,t):\Omega(t) \to \Rn{}$, respectively.  For each $t>0$ we require that $(u, p, \eta)$ satisfy the following equations
\begin{equation}\label{ns_euler}
\begin{cases}
\partial_t u + u \cdot \nabla u + \nabla p-\mu  \Delta u  = 0& \text{in }
\Omega(t) \\
\diverge{u}=0 & \text{in }\Omega(t) \\
(p I - \mu \mathbb{D}(u) ) \nu = g \eta \nu-\sigma H \nu & \text{on } \Sigma(t) \\
\partial_t \eta = u_3 - u_1 \partial_{y_1}\eta - u_2 \partial_{y_2}\eta &
\text{on }\Sigma(t) \\
u = 0 & \text{on } \Sigma_b%
\end{cases}%
\end{equation}
for $I$ the $3 \times 3$ identity matrix,  $(\mathbb{D} u)_{ij} = \partial_i u_j + \partial_j u_i$ the symmetric gradient of $u$, $\mu>0$ the viscosity, and $g>0$ the strength of gravity. The tensor $p I - \mu \mathbb{D}(u)$ is known as the viscous stress tensor.
We take $\sigma> 0$ to be the constant coefficient of surface tension, and it should be understood that when $\sigma=0$ it means the case without surface tension.
In this paper, we let $\nabla_\ast$ denote the horizontal gradient, $\diverge_\ast$ denote the horizontal divergence and $\Delta_\ast$ denote the horizontal Laplace operator. Then the outward-pointing unit normal on $\Sigma(t)$, $\nu$, is given by
\begin{equation}
\nu=\frac{(-\nabla_\ast\eta ,1)}
{\sqrt{1+|\nabla_\ast\eta |^2}},
\end{equation}
and  $H $, twice the mean curvature of the surface $\Sigma(t)$, is given by the formula
\begin{equation}
H =\diverge_\ast\left(\frac{\nabla_\ast\eta }
{\sqrt{1+|\nabla_\ast\eta|^2}}\right).
\end{equation}
The fourth equation in \eqref{ns_euler} is called the kinematic boundary condition which implies that the free surface is advected with the fluid.  Note that in \eqref{ns_euler} we have shifted the gravitational forcing to the boundary and eliminated the constant atmospheric pressure, $p_{atm}$, in the usual way by adjusting the actual pressure $\bar{p}$ according to $p = \bar{p} + g y_3 - p_{atm}$. Without loss of generality, we may assume that $\mu = g = 1$. For a more physical description of the equations \eqref{ns_euler} and the boundary conditions in \eqref{ns_euler}, we refer to \cite{WL}.

To complete the formulation of the problem, we must specify the initial conditions. We suppose that the initial surface $\Sigma(0)$ is given by the graph of the function $\eta(0)=\eta_0: \Sigma\rightarrow \mathbb{R}$, which yields the initial domain $\Omega(0)$ on which we specify the initial data for the velocity, $u(0)=u_0: \Omega(0) \rightarrow \mathbb{R}^3$. We will assume that $\eta_0 > -b$ on $\Sigma$ and that $(u_0, \eta_0)$ satisfy certain compatibility conditions, which we will describe in detail later.  When $\Sigma = (L_1 \mathbb{T}) \times (L_2 \mathbb{T})$ we shall refer to the problem as the ``periodic case'',  and when $\Sigma = \Rn{2}$  we shall refer to it as  the ``non-periodic case''. In this paper, for the local-in-time analysis we will consider both the non-periodic and periodic cases simultaneously.  When different analysis is needed for each case, we will indicate so.  Otherwise, the argument we write works in both cases. For the global-in-time analysis we will only consider the periodic case for simplify.

\subsection{Geometric form of the equations}
The movement of the free boundary $\Sigma(t)$ and the subsequent change of the domain $\Omega(t)$  create numerous mathematical difficulties. To circumvent these, as usual, we will transform the free boundary problem under consideration to a problem with a fixed domain and fixed boundary.
We will not use a Lagrangian coordinate transformation as in \cite{beale_1,solon}, but rather a flattening transformation introduced by Beale in \cite{beale_2}.  To this end, we consider the fixed equilibrium domain
\begin{equation}
\Omega:= \{x \in \Sigma \times \Rn{} \; \vert\;  -b(x_1,x_2) < x_3 < 0  \}
\end{equation}
for which we will write the coordinates as $x\in \Omega$.  We will think of $\Sigma$ as the upper boundary of $\Omega$, and we continue to view $\eta$ as a function on $\Sigma \times \Rn{+}$.  We then define \begin{equation}
 \bar{\eta}:= \mathcal{P} \eta = \text{harmonic extension of }\eta \text{ into the lower half space},
\end{equation}
where $\mathcal{P} \eta$ is defined by \eqref{poisson_def_inf} when $\Sigma = \Rn{2}$ and by \eqref{poisson_def_per} when $\Sigma = (L_1 \mathbb{T}) \times (L_2 \mathbb{T})$.  The harmonic extension $\bar{\eta}$ allows us to flatten the coordinate domain via the mapping
\begin{equation}\label{mapping_def}
 \Omega \ni x \mapsto   (x_1,x_2, x_3 +  \bar{\eta}(x,t)(1+ x_3/b(x_1,x_2) )) := \Phi(x,t) = (y_1,y_2,y_3) \in \Omega(t).
\end{equation}
Note that $\Phi(\Sigma,t) = \Sigma(t)$ and $\Phi(\cdot,t)\vert_{\Sigma_b} = Id_{\Sigma_b}$, i.e. $\Phi$ maps $\Sigma$ to the free surface and keeps the lower surface fixed.   We have
\begin{equation}\label{A_def}
 \nab \Phi =
\begin{pmatrix}
 1 & 0 & 0 \\
 0 & 1 & 0 \\
 A & B & J
\end{pmatrix}
\text{ and }
 \mathcal{A} := (\nab \Phi^{-1})^T =
\begin{pmatrix}
 1 & 0 & -A K \\
 0 & 1 & -B K \\
 0 & 0 & K
\end{pmatrix}
\end{equation}
for
\begin{equation}\label{ABJ_def}
\begin{split}
A &= \p_1 \bar{\eta} \tilde{b} -( x_3 \bar{\eta} \p_1 b )/b^2,\;\;\;  B = \p_2 \bar{\eta} \tilde{b} -( x_3 \bar{\eta} \p_2 b )/b^2,  \\
J &=  1+ \bar{\eta}/b + \p_3 \bar{\eta} \tilde{b},  \;\;\; K = J^{-1}, \\
\tilde{b}  &= (1+x_3/b).
\end{split}
\end{equation}
Here $J = \det{\nab \Phi}$ is the Jacobian of the coordinate transformation.

If $\eta$ is sufficiently small (in an appropriate Sobolev space), then the mapping $\Phi$ is a diffeomorphism.  This allows us to transform the problem to one on the fixed spatial domain $\Omega$ for $t \ge 0$.  In the new coordinates, the system \eqref{ns_euler} becomes
\begin{equation}\label{geometric}
 \begin{cases}
  \dt u - \dt \bar{\eta} \tilde{b} K \p_3 u + u \cdot \naba u -\da u + \naba p     =0 & \text{in } \Omega \\
 \diva u = 0 & \text{in }\Omega \\
 S_\a(p,u) \n = \eta \n-\sigma H \n & \text{on } \Sigma \\
 \dt \eta = u \cdot \n & \text{on } \Sigma \\
 u = 0 & \text{on } \Sigma_b \\
 (u , \eta)\mid_{t=0} = (u_0,\eta_0).
 \end{cases}
\end{equation}
Here we have written the differential operators $\naba$, $\diva$, and $\da$ with their actions given by $(\naba f)_i := \a_{ij} \p_j f$, $\diva X := \a_{ij}\p_j X_i$, and $\da f := \diva \naba f$ for appropriate $f$ and $X$.  We have also written  $\n := (-\p_1 \eta, - \p_2 \eta,1)$ for the non-unit normal to $\Sigma(t)$,  and we have written $\Sa(p,u): = p I  - \sg_{\a} u$ for the stress tensor, where $(\sg_{\a} u)_{ij}: = \a_{ik} \p_k u_j + \a_{jk} \p_k u_i$ is the symmetric $\a-$gradient.  Note that if we extend $\diva$ to act on symmetric tensors in the natural way, then $\diva \Sa(p,u) = \naba p - \da u$ for vector fields satisfying $\diva u=0$.

Recall that $\a$ is determined by $\eta$ through the relation \eqref{A_def}.  This means that all of the differential operators in \eqref{geometric} are connected to $\eta$, and hence to the geometry of the free surface.  This geometric structure is essential to our analysis, as it allows us to control high-order derivatives that would otherwise be out of reach. This was well explained in \cite{GT_lwp,GT_per,GT_inf}.

\subsection{Main results}
Free boundary problems in fluid mechanics have been studied by many authors in many different contexts. Here we will mention only the work most relevant to our present setting. For a more thorough review of the literature, we refer to the review paper of Shibata and Shimizu \cite{ShSh} and the references therein.

The problem of describing the motion of an isolated mass of viscous fluid bounded by a free boundary was studied by Solonnikov in a series of papers; for instance, \cite{solon} considers the problem without surface tension, while \cite{So_2,So_3} concern the problem with surface tension.  Solonnikov's technique for proving the well-posedness in these papers did not rely on energy methods, but rather on H\"older estimates in the case without surface tension and Fourier-Laplace transform methods in the case with surface tension.  A local well-posedness for the problem with surface tension that based on the energy method was developed by Coutand and Shkoller \cite{coutand_shkoller_2}.

The viscous surface wave problem in our present setting that describes the motion of a layer of viscous fluid lying above a fixed bottom  has attracted the attention of many mathematicians since the pioneering work of Beale \cite{beale_1}. For the problem without surface tension, Beale \cite{beale_1} proved the local well-posedness in the Sobolev spaces and Sylvester \cite{sylvester} studied the global well-posedness by using Beale's method. For the periodic case with a completely flat fixed bottom, Hataya \cite{hataya} obtained the global existence of small solutions with an algebraic decay rate in time. Recently, Guo and Tice \cite{GT_lwp,GT_per,GT_inf} developed a new two-tier energy method to prove the global well-posedness and decay of this problem; they proved that the solution for the non-periodic case decays to the equilibrium at an algebraic rate, while the solution for the periodic case decays at an almost exponential rate. For the problem with surface tension, Beale \cite{beale_2} proved the global well-posedness in the Sobolev spaces, while Allain \cite{A} obtained a local existence theorem in two dimension using a different method. Bae \cite{bae} showed the global solvability in Sobolev spaces via the energy method. Beale and Nishida \cite{beale_nishida} showed that the solution  for the non-periodic case obtained in \cite{beale_2} decays at an optimal algebraic rate. For the periodic case with a flat fixed bottom, Nishida, Teramoto and Yoshihara \cite{nishida_1}  showed the global existence of small solutions with an exponential decay rate. Tani \cite{Ta} and Tani and Tanaka \cite{tani_tanaka} also discussed the solvability of the problem with or without surface tension within Beale-Solonnikov's functional framework.

Formally, when the surface tension coefficients $\sigma$ tend to
zero the solutions of the problem \eqref{geometric} with $\sigma>0$
should converge to the limit that solves the problem without surface
tension. However, there is no general theory that guarantees such
asymptotic limit; it depends on whether one can develop an
appropriate well-posedness theory for the problem which is unified
in the surface tension. We note that the well-posedness theories for
the problem with surface tension mentioned above highly depend on
the surface tension, so it can not be applied to show the zero
surface tension limit. Motivated by the recent works of Guo and Tice
\cite{GT_lwp,GT_per,GT_inf}, in this paper we shall develop a {\it
unified} framework of nonlinear energy method of proving the
well-posedness of \eqref{geometric} with and without surface
tension. Hence, it will allow us to rigorously justify the zero
surface tension limit in \eqref{geometric}. Moreover, for the small
initial data this limit is showed to be global in time. We remark
that the problem of the zero surface tension limit was considered in
the context of the water wave problem by Ambrose and Masmoudi
\cite{AM1,AM2} and  the Stefan problem by Hadzic and Shkoller
\cite{HS}, and both of them are restricted within a local time
interval.

Before we state our results, let us now mention the issue of compatibility conditions for the initial data $(u_0,\eta_0)$.  We will work in a high-regularity context, essentially with regularity up to $2N$ temporal derivatives.  This requires us to use $u_0$ and $\eta_0$ to construct the initial data $\dt^j u(0)$ and $\dt^j \eta(0)$ for $j=1,\dotsc,2N$ and $\dt^j p(0)$ for $j = 0,\dotsc, 2N-1$.  These other data must then satisfy various conditions (essentially what one gets by applying $\dt^j$ to \eqref{geometric} and then setting $t=0$), which in turn require $u_0$ and $\eta_0$ to satisfy $2N$ compatibility conditions.  We describe these conditions in detail in Section \ref{l_data_section} and state them explicitly in \eqref{l_comp_cond_2N}.

To state out our results, we now define some quantities. We shall recall from Section \ref{nota} our notation for Sobolev spaces and norms. For a generic integer $n\ge 3$, we define the  energy as
\begin{equation}\label{p_energy_def}
 \se{n}^\sigma: = \sum_{j=0}^{n}  \ns{\dt^j u}_{2n-2j}   + \sum_{j=0}^{n-1} \ns{\dt^j p}_{2n-2j-1}+\sigma\ns{\eta}_{2n+1}+\ns{\eta}_{2n} +\sum_{j=1}^{n+1}  \ns{\dt^j \eta}_{2n-2j+3/2}
\end{equation}
and the corresponding dissipation as
\begin{equation}\label{p_dissipation_def}
\begin{split}
 \sd{n}^\sigma := &\sum_{j=0}^{n} \ns{\dt^j u}_{2n-2j+1} + \sum_{j=0}^{n-1} \ns{\dt^j p}_{2n-2j}+  \sigma^2\ns{ \eta}_{2n+3/2} +  \sigma^2\ns{ \dt\eta}_{2n+1/2} \\
&+  \ns{ \eta}_{2n-1/2}+ \ns{\dt \eta}_{2n-1/2} + \sum_{j=2}^{n+1} \ns{\dt^j \eta}_{2n-2j+5/2}.
\end{split}
\end{equation}
For our use, we will take both $n=2N$ and $n=N+2$. We define
\begin{equation}\label{00eta}
\mathcal{E}_0^\sigma:=\ns{u_0^\sigma}_{4N}+\ns{\eta_0^\sigma}_{4N+1/2}+ \sigma\ns{\eta_0^\sigma}_{4N+1}.
\end{equation}
Note that we allow for $\sigma=0$ in the definitions of energy functionals throughout the paper. We also define
\begin{equation}\label{fff}
\f:=  \ns{\eta}_{4N+1/2}  .
\end{equation}

We first state the local-in-time results for the ``semi-small" initial data; we take either $\Sigma = \Rn{2}$ or $\Sigma = (L_1 \mathbb{T}) \times (L_2 \mathbb{T})$.
For any integer $N\ge 3$, we define
\begin{equation} \label{energysigma}
\mathcal{K}_{2N}^\sigma(t):=   \sup_{0 \le r \le t} \se{2N}^\sigma(r) + \int_0^t \sd{2N}^\sigma(r) dr
+ \ns{\dt^{2N+1} u}_{(\x_t)^*} + \sup_{0 \le r \le t}  \f(r),
\end{equation}
where we have referred to the space $\x_t$ defined later in \eqref{X_def}.
\begin{thm}\label{lwp}
Let $N\ge 3$ be an integer.
Assume that the initial data $ u_0^\sigma$ and $\eta_0^\sigma$ satisfy the bound $\mathcal{E}_0^\sigma<\infty$ as well as the $(2N)^{th}$ compatibility conditions  \eqref{l_comp_cond_2N} with $0<\sigma\le 1$. There exists a universal $\varepsilon_0>0$ such that if $\ns{\eta_0^\sigma}_{4N-1/2}\le \varepsilon_0/2$, then there exists a $T_0=T_0(\mathcal{E}_0^\sigma)>0$ so that there exists a unique solution $(u^\sigma,p^\sigma,\eta^\sigma)$ to \eqref{geometric} with initial data $(u_0^\sigma,\eta_0^\sigma)$ on the time interval $[0,T_0]$. The solution $(u^\sigma,p^\sigma,\eta^\sigma)$ obeys the estimates
\begin{equation}\label{zero_1}
\mathcal{K}_{2N}^\sigma(u^\sigma,p^\sigma,\eta^\sigma)(T_0)
\le   P ( \mathcal{E}_0^\sigma)\text{ and } \sup_{0\le t\le T_0}\ns{\eta^\sigma(t)}_{ {4N-1/2}}\le \varepsilon_0
\end{equation}
for a universal polynomial $P$ with $P(0)=0$.

As a consequence, if as $\sigma\rightarrow0$, $u_0^\sigma\rightarrow u_0$ in $H^{4N}(\Omega)$, $\eta_0^\sigma\rightarrow \eta_0$ in $H^{4N+1/2}(\Sigma)$ and $\sqrt{\sigma} \eta_0^\sigma \rightarrow0$ in $H^{4N+1}(\Sigma)$, then $(u^\sigma,p^\sigma,\eta^\sigma)$ converges to the limit $(u,p,\eta)$ which is the unique solution to \eqref{geometric} for $\sigma=0$ with initial data $(u_0,\eta_0)$ on  $[0,T_0]$. The limit $(u,p,\eta)$ obeys the estimates
\begin{equation}\label{zero_123}
\mathcal{K}_{2N}^0(u,p,\eta)(T_0)
\le   P (\mathcal{E}_0)\text{ and } \sup_{0\le t\le T_0}\ns{\eta(t)}_{ {4N-1/2}} \le \varepsilon_0
\end{equation}
\end{thm}

We now state the global-in-time results for the small initial data; we take $\Sigma = (L_1 \mathbb{T}) \times (L_2 \mathbb{T})$. We assume further that the initial surface function satisfies the ``zero-average'' condition
\begin{equation}\label{z_avg}
\frac{1}{L_1 L_2} \int_\Sigma \eta_0 =0.
\end{equation}
For any integer $N\ge 3$, we define
\begin{equation}
\mathcal{G}_{2N}^\sigma (t) := \sup_{0 \le r \le t} \se{2N}^\sigma(r) + \int_0^t \sd{2N}^\sigma(r) dr + \sup_{0 \le r \le t} (1+r)^{4N-8} \se{N+2}^\sigma(r) + \sup_{0 \le r \le t} \frac{\f(r)}{(1+r)}.
\end{equation}

\begin{thm}\label{gwp}
Let $N\ge 3$ be an integer.
Assume that the initial data $ u_0^\sigma$ and $\eta_0^\sigma$ satisfy the $(2N)^{th}$ compatibility conditions  \eqref{l_comp_cond_2N} with $0<\sigma\le 1$, and that $\eta_0^\sigma$ satisfies the zero average condition \eqref{z_avg}. There exists a universal $\delta_0>0$ such that if $\mathcal{E}_0^\sigma\le \delta_0$, then there exists a unique solution $(u^\sigma,p^\sigma,\eta^\sigma)$ to \eqref{geometric} with initial data $(u_0^\sigma,\eta_0^\sigma)$ on  $[0,\infty)$.  The solution $(u^\sigma,p^\sigma,\eta^\sigma)$ obeys the estimate
\begin{equation}\label{zz22}
\g^\sigma(u^\sigma,p^\sigma,\eta^\sigma)(\infty)
\le  C \mathcal{E}_0^\sigma
\end{equation}
for a universal constant $C>0$.

As a consequence, if as $\sigma\rightarrow0$, $u_0^\sigma\rightarrow u_0$ in $H^{4N}(\Omega)$, $\eta_0^\sigma\rightarrow \eta_0$ in $H^{4N+1/2}(\Sigma)$ and $\sqrt{\sigma} \eta_0^\sigma \rightarrow0$ in $H^{4N+1}(\Sigma)$, then $(u^\sigma,p^\sigma,\eta^\sigma)$ converges to the limit $(u,p,\eta)$ which is the unique solution to \eqref{geometric} for $\sigma=0$ with initial data $(u_0,\eta_0)$ on  $[0,\infty)$. The limit $(u,p,\eta)$ obeys the estimate
\begin{equation}\label{zz33}
\g^0(u,p,\eta)(\infty)
\le  C \mathcal{E}_0.
\end{equation}
\end{thm}

\begin{remark}
In Theorem \ref{lwp} (resp. Theorem \ref{gwp}) the convergence of $(u^\sigma,p^\sigma,\eta^\sigma)$ to $(u ,p ,\eta )$ is weak  in the norms in the definition of $\mathcal{K}_{2N}^0$ (resp. $\g^0$)
and strong in the norms of any functional space that $\mathcal{K}_{2N}^0$ (resp. $\g^0$) can compactly embed into. The convergence is more that sufficient for us to pass to the limit as $\sigma\rightarrow0$ in \eqref{geometric}.
\end{remark}

\begin{remark}
Note that in Theorem \ref{lwp} we only assume that the $H^{4N-1/2}$ norm of $\eta_0$ is small, but the higher-order norm of $\eta_0$ and the norm of $u_0$ can be arbitrarily large. In this sense we may regard the initial data as to be ``semi-small".
Theorems \ref{lwp}--\ref{gwp} contain the local well-posedness of the viscous surface wave problem \eqref{geometric} both with and without surface tension for the ``semi-small" initial data and global well-posedness for the small initial data. Moreover, the solution without surface tension decays at an almost exponential rate and the solution with surface tension decays at an exponential rate (cf. Theorem \ref{globall}). Since $\eta$ is such that the mapping $\Phi(\cdot,t)$, defined by \eqref{mapping_def}, is a diffeomorphism ($C^{4N-3/2}$ for $\sigma>0$ and $C^{4N-2}$ for $\sigma=0$) for each $t>0$, we may change coordinates to $y \in \Omega(t)$ to produce solutions to \eqref{ns_euler}.
\end{remark}

We now explain the strategy of proving Theorems \ref{lwp} and \ref{gwp}. After collecting some preliminaries in Section \ref{sec pre}, we start to prove Theorem \ref{lwp} of the local-in-time results. The key step is to establish a {\it unified} local well-posedness theory for the problem \eqref{geometric} with and without surface tension. We manage to first produce the local unique strong solutions to \eqref{geometric} for each fixed $\sigma>0$. In the framework of the energy method, it is more natural to use the geometric formulation \eqref{geometric} rather than the perturbed form as well explained in Guo and Tice \cite{GT_lwp,GT_per,GT_inf}. For the local well-posedness without surface tension, Guo and Tice \cite{GT_lwp} used an iteration scheme based on the solvability of the two linear problems which are separated from \eqref{geometric}, that is, the linear $\a$--Navier--Stokes equations for $(u,p)$, where   $\a$ and $\n$ are given (in terms of $\eta$),
\begin{equation}\label{intro_A_NS}
 \begin{cases}
\dt u - \da u + \naba p = F^1 & \text{in }\Omega \\
\diva{u}=0 & \text{in }\Omega\\
\Sa(p,u) \n = F^3 & \text{on }\Sigma \\
u =0 & \text{on }\Sigma_b,
 \end{cases}
\end{equation}
and the linear transport equation for $\eta$, where $u$ is given,
\begin{equation}\label{intro_A_T}
\dt \eta + u\cdot\nab_\ast\eta = u_3\ \text{ on }\Sigma.
\end{equation}
Hereafter $u\cdot\nab_\ast\eta:=u_1 \p_1 \eta + u_2 \p_2 \eta$, and
etc. A subtle point for the iteration in \cite{GT_lwp} could be
useful is that for the nonlinear problem \eqref{geometric} without
surface tension, $F^3=\eta\n$ ($\sim\nab\bar{\eta}\sim \nab u$);
this would imply that the regularity of $\eta$ obtaining from
\eqref{intro_A_T} ($ \eta\sim \int_0^tu_3$) can control the
nonlinear term $F^3$. However, for the problem \eqref{geometric}
with surface tension,  $F^3=\eta \n-\sigma H\n$
($\sim\nab_\ast^2\eta\sim \nab u$); then the regularity of $\eta$
obtaining from \eqref{intro_A_T} is insufficient to control the
nonlinear term $F^3$. Hence, we can not apply the iteration scheme
of \cite{GT_lwp} directly. One may then instead use an iteration
scheme based on the solvability of the following linear problem for
$(u,p,\eta)$, where $\a$ and $\n$ are given (in terms of, say,
$\zeta$),
\begin{equation}\label{lld}
 \begin{cases}
  \dt u - \da u + \naba p = F^1  & \text{in } \Omega \\
 \diva u = 0 & \text{in }\Omega \\
 S_\a(p,u) \n = \eta \n-\sigma \Delta_\ast\eta\n+F^3  & \text{on } \Sigma \\
\dt \eta = u \cdot \n & \text{on } \Sigma\\
  u = 0 & \text{on } \Sigma_b .
 \end{cases}
\end{equation}
We may remark here that we could solve the problem \eqref{lld} in a weak sense by using a time-dependent Galerkin method similarly as \cite{GT_lwp} (restricted to the $\diverge_\a $--free vectors), however, the standard way of improving the regularity by using the horizontal difference quotient as in \cite{coutand_shkoller_2,bae} would fail to improve the regularity of the weak solution due to that the horizontal difference quotient does not commute with the projection along the normal $\n$.

Our way of getting around the difficulty explained above is to introduce an artificial viscosity term in the transport equation, that is, we consider the following regularized $\kappa$-problem, where $\kappa>0$ is the artificial viscosity coefficient,
\begin{equation} \label{lldsd}
 \begin{cases}
  \dt u - \dt \bar{\eta} \tilde{b} K \p_3 u + u \cdot \naba u -\da u + \naba p     =0 & \text{in } \Omega \\
 \diva u = 0 & \text{in }\Omega \\
 S_\a(p,u) \n =  \eta\n -\sigma H \n & \text{on } \Sigma \\
 \dt \eta -\kappa\Delta_\ast\eta=\kappa\Psi+ u \cdot \n & \text{on } \Sigma \\
 u = 0 & \text{on } \Sigma_b.
 \end{cases}
\end{equation}
Note that we have also introduced along the so-called compensator
function $\Psi$ in the transport equation owing to the issue of
compatibility conditions for the initial data. The function $\Psi$
is determined by the initial data, and we postpone its construction
in Section \ref{sec compensator}. By the introduction of such
$\Psi$, then at time $t = 0$, we essentially add nothing on the
boundary, and hence the compatibility conditions to \eqref{lldsd}
are same as \eqref{geometric}. This allows us that the initial data
$(u_0,\eta_0)$ for the original problem \eqref{geometric} can be
exactly taken as the one for the regularized problem  \eqref{lldsd};
otherwise, even though $(u_0,\eta_0)$ satisfy the compatibility
conditions for \eqref{geometric}, in general they do not satisfy the
corresponding compatibility conditions for \eqref{lldsd}. One may
refer to \cite{CS3} for more discussions about the issue of
compatibility conditions. It is worth noting that we employ a
different construction of the compensator $\Psi$ from \cite{CS3}.
The way of \cite{CS3} is to define the Taylor polynomial of
$-\Delta_\ast\eta(t)$ at $t=0$ to be the compensator which requires
more regularity on the initial data; hence a very complicated
argument of regularizing the initial data that ensures the
compatibility conditions needs to be carried out. Our construction
avoids this.

By adding this artificial viscosity term, the regularity of $\eta$ obtaining from the fourth equation of \eqref{lldsd} ($\Delta_\ast \eta\sim u_3$) can control the nonlinear term $\eta \n-\sigma H\n$. This will allow us to modify the iteration scheme of \cite{GT_lwp} to establish a local well-posedness of the $\kappa$-problem \eqref{lldsd} in Section \ref{sec kk} based on the solvability of the linear $\a$--Navier--Stokes equations \eqref{intro_A_NS} and the following linear regularized surface problem
\begin{equation}
 \dt \eta -\kappa\Delta_\ast\eta=\kappa\Psi+ F^4 \  \text{on } \Sigma.
\end{equation}
Having obtained the unique strong solutions
$(u^\kappa,p^\kappa,\eta^\kappa)$ to \eqref{lldsd} on a local time
interval $[0,T_0^\kappa]$ for each $\kappa>0$ (cf. Theorem
\ref{l_knwp}), in Section \ref{sec local} we then derive the
$\kappa$-independent estimates for $(u^\kappa,p^\kappa,\eta^\kappa)$
which allow us to pass to the limit as $\kappa\rightarrow0$ in
\eqref{lldsd} to produce the unique strong solution
$(u^\sigma,p^\sigma,\eta^\sigma)$ to \eqref{geometric} on a local
time interval $[0,T_0^\sigma]$ for each fixed $\sigma>0$ (cf.
Theorem \ref{l_sigmanwp}). We next continue to derive the
$\sigma$-independent estimates for $(u^\sigma,p^\sigma,\eta^\sigma)$
which allow us to pass to the limit as $\sigma\rightarrow0$ in
\eqref{geometric} within a local time interval $[0,T_0]$. In order
to derive the $\sigma$-independent estimates (and the derivation of
the $\kappa$-independent estimates is similar), it is a very key
point to design the special quantity $\mathfrak{K}^\sigma$ (defined
by \eqref{ensigma}) which allows us to establish the following
estimate
\begin{equation} \label{cco}
    \mathfrak{K}^{\sigma}
 \le  \left(  P(\mathcal{E}_0^\sigma)  + P (1+ \mathfrak{K}^{\sigma}) T^{1/4}
  \right)  \exp\left( \left(  P(\mathcal{E}_0^\sigma)  + P(  T^{1/4} \mathfrak{K}^{\sigma})
  \right) T \right).
\end{equation}
Since $\mathcal{E}_0^\sigma(u_0^\sigma,
\eta_0^\sigma)\rightarrow\mathcal{E}_0^0(u_0, \eta_0)$ as
$\sigma\rightarrow0$, a standard continuity argument basing on
\eqref{cco} then implies that the solutions
$(u^\sigma,p^\sigma,\eta^\sigma)$ to \eqref{geometric} all exist on
a $\sigma$-independent time interval $[0,T_0]$ and satisfy
$\mathfrak{K}^{\sigma}\le 2P(\mathcal{E}_0^\sigma)$. We can then
improve this estimate to be \eqref{zero_1} by improving the
estimates of $\eta$. Note that we will refine many estimates of
\cite{GT_lwp} so that Theorem \ref{lwp} can be established for our
``semi-small" initial data which is not necessary to be genuinely
small. This completes the proof of Theorem \ref{lwp}.

We then proceed to prove Theorem \ref{gwp} of the global-in-time
results in Section \ref{sec global}, and the key step is to derive
the global-in-time $\sigma$-independent estimates for the local
solutions implicitly contained in Theorem \ref{lwp}. We will employ
the two-tier energy method recently developed by Guo and Tice
\cite{GT_per,GT_inf}. In this paper, for simplify we choose to
consider the ``periodic case" but we could treat the ``non-periodic
case" by using the arguments of \cite{GT_inf}. For the sake of
completeness, here we will briefly sketch the two-tier energy
method. Basing on the natural energy evolution and making full use
of the structure of the problem \eqref{geometric}, we may prove the
energy estimate, where $\k\ls  \mathcal{E}_{N+2}^\sigma$,
 \begin{equation}\label{pp}
\mathcal{E}_{2N}^\sigma(t)+\int_0^t \mathcal{D}_{2N}^\sigma \lesssim
 \mathcal{E}_{2N}^\sigma(0)  +\k\f
+ \int_0^t
 \sqrt{\mathcal{D}_{2N}^\sigma\mathcal{K}\mathcal{F}_{2N}}+\mathcal{K}\mathcal{F}_{2N}.
\end{equation}
The derivation of \eqref{pp} is divided into two steps. First to use the energy evolution, we can only consider the ``horizontal" energy and dissipation with localization, say, $\seb{n}^\sigma$ and $\sdb{n}$ (cf. Section \ref{sss} for the definitions), and we may prove
\begin{equation} \label{q112}
 \seb{2N}^\sigma(t) + \int_0^t\sdb{2N}
\lesssim  \se{2N}^\sigma(0) +
(\mathcal{E}_{2N}^\sigma(t))^{ {3}/{2}}
+ \int_0^t  \sqrt{\mathcal{E}_{2N}^\sigma }
\mathcal{D}_{2N}^\sigma+
\sqrt{\mathcal{D}_{2N}^\sigma\mathcal{K}\mathcal{F}_{2N}}+\varepsilon\mathcal{D}_{2N}^\sigma
 .
\end{equation}
Second we use the structure of \eqref{geometric} to show the comparison estimates that
\begin{equation}\label{q012}
 {\mathcal{E}}_{2N}^\sigma
\lesssim \bar{\mathcal{E}}_{2N}^\sigma + \mathcal{K}
\mathcal{F}_{2N}, \text{ and }  {\mathcal{D}}_{2N}^\sigma \lesssim
\bar{\mathcal{D}}_{2N} + \mathcal{K} \mathcal{F}_{2N}.
\end{equation}
It is easy to show that ${\mathcal{E}}_{2N}^\sigma$ is comparable to $\seb{2N}^\sigma$ as stated in \eqref{q012} by applying the Stokes elliptic theory of Lemma \ref{i_linear_elliptic}.  However, due to the lack of  $\eta$ terms in $\bar{\mathcal{D}}_{2N}$, we can not compare the dissipations by using this technique. To overcome this difficulty, Guo and Tice \cite{GT_per,GT_inf} instead use the structure of the equations as well as the equation for the vorticity (from which the pressure and $\eta$ can effectively be eliminated) to derive various estimates for $u$ and $\nabla p$ without reference to $\eta$. Then a bootstrapping procedure is employed to obtain estimates for $\eta$ and $p$ (not just its gradient). In this paper, we use an alternative procedure which seems much more simpler.  More precisely, since
$\bar{\mathcal{D}}_{2N}$ controls horizontal derivatives, we can use it to gain the regularity of $u$ on $\Sigma$.  The idea is then to apply the Stokes elliptic regularity theory of Lemma \ref{i_linear_elliptic2} to deduce the desired estimates of $u$.  Then we get the
desired estimates of $p$ and $\eta$  by using the equations directly. Then \eqref{pp} follows by \eqref{q112} and \eqref{q012}.

In order to close the estimates \eqref{pp}, we appeal to the control of $\f$. Note that for $\sigma>0$, we can bound by $\mathcal{F}_{2N}\le \sigma^{-2}\mathcal{D}_{2N}^\sigma$. Hence, for each fixed $\sigma>0$, from \eqref{pp} we may prove the global well-posedness and exponential decay of \eqref{geometric} for the initial data sufficiently small with respect to $\sigma$ (cf. Theorem \ref{globall}). However, to prove the zero surface tension limit, we need to derive the $\sigma$-independent estimates for the solutions. The only way to estimate $\f$ then is through the kinematic transport equation for $\eta$, and we may derive
\begin{equation} \label{pp11}
   \f(t) \ls
\f(0) +   t \int_0^t \sd{2N}^\sigma.
\end{equation}
This estimate allows $\f$ to grow linearly in time. So to balance such growth of $\f$, we must derive a strong decay of $\k\ls \se{N+2}$. This is accomplished by the following differential inequality
\begin{equation}  \label{n+223}
\frac{d}{dt} \mathcal{E}_{N+2}^\sigma +  \mathcal{D}_{N+2}^\sigma\le
0.
\end{equation}
Indeed, we interpolate to have ${\mathcal{E}}_{N+2}^\sigma\lesssim({\mathcal{D}}_{N+2}^\sigma)^{(4N-8)/(4N-7)}({\mathcal{E}}_{2N}^\sigma)^{1/(4N-7)}$, so we can deduce from \eqref{n+223} that
\begin{equation}\label{pp22}
(1+t)^{4N-8} \mathcal{E}_{N+2}^\sigma(t)\lesssim
\mathcal{E}_{2N}^\sigma(0)+ \mathcal{F}_{2N}(0).
\end{equation}
Then \eqref{pp}, \eqref{pp11} and \eqref{pp22} imply the desired bound $\g^\sigma(t)\ls \g^\sigma(0)$. This $\sigma$-independent estimates will then allow us to pass to the limit as $\sigma\rightarrow0$ in \eqref{geometric} on the whole interval $[0,\infty)$. Hence we can complete the proof of Theorem \ref{gwp}.

\subsection{Notation convention}\label{nota}
We shall also employ the notation convention from \cite{GT_lwp,GT_per,GT_inf}.

We write $H^k(\Omega)$ with $k\ge 0$ and $H^s(\Sigma)$ with $s \in \Rn{}$ for the usual Sobolev spaces. We also typically write $H^0 = L^2$.  To avoid notational clutter, we will avoid writing $H^k(\Omega)$ or $H^k(\Sigma)$ in our norms and typically write only $\norm{\cdot}_{k}$.  Since we will do this for functions defined on both $\Omega$ and $\Sigma$, this presents some ambiguity.  We avoid this by adopting two conventions.  First, we assume that functions have natural spaces on which they ``live.''  For example, the functions $u,$ $p$, and $\bar{\eta}$ live on $\Omega$, while $\eta$ itself lives on $\Sigma$.  As we proceed in our analysis, we will introduce various auxiliary functions; the spaces they live on will always be clear from the context.  Second, whenever the norm of a function is computed on a space different from the one in which it lives, we will explicitly write the space.  This typically arises when computing norms of traces onto $\Sigma$ of functions that live on $\Omega$. We use $\norm{\cdot}_{L^pX}$ to denote the norm of the space $L^p(0,T;X)$.

We write $\mathbb{N} = \{ 0,1,2,\dotsc\}$ for the collection of non-negative integers.  When using space-time differential multi-indices, we will write $\mathbb{N}^{1+m} = \{ \alpha = (\alpha_0,\alpha_1,\dotsc,\alpha_m) \}$ to emphasize that the $0-$index term is related to temporal derivatives.  For just spatial derivatives we write $\mathbb{N}^m$.  For $\alpha \in \mathbb{N}^{1+m}$ we write $\pa = \dt^{\alpha_0} \p_1^{\alpha_1}\cdots \p_m^{\alpha_m}.$ We define the parabolic counting of such multi-indices by writing $\abs{\alpha} = 2 \alpha_0 + \alpha_1 + \cdots + \alpha_m.$  We will write $\nab_\ast f$ for the horizontal gradient of $f$, i.e. $\nab_\ast f = \p_1 f e_1 + \p_2 f e_2$, while $\nab f$ will denote the usual full gradient.

For a given norm $\norm{\cdot}$ and  integers $k\ge m\ge 0$, we
introduce the following notation for sums of spatial derivatives:
\begin{equation}
 \norm{{\nab_{\ast}}_m^k f}^2 := \sum_{\substack{\alpha \in \mathbb{N}^2 \\ m\le \abs{ \alpha}\le k} } \norm{\pa  f}^2 \text{ and }
\norm{\nab_m^k f}^2 := \sum_{\substack{\alpha \in \mathbb{N}^{3} \\  m\le \abs{\alpha}\le k} } \norm{\pa  f}^2.
\end{equation}
For space-time derivatives we add bars to our notation:
\begin{equation}
 \norm{\bar{\nab}_{\ast m}^{\ \, k}  f}^2 := \sum_{\substack{\alpha \in \mathbb{N}^{1+2} \\  m\le \abs{\alpha}\le k} } \norm{\pa  f}^2 \text{ and }
\norm{\bar{\nab}_m^k f}^2 := \sum_{\substack{\alpha \in \mathbb{N}^{1+3} \\ m\le  \abs{\alpha}\le k} } \norm{\pa  f}^2.
\end{equation}
We allow for composition of derivatives in this counting scheme in a natural way; for example,
\begin{equation}
 \norm{\nab_\ast\bar{\nab}_{\ast m}^{\ \, k} f}^2 =  \norm{\bar{\nab}_{\ast m}^{\ \, k} \nab_\ast f}^2 = \norm{\bar{\nab}_{\ast m+1}^{\ \, k+1}  f}^2.
\end{equation}

Throughout the paper we assume that $N\ge 3$ is an integer. We will
employ the Einstein convention of summing over repeated indices.
Throughout the paper $C>0$ will denote a generic constant that does
not depend on the data, the surface tension coefficient $\sigma$ and
the artificial viscosity coefficient $\kappa$, but can depend on the
other parameters of the problem, $N\ge 3$ and $\Omega$.  We refer to
such constants as ``universal''. We also use $P$ to denote for a
universal positive polynomial  with $ P(0)=0$. Such constants and
polynomials are allowed to change from line to line. When they
depend on a quantity $z$ we will write $C_z$ and $P_z$ to indicate
this, or we will explicitly point out along the context.  We will
employ the notation $a \lesssim b$ to mean that $a \le C b$ for a
universal constant $C>0$.

\section{Preliminaries}\label{sec pre}

\subsection{Analytic tools}

\subsubsection{Poisson integral }
For $\Sigma = \Rn{2}$, we define the Poisson integral in $\Rn{2} \times (-\infty,0)$ by
\begin{equation}\label{poisson_def_inf}
 \mathcal{P}f(x',x_3) = \int_{\Rn{2}} \hat{f}(\xi) e^{2\pi \abs{\xi}x_3} e^{2\pi i x' \cdot \xi} d\xi,
\end{equation}
where we have written $\hat{f}(\xi)$ the Fourier transform of $f$ in $\Rn{2}$.
For $\Sigma =(L_1 \mathbb{T}) \times ( L_2 \mathbb{T})$, we define the Poisson integral in $(L_1 \mathbb{T}) \times ( L_2 \mathbb{T})\times (-\infty,0)$ by
\begin{equation}\label{poisson_def_per}
\mathcal{P} f(x) = \sum_{n \in   (L_1^{-1} \mathbb{Z}) \times (L_2^{-1} \mathbb{Z}) }  e^{2\pi i n \cdot x'} e^{2\pi \abs{n}x_3} \hat{f}(n),
\end{equation}
where for $n \in   (L_1^{-1} \mathbb{Z}) \times (L_2^{-1} \mathbb{Z})$ we have written
\begin{equation}
 \hat{f}(n) = \int_\Sigma f(x')  \frac{e^{-2\pi i n \cdot x'}}{L_1 L_2} dx'.
\end{equation}
It is well known that $\mathcal{P}: H^{s}(\Sigma) \rightarrow H^{s+1/2}(\Omega_-)$ for $\Omega_-=\Sigma\times(-\infty,0)$ is a bounded linear operator for $s>0$. However, if restricted to the domain $\Omega$, we have the following improvements.

\begin{lem}\label{p_poisson}
Let $\mathcal{P} f$ be the Poisson integral of a function $f$ that is either in $\dot{H}^{q}(\Sigma)$ or $\dot{H}^{q-1/2}(\Sigma)$ for $q \in \mathbb{N}$, where $\dot{H}^s$ is the usual homogeneous Sobolev space of order $s$.  Then
\begin{equation}
 \ns{\nab^q \mathcal{P}f }_{0} \ls \norm{f}_{\dot{H}^{q-1/2}(\Sigma)}^2 \text{ and }  \ns{\nab^q \mathcal{P}f }_{0} \ls \norm{f}_{\dot{H}^{q}(\Sigma)}^2.
\end{equation}
\end{lem}
\begin{proof}
For the proof, we refer to Lemma A.7 of \cite{GT_lwp} for the case $\Sigma=\Rn{2}$ and Lemma A.9 of \cite{GT_lwp} for the case $\Sigma=(L_1 \mathbb{T}) \times ( L_2 \mathbb{T})$.
\end{proof}

\subsubsection{Some inequalities}
We need some estimates of the product of functions in Sobolev spaces.

\begin{lem}\label{i_sobolev_product_1}
Let the domain of function spaces be either $\Sigma$ or $\Omega$.
 Let $0\le r \le s_1 \le s_2$ be such that  $s_2 >r+ n/2$, then the following hold  when the right hand side is finite:
\begin{equation}\label{i_s_p_02}
 \norm{fg}_{ H^r} \lesssim \norm{f}_{ H^{s_1}} \norm{g}_{ H^{s_2}}
\end{equation}
and
\begin{equation}\label{i_s_p_03}
 \norm{fg}_{H^{-s_1}} \ls \norm{f}_{H^{-r}} \norm{g}_{H^{s_2}}.
\end{equation}
\end{lem}
\begin{proof}
See Lemma A.1 of \cite{GT_lwp}.
\end{proof}
\begin{remark}
We will use a slight variant from \eqref{i_s_p_03} when the domain is $\Omega$: for $s>5/2$,
\begin{equation}\label{i_s_p_04}
 \norm{fg}_{(\H1)^\ast} \ls \norm{f}_{(\H1)^\ast} \norm{g}_{H^s(\Omega)},
\end{equation}
where the space $\H1$ is defined in \eqref{function_spaces}.
\end{remark}

We then record some Poincar\'{e}-type inequalities for the domain $\Omega$.
\begin{lem}\label{poincare_b}
Let $f\in H^1(\Omega)$, then
\begin{equation}\label{poincare_b1}
 \norm{f}_{L^2(\Omega)}  \lesssim \norm{f}_{L^2(\Sigma)}  +  \norm{\partial_3f}_{L^2(\Omega)} .
 \end{equation}
Let $f\in H^1(\Omega)$ be so that $f=0$ on $\Sigma_b$, then
\begin{equation}\label{poincare_b2}
\norm{f}_{L^2(\Omega)} \lesssim \norm{f}_{H^1(\Omega)}\lesssim \norm{\nabla f}_{L^2(\Omega)}.
 \end{equation}

\end{lem}
\begin{proof}
See Lemmas A.12 and A.14 of \cite{GT_lwp}.
\end{proof}

We will need the following version of Korn's inequality.
\begin{lem}\label{i_korn}
It holds that $\norm{u}_{1} \ls \norm{\sg u}_{0}$ for all $u \in H^1(\Omega;\Rn{3})$ so that $u=0$ on $\Sigma_b$.
\end{lem}
\begin{proof}
See Lemma 2.3 of \cite{beale_2}.
\end{proof}

\subsubsection{Continuity and temporal derivatives}

In our local well-posedness arguments it is crucial to use the following interpolation result, which affords us control of the $L^\infty H^k$ norm of a function $f$, given that  we control  $f$ in  $L^2 H^{k+m}$ and $\dt f$ in $L^2 H^{k-m}$.

\begin{lem}\label{l_sobolev_infinity}
Let the domain of function spaces be either $\Sigma$ or $\Omega$. Suppose that $f \in L^2([0,T]; H^{s_1} )$ and $\dt f \in L^2([0,T]; H^{s_2} )$ for $s_1 \ge s_2 $.  Let $s = (s_1+s_2)/2$.  Then  $f \in C^0([0,T]; H^{s} )$ (after possibly being redefined on a set of measure $0$), and
\begin{equation}\label{l_sobi_01}
 \ns{f}_{L^\infty H^{s}} \le  \ns{f(0)}_{ H^{s}}+  \ns{ f}_{L^2 H^{s_1}} + \ns{\dt f}_{L^2 H^{s_2}} .
\end{equation}
\end{lem}
\begin{proof}
See the proof of Lemma A.4 of \cite{GT_lwp}.
\end{proof}

\subsubsection{Extension results}

In our Local well-posedness arguments we need to be able to take the
initial data $\dt^j u(0)$, $j=0,\dotsc,2N$ and extend it to a
function $u^0$ satisfying certain estimates.

\begin{lem}\label{l_sobolev_extension}
Suppose that $\dt^j u(0) \in H^{4N-2j}(\Omega)$ for $j=0,\dotsc,2N$.  Then there exists an extension $u^0$, achieving the initial data,  so that
\begin{equation}
\dt^j u^0\in L^2([0,\infty);H^{4N-2j+1}(\Omega)) \cap L^\infty([0,\infty); H^{4N-2j}(\Omega))
\end{equation}
for $j=0,\dotsc,2N$.  Moreover,
\begin{equation}\label{kk11}
 \mathfrak{K}_{2N}(u^0):=\sum_{j=0}^{2N} \ns{\dt^j u^0}_{L^2 H^{4N-2j +1}}
+ \ns{\dt^j u^0}_{L^\infty H^{4N-2j }}\ls \sum_{j=0}^{2N}   \ns{\dt^j u(0)}_{  H^{4N-2j }}.
\end{equation}
\end{lem}
\begin{proof}
See the proof of Lemma A.5 of \cite{GT_lwp}.
\end{proof}

A similar extension result holds for $\dt^j \eta(0)$, $j=0,\dotsc,2N$.
\begin{lem}\label{2_sobolev_extension}
Suppose that $\dt^j \eta(0) \in H^{4N-2j+1}(\Sigma)$ for $j=0,\dotsc,2N$.  Then there exists an extension $\eta^0$, achieving the initial data,  so that
\begin{equation}
\dt^j \eta^0\in L^2([0,\infty);H^{4N-2j+2}(\Sigma)) \cap L^\infty([0,\infty); H^{4N-2j+1}(\Sigma))
\end{equation}
for $j=0,\dotsc,2N$. Moreover,
\begin{equation}\label{kk12}
  \mathfrak{K}_{2N}(\eta^0):=\sum_{j=0}^{2N} \ns{\dt^j \eta^0}_{L^2 H^{4N-2j +2}}
+ \ns{\dt^j \eta^0}_{L^\infty H^{4N-2j+1}}\ls \sum_{j=0}^{2N}   \ns{\dt^j \eta(0)}_{  H^{4N-2j+1}}.
\end{equation}
\end{lem}
\begin{proof}
The proof is same as Lemma \ref{l_sobolev_extension}.
\end{proof}

\subsection{Time-dependent functional setting}

\subsubsection{Time-dependent function spaces}
We begin with introducing some function spaces. Define
\begin{equation}\label{function_spaces}
\begin{split}
 \H1 &:= \{ u \in H^1(\Omega) \;\vert\;  u\vert_{\Sigma_b}=0\}, \\
{^0}H^1(\Omega) &:= \{u\in H^1(\Omega) \;\vert\; u\vert_{\Sigma}=0\}, \text{ and } \\
\Hsig &:= \{ u \in \H1 \;\vert\; \diverge{u}=0 \},
\end{split}
\end{equation}
with the obvious restriction that the last space is for vector-valued functions only.

For our time-dependent function spaces we will consider $\eta$ as given with $\a$, $J$, etc determined by $\eta$ via \eqref{ABJ_def}.  We define a time-dependent inner-product on  $L^2=H^0$ by introducing
\begin{equation}
 \iph{u}{v}{0} := \int_\Omega  ( u \cdot v)  J(t)
\end{equation}
with corresponding norm $\hn{u}{0} := \sqrt{\iph{u}{u}{0}}$.  Then we write
$\h^0(t) := \{ \hn{u}{0} < \infty \}$.  Similarly, we  define a time-dependent inner-product on $\H1$ according to
\begin{equation}
 \iph{u}{v}{1} :=  \int_\Omega \left(\sg_{\a(t)} u : \sg_{\a(t)} v \right)  J(t),
\end{equation}
and we define the corresponding norm by $\hn{u}{1} = \sqrt{\iph{u}{u}{1}}$.  Then we define
\begin{equation}\label{H1_def}
 \h^1(t) := \{ u \;\vert\;  \hn{u}{1} < \infty ,  u\vert_{\Sigma_b}=0\} \text{ and }\x(t) := \{ u \in \h^1(t) \;\vert\; \diverge_{\a(t)}{u}=0\}.
\end{equation}
We will also need the orthogonal decomposition $\h^0(t) = \y(t) \oplus \y(t)^\bot,$ where
\begin{equation}\label{l_Y_space_def}
 \y(t)^\bot := \{ \nab_{\a(t)} \varphi \;\vert\; \varphi \in {^0}H^1(\Omega)  \}.
\end{equation}
In our use of these norms and spaces, we will often drop the $(t)$ when there is no potential for confusion.
Finally, for $T>0$ and $k=0,1$, we define inner-products on $L^2([0,T];H^k(\Omega))$ by
\begin{equation}\label{HkT_def}
 \ip{u}{v}_{\h^k_T} := \int_0^T \iph{u(t)}{v(t)}{k} dt.
\end{equation}
Write $\norm{u}_{\h^k_T}$ for the corresponding norms and $\h^k_T$ for the corresponding spaces.  We define the subspace of $\diva$-free vector fields as
\begin{equation}\label{X_def}
 \x_T := \{ u \in \h^1_T \;\vert\; \diverge_{\a(t)}{u(t)} =0 \text{ for a.e. } t\in[0,T]\}.
\end{equation}

Under a smallness assumption on $\eta$, we can have that the spaces $\h^k(t)$ and $\h^k_T$ have the same topology as $H^k$ and $L^2 H^k$, respectively.

\begin{lem}\label{l_norm_equivalence}
There exists a universal $\ep_0 > 0$ so that if
$
\sup_{0\le t \le T} \norm{\eta(t)}_{3} < \ep_0,$
then
\begin{equation}\label{l_norm_e_01}
\frac{1}{\sqrt{2}} \norm{u}_{k} \le \hn{u}{k} \le \sqrt{2} \norm{u}_{k}
\end{equation}
for $k=0,1$ and for all $t \in [0,T]$.  As a consequence, for $k=0,1$,
\begin{equation}\label{l_norm_e_02}
\frac{1}{\sqrt{2}} \norm{u}_{L^2 H^k} \le \norm{u}_{\h^k_T} \le \sqrt{2} \norm{u}_{L^2 H^k}.
\end{equation}
\end{lem}
\begin{proof}
See Lemma 2.1 of \cite{GT_lwp}.
\end{proof}

\begin{remark}  Throughout the rest of this paper, we will assume the assumption of Lemma \ref{l_norm_equivalence} is satisfied so that \eqref{l_norm_e_01}--\eqref{l_norm_e_02} hold.
\end{remark}

The following lemma provides the $H^{-1/2}$ boundary estimates of $u$ knowing $u,\diva u \in \h^0(t)$.

\begin{lem}\label{l_boundary_dual_estimate}
If $u \in \h^0(t)$ and $\diva u \in \h^0(t)$, then
\begin{equation}
 \snormspace{u \cdot \n}{-1/2}{\Sigma} +  \snormspace{u \cdot \nu}{-1/2}{\Sigma_b} \ls \hn{u}{0} + \hn{\diva{u}}{0}.
\end{equation}
\end{lem}
\begin{proof}
See Lemma 3.3 of of \cite{GT_lwp}.
\end{proof}

The following lemma records the differentiability of norms in the time-dependent spaces.

\begin{lem}\label{l_x_time_diff}
Suppose that $u \in \h^1_T$, $\dt u \in (\h^1_T)^*$.  Then the mapping $t \mapsto \norm{u(t)}_{\h^0(t)}^2$ is absolutely continuous, and
\begin{equation}\label{l_x_t_d_01}
 \frac{d}{dt}  \norm{u(t)}_{\h^0}^2 = 2 \br{\dt u(t),u(t)}_{(\h^1)^*} + \int_\Omega \abs{u(t)}^2 \dt J(t)
\end{equation}
for a.e. $t \in [0,T]$.  Moreover, $u \in C^0([0,T];H^0(\Omega))$.  If $v \in \h^1_T$, $\dt v \in (\h^1_T)^*$ as well, then
\begin{equation}\label{l_x_t_d_02}
 \frac{d}{dt} \iph{u(t)}{v(t)}{0} =  \br{\dt u(t),v(t)}_{(\h^1)^*} +  \br{\dt v(t),u(t)}_{(\h^1)^*} + \int_\Omega u(t) \cdot v(t) \dt J(t).
\end{equation}
A similar result holds for $u \in \x_T$ with $\dt u \in (\x_T)^*$.
\end{lem}
\begin{proof}
See Lemma 2.4 of \cite{GT_lwp}.
\end{proof}

It is more important to show that the space $\Hsig$ is related to the space $\x(t)$.  To this end, we define the matrix
\begin{equation}\label{l_M_def}
 M := M(t) = K \nab \Phi =
\begin{pmatrix}
K & 0 & 0 \\
0 & K & 0 \\
AK & BK & 1
\end{pmatrix}.
\end{equation}
Note that $M$ is invertible, and $M^{-1} = J \a^T$. The matrix $M(t)$ induces a linear operator $\mathcal{M}_t: u \mapsto \mathcal{M}_t(u) = M(t) u$, the properties of which are recorded in the following.

\begin{prop}\label{l_M_iso}
For each $t \in[0,T]$, $\mathcal{M}_t$ is a bounded, linear isomorphism:
from $H^k(\Omega)$ to $H^k(\Omega)$ for $k=0,1,2$;  from $L^2(\Omega)$ to $\h^0(t)$;  from $\H1$ to $\h^1(t)$; and  from $\Hsig$ to $\x(t)$.  In each case the norms of the operators $\mathcal{M}_t, \mathcal{M}_t^{-1}$ are bounded by a constant times $1 + \norm{\eta(t)}_{9/2}$.

Moreover,  the mapping $\mathcal{M}$ given by $\mathcal{M}u(t) := \mathcal{M}_t u(t)$ is a bounded, linear isomorphism: from $L^2([0,T];H^k(\Omega))$ to $L^2([0,T];H^k(\Omega))$ for $k=0,1,2$; from $L^2([0,T];H^0(\Omega))$ to $\h^0_T$;  from $L^2([0,T];\H1)$ to $\h^1_T$; and from $L^2([0,T];\Hsig)$ to $\x_T$.  In each case, the norms of the operators $\mathcal{M}$ and  $\mathcal{M}^{-1}$ are bounded by a constant times the sum $1+\sup_{0\le t \le T} \norm{\eta(t)}_{9/2}$.
\end{prop}
\begin{proof}
 We note that to show that $\mathcal{M}_t$ map $\diverge$-free vector fields to $\diverge_\a$-free vector fields, the following will be used: since $J\neq 0$,
\begin{equation}\label{l_div_preserve}
 p = \diva v \Leftrightarrow
 Jp = J \diva v   = \p_j ( J \a_{ij} v_i ) = \p_j (J \a^T v)_j = \p_j( M^{-1} v)_j = \diverge( M^{-1} v ).
\end{equation}
See Proposition 2.5 of \cite{GT_lwp} for the proof of this proposition.
\end{proof}

\subsubsection{Pressure as a Lagrange multiplier}
The following proposition allows us to introduce the pressure as a Lagrange multiplier.
\begin{prop}\label{l_pressure_decomp}
If $\Lambda_t \in (\h^1(t))^*$ is such that $\Lambda_t(v) = 0$ for all $v \in \x(t)$, then there exists a unique $p(t) \in \h^0(t)$ so that
\begin{equation}
 \iph{p(t)}{\diva v}{0} = \Lambda_t(v) \text{ for all } v\in \h^1(t)
\end{equation}
and $\hn{p(t)}{0} \ls (1+\norm{\eta(t)}_{9/2}) \norm{\Lambda_t}_{(\h^1(t))^*}$.

If $\Lambda \in (\h^1_T)^*$ is such that $\Lambda(v) = 0 $ for all $v \in \x_T$, then there exists a unique $p \in \h^0_T$ so that
\begin{equation}
 \ip{p}{\diva v}_{\h^0_T} = \Lambda(v) \text{ for all } v\in \h^1_T
\end{equation}
and $\norm{p}_{\h^0_T} \ls (1+\sup_{0\le t\le T} \norm{\eta(t)}_{9/2}) \norm{\Lambda}_{(\h^1_T)^*}$.
\end{prop}
\begin{proof}
See Proposition 2.9 of \cite{GT_lwp}.
\end{proof}

\subsection{Elliptic estimates}

\subsubsection{Two classical Stokes problems}

We recall the elliptic estimates for two classical Stokes problems with different boundary conditions.
\begin{lem}\label{i_linear_elliptic}
Let $r \ge 2$.  Suppose that $f\in H^{r-2}(\Omega), h\in H^{r-1}(\Omega)$ and $\psi\in H^{r-3/2}(\Sigma)$. Then there exists unique $u\in H^r(\Omega), p\in H^{r-1}(\Omega)$  solving the problem
\begin{equation}\label{stokes1}
 \begin{cases}
  -\Delta u + \nab p =f  &\text{in }\Omega \\
  \diverge{u} = h &\text{in }\Omega \\
  (pI - \sg (u) )e_3 = \psi &\text{on }\Sigma \\
  u =0 &\text{on }\Sigma_b.
 \end{cases}
\end{equation}
 Moreover,
\begin{equation}
 \snorm{u}{r}^2 + \snorm{p}{r-1}^2 \ls \snorm{f}{r-2}^2 + \snorm{h}{r-1}^2 + \snorm{\psi}{r-3/2}^2.
\end{equation}
\end{lem}
\begin{proof}
 See \cite{beale_1}.
\end{proof}

\begin{lem}\label{i_linear_elliptic2}
Let $r \ge 2$.  Suppose that $f\in H^{r-2}(\Omega), h\in H^{r-1}(\Omega), \varphi\in H^{r-1/2}(\Sigma)$ and that $h,\varphi$ satisfy the compatibility condition
\begin{equation}
\int_\Omega h =\int_{\Sigma} \varphi \cdot e_3.
\end{equation}
Then there exists unique $u\in H^r(\Omega), p\in H^{r-1}(\Omega)$ (up to constants) solving the problem
\begin{equation}\label{stokes2}
 \begin{cases}
  -\Delta u + \nab p =f  &\text{in }\Omega \\
  \diverge{u} = h &\text{in }\Omega \\
u= \varphi &\text{on }\Sigma \\
  u =0 &\text{on }\Sigma_b.
 \end{cases}
\end{equation}
 Moreover,
\begin{equation}
 \snorm{u}{r}^2 + \snorm{\nab p}{r-2}^2 \ls \snorm{f}{r-2}^2 + \snorm{h}{r-1}^2 + \snorm{\varphi}{r-1/2}^2.
\end{equation}
\end{lem}
\begin{proof}
 See \cite{L}.
\end{proof}

\subsubsection{The $\a-$Stokes problems}

In our local well-posedness argument, we will need the regularity of the solutions to the following $\a-$Stokes problem
\begin{equation}\label{l_linear_elliptic}
 \begin{cases}
 \diva S_\a(p,u) = F^1 & \text{in }\Omega \\
\diva{u}=F^2 & \text{in }\Omega\\
\Sa(p,u) \n = F^3 & \text{on }\Sigma \\
u =0 & \text{on }\Sigma_b.
 \end{cases}
\end{equation}
We shall use the regularity of \eqref{l_linear_elliptic} for each fixed $t$ in the context of the time-dependent problem, and hence we will temporarily ignore the time dependence of $\eta,$ $\mathcal{A},$ $\mathcal{N}$, etc, to view \eqref{l_linear_elliptic} as a stationary problem.
\begin{prop}\label{l_stokes_regularity}
Let $k\ge 4$ be an integer and suppose that $\eta \in H^{k+1/2}(\Sigma)$.  There exists a universal $\ep_0 >0$ so that if $\ns{\eta}_{k-1/2} \le \ep_0$, then solutions to \eqref{l_linear_elliptic} satisfy
\begin{equation}\label{l_s_reg_0}
  \norm{u}_{r} + \norm{p}_{r-1} \ls \norm{F^1}_{r-2} + \norm{F^2}_{r-1} + \norm{F^3}_{r-3/2}
\end{equation}
for $r=2,\dotsc,k$, whenever the right side is finite.

In the case $r=k+1$,  solutions to \eqref{l_linear_elliptic} satisfy
\begin{equation}\label{l_s_reg_01}
\begin{split}
\norm{u}_{k+1} + \norm{p}_{k} \ls & \norm{F^1}_{k-1} + \norm{F^2 }_{k} + \norm{F^3}_{k-1/2}   \\&
+  \norm{\eta}_{k+1/2} \left( \norm{F^1}_{2} + \norm{F^2 }_{3} + \norm{F^3}_{5/2}  \right).
\end{split}
\end{equation}

\end{prop}
\begin{proof}
Under the smallness assumption of $\eta$, we can solve \eqref{l_linear_elliptic} by viewing it as a perturbation of the classical Stokes problem  \eqref{stokes1}.  See Proposition 3.7 of \cite{GT_lwp} for the detailed proof.
\end{proof}

\subsubsection{The $\a-$Poisson problem}

We now consider the scalar elliptic problem
\begin{equation}\label{l_linear_elliptic_p}
 \begin{cases}
  \da p = f^1 & \text{in } \Omega \\
  p = f^2 & \text{on }\Sigma \\
   \naba p \cdot \nu = f^3 & \text{on } \Sigma_b,
 \end{cases}
\end{equation}
where $\nu$ is the outward-pointing normal on $\Sigma_b$.

We first consider the weak formulation of \eqref{l_linear_elliptic_p}.  For this, we define a scalar $\h^1$ in a natural way through the norm
\begin{equation}
 \hn{f}{1}^2 = \int_\Omega J \abs{\naba f}^2.
\end{equation}
As in Lemma \ref{l_norm_equivalence}, this scalar norm generates the same topology as the usual scalar $H^1$ norm.

For the weak formulation we suppose $f^1 \in ({^0}H^1(\Omega))^*$, $f^2 \in H^{1/2}(\Sigma)$, and $f^3 \in H^{-1/2}(\Sigma_b)$.  Let $\bar{p}\in H^1(\Omega)$ be an extension of $f^2$ so that $ \bar{p}=0$ near $\Sigma_b$.  We switch the unknown  to $q = p - \bar{p}$ and then define a weak formulation of \eqref{l_linear_elliptic_p} by finding a $q \in {^0}H^1(\Omega)$ so that
\begin{equation}\label{l_a_poisson_weak}
\iph{q}{\varphi}{1} = -\iph{\bar{p}}{\varphi}{1} -\br{f^1,\varphi}_{*} + \br{f^3,\varphi}_{-1/2} \text{ for all } \varphi \in {^0}H^1(\Omega),
\end{equation}
where  $\br{\cdot,\cdot}_{*}$ is the dual pairing with ${^0}H^1(\Omega)$ and $\br{\cdot,\cdot}_{-1/2}$ is the dual pairing with $H^{1/2}(\Sigma_b)$.  The existence and uniqueness of a solution to \eqref{l_a_poisson_weak} follow from standard arguments, and the resulting $p =q + \bar{p} \in H^1(\Omega)$ satisfies
\begin{equation}\label{l_a_poisson_weak_estimate}
 \hn{p}{1}^2 \ls  \norm{f^1}_{({^0}H^1(\Omega))^*}^2 + \snormspace{f^2}{1/2}{\Sigma}^2 +  \snormspace{f^3}{-1/2}{\Sigma_b}^2  .
\end{equation}

In the event of determining the initial pressure the action of $f^1\in
({}^0H^1(\Omega))^\ast$ will be given in a more specific fashion by
\begin{equation}
 \br{f^1,\varphi}_* = \iph{g_0}{\varphi}{0} + \iph{G}{\naba \varphi}{0}
\end{equation}
for $ g_0, G \in H^0(\Omega )$ with $\norm{g_0}_{0}^2 + \norm{G}_{0}^2 = \norm{f^1}_{({^0}H^1(\Omega))^*}^2$.
Then \eqref{l_a_poisson_weak} may be rewritten as
\begin{equation}\label{l_a_poisson_weak_2}
\iph{\naba p+G}{\naba \varphi}{0} = - \iph{g_0}{\varphi}{0} + \br{f^3,\varphi}_{-1/2} \text{ for all } \varphi \in {^0}H^1(\Omega).
\end{equation}
This serves as motivation for us to say that $p$ is a weak solution to the PDE
\begin{equation}\label{l_a_poisson_div}
 \begin{cases}
  \diva(\naba p +G) = g_0 \in H^0(\Omega)  \\
  p = f^2 \in H^{1/2}(\Sigma)  \\
   (\naba p +G)\cdot \nu = f^3 \in H^{-1/2}(\Sigma_b).
 \end{cases}
\end{equation}

Now we record the analogue of Proposition \ref{l_stokes_regularity} for the problem \eqref{l_linear_elliptic_p}.
\begin{prop}\label{l_a_poisson_regularity}
Let $k\ge 4$ be an integer and suppose that $\eta \in H^{k+1/2}(\Sigma)$.  There exists a universal $\ep_0 >0$ so that if $\ns{\eta}_{k-1/2} \le \ep_0$, then solutions to \eqref{l_linear_elliptic_p} satisfy
\begin{equation}\label{l_ap_reg_0}
   \norm{p}_{r} \ls \norm{f^1}_{r-2} + \norm{f^2}_{r-1/2} + \norm{f^3}_{r-3/2}
\end{equation}
for $r=2,\dotsc,k$, whenever the right side is finite.
\end{prop}
\begin{proof}
 The proof is similar as Proposition \ref{l_stokes_regularity}.
\end{proof}

\subsection{Transport estimate}
Consider the equation
\begin{equation}\label{i_transport_eqn}
\begin{cases}
 \dt \eta + u \cdot \nab_\ast \eta = g & \text{in } \Sigma  \\
 \eta\mid_{ t=0}  = \eta_0.
\end{cases}
\end{equation}
 We have the following estimate of the transport of regularity for solutions to \eqref{i_transport_eqn}.

\begin{lem}\label{i_sobolev_transport}
 Let $\eta$ be a solution to \eqref{i_transport_eqn}.  Then there is a universal constant $C>0$ so that for any $0 \le s <2$
\begin{equation}
 \sup_{0\le r \le t} \snorm{\eta(r)}{s} \le \exp\left(C \int_0^t \snorm{\nab_\ast u(r)}{3/2} dr \right) \left( \snorm{\eta_0}{s} + \int_0^t \snorm{g(r)}{s}dr  \right).
\end{equation}
\end{lem}
\begin{proof}
Use $p=p_2 = 2$, $N=2$, and $\sigma=s$ in Proposition 2.1 of \cite{danchin} along with the embedding    $H^{3/2}  \hookrightarrow B^1_{2,\infty}\cap L^\infty  .$
\end{proof}

\section{Local well-posedness of the $\kappa$-problem}\label{sec kk}

The aim of this section is to solve the following regularized $\kappa$-problem, where $\kappa>0$ is the artificial viscosity coefficient,
\begin{equation}\label{kgeometric}
 \begin{cases}
  \dt u - \dt \bar{\eta} \tilde{b} K \p_3 u + u \cdot \naba u -\da u + \naba p     =0 & \text{in } \Omega \\
 \diva u = 0 & \text{in }\Omega \\
 S_\a(p,u) \n =  \eta\n -\sigma H\n & \text{on } \Sigma \\
 \dt \eta -\kappa\Delta_\ast\eta=\kappa\Psi+ u \cdot \n & \text{on } \Sigma \\
 u = 0 & \text{on } \Sigma_b \\
 (u , \eta)\mid_{t=0} = (u_0,\eta_0).
 \end{cases}
\end{equation}
Here $\Psi$ is the compensator determined by the initial data, and
we postpone its construction in Section \ref{sec compensator}. We
will temporally ignore the value of $\sigma$ since it plays no role
in this section.

\subsection{The linear $\mathcal{A}$--Navier-Stokes problem}

In this subsection we will consider the following linear problem for $(u,p)$, where $\eta$ (and hence $\a, \n$, etc), $F^1$ and $F^3$ are given,
\begin{equation}\label{l_linear_forced}
 \begin{cases}
\dt u - \da u + \naba p = F^1 & \text{in }\Omega \\
\diva{u}=0 & \text{in }\Omega\\
\Sa(p,u) \n = F^3 & \text{on }\Sigma\\
u =0 & \text{on }\Sigma_b\\
u\mid_{t=0}=u_0.
 \end{cases}
\end{equation}
Since we are considering the same problem as Section 4 of \cite{GT_lwp}, most of the arguments will be almost same as there. However, we will relax the assumptions on $\eta$; as we will see, such relaxation will be crucial for our local-wellposedness of the $\kappa$-problem \eqref{kgeometric} (and the original problem \eqref{geometric}) with the ``semi-small" initial data.
For the sake of completeness, we will repeat main arguments from \cite{GT_lwp}; along the context we will point out the differences and the necessary modifications due to our weaker assumption on $\eta$.

Although our aim is to construct strong solutions to \eqref{l_linear_forced} with high regularity, we need the definition of a weak solution to \eqref{l_linear_forced}. Assuming the existence of a smooth solution to \eqref{l_linear_forced} and taking the $L^2L^2$ inner product with $J v$ for  $v \in \h^1_T$, we obtain
\begin{equation}\label{l_weak_motivation}
 \ip{\dt u}{v}_{\h^0_T} + \hal \ip{u}{v}_{\h^1_T} - \ip{p}{\diva v}_{\h^0_T}= \ip{F^1}{v}_{\h^0_T} - \ip{F^3}{v}_{0,\Sigma,T},
\end{equation}
where $ \ip{\cdot}{\cdot}_{0,\Sigma,T}$ is the $L^2L^2$ inner product in $\Sigma\times(0,T)$. If we were to restrict test functions to $v \in \x_T$ (defined by \eqref{X_def}), then the term $\ip{p}{\diva v}_{\h^0_T}$ would vanish above, and we would be left with a ``pressureless'' formulation of the problem involving only the velocity field.  This leads us to define a weak formulation without the pressure as follows. Suppose that
\begin{equation}\label{l_weak_data_assumptions}
 \bar{F} \in (\x_T)^* \text{ and } u_0 \in \y(0),
\end{equation}
where $\y(0)$ is defined by \eqref{l_Y_space_def}. Then our definition of a weak solution requires that $u$ satisfies
\begin{equation}\label{l_weak_solution_pressureless}
 \begin{cases}
  u \in \x_T, \dt u \in (\x_T)^*,   \\
  \br{\dt u,\psi}_{*}  + \hal \ip{u}{\psi}_{\h^1_T}  = \br{\bar{F},\psi}_{*}   \text{ for every }\psi \in \x_T, \\
u(0) = u_0,
 \end{cases}
\end{equation}
where $\br{\cdot,\cdot}_{*}$ is the dual pairing with $\x_T$.  Note that the third condition in \eqref{l_weak_solution_pressureless} makes sense in light of Lemma \ref{l_x_time_diff}.  Our formulation requires only that $u \in \x_T$, which means that $\bar{F} \in (\x_T)^*$ is natural.  Within the context of problem \eqref{l_linear_forced}, $\bar{F}$ is most naturally of the form appearing on the right side of \eqref{l_weak_motivation}; if $\bar{F}$ admits a representation of this form, we may say that a solution to \eqref{l_weak_solution_pressureless} is a weak solution of \eqref{l_linear_forced}.
We will not need to directly construct weak solutions to \eqref{l_weak_solution_pressureless}.  Rather, weak solutions will arise as a byproduct of our construction of strong solutions to \eqref{l_linear_forced}.

Now we turn to the construction of strong solutions to \eqref{l_linear_forced}. We will first construct the strong solutions with lower regularity by assuming that
\begin{equation}\label{l_F_assumptions}
\begin{split}
&F^1 \in L^2([0,T]; H^1(\Omega)) \cap C^0([0,T];H^0(\Omega)), \\
&F^3 \in L^2([0,T]; H^{3/2}(\Sigma)) \cap C^0([0,T];H^{1/2}(\Sigma)), \\
&\dt(F^1 - F^3) \in (\x_T)^*,\text{ and }u_0 \in H^2(\Omega) \cap \x(0).
\end{split}
\end{equation}

The solution that we construct will satisfy \eqref{l_linear_forced} in the strong sense, but we will also show that $D_t u$ satisfies an equation of the form \eqref{l_linear_forced} in the weak sense of \eqref{l_weak_solution_pressureless}.  Here we define
\begin{equation}\label{l_Dt_def}
 D_t u := \dt u - R u \text{ for } R:= \dt M M^{-1}
\end{equation}
with $M$ the matrix defined by \eqref{l_M_def}. We employ the operator $D_t$ because it preserves the $\diva-$free condition.  Before turning to the result, we define the quantities
\begin{equation}\label{l_K_def}
 \mathcal{K}(\eta) : = \sup_{0 \le t \le T} \left( \norm{\eta}_{9/2}^2 + \norm{\dt \eta}_{7/2}^2 + \norm{\dt^2 \eta}_{5/2}^2 \right)\text{ and }\mathcal{L}(\eta) : = \sup_{0 \le t \le T}\norm{\eta}_{7/2}^2.
\end{equation}

\begin{thm}\label{l_strong_solution}
Suppose that $F^1,F^3,u_0$ satisfy \eqref{l_F_assumptions}, and that $u_0$, $F^3(0)$ satisfy the compatibility condition
\begin{equation}\label{l_ss_01}
 \Pi_0 \left( F^3(0) + \sg_{\a_0} u_0 \n_0 \right) =0,
\end{equation}
where $\Pi_0$ is the orthogonal projection onto the tangent space of the surface $\{x_3 = \eta_0\}$ defined according to
\begin{equation}\label{l_Pi0_def}
 \Pi_0 v = v - (v\cdot \n_0) \n_0 \abs{\n_0}^{-2}.
\end{equation}
Further suppose that $\mathcal{K}(\eta)<\infty$, and that $\mathcal{L}(\eta)$ is less than the smaller of $\ep_0$ from Lemma \ref{l_norm_equivalence} and  Proposition \ref{l_stokes_regularity}.  Then there exists a unique strong solution $(u,p)$ to \eqref{l_linear_forced} so that
\begin{equation}\label{l_ss_06}
\begin{split}
& u  \in \x_T \cap C ([0,T]; H^2(\Omega)  ) \cap L^2([0,T]; H^3(\Omega)  ), \\
& \dt u  \in C ([0,T]; H^0(\Omega) ) \cap L^2([0,T]; H^1(\Omega)  ), D_t u \in \x_T ,\dt^2 u \in (\x_T)^*, \\
&  p \in C ([0,T]; H^1(\Omega) ) \cap L^2([0,T]; H^2(\Omega) ).
\end{split}
\end{equation}
The solution satisfies the estimate
\begin{equation}\label{l_ss_02}
\begin{split}
   &\norm{u}_{L^\infty H^2}^2 + \norm{u}_{L^2 H^3}^2 + \norm{\dt u}_{L^\infty H^0}^2 + \norm{\dt u}_{L^2 H^1}^2  + \norm{\dt^2 u}_{(\x_T)^*}^2 + \norm{p}_{L^\infty H^1}^2  + \norm{p}_{L^2 H^2}^2
\\&\quad\le   P(1+ \mathcal{K}(\eta)) \exp\left( P(1+ \mathcal{K}(\eta)) T \right)
\left( \norm{u_0}_{2}^2 + \norm{F^1(0)}_{0}^2 + \norm{F^3(0)}_{ 1/2}^2 \right. \\
&\qquad \left.     +\norm{F^1}_{L^2H^1}^2   +  \norm{F^3}_{L^2H^{3/2}}^2 + \norm{\dt (F^1 - F^3)}_{(\x_T)^*}^2
  \right).
  \end{split}
\end{equation}
The initial pressure, $p(0)\in H^1(\Omega)$, is determined in terms of $u_0, F^1(0), F^3(0)$ as the weak solution to
\begin{equation}\label{l_ss_04}
 \begin{cases}
  \diverge_{\a_0}(\nab_{\a_0} p(0) - F^1(0)) = -\diverge_{\a_0} (R_0 u_0 ) \in H^0(\Omega) \\
  p(0) = ( F^3(0)  + \sg_{\a_0} u_0  \n_0 )\cdot \n_0 \abs{\n_0}^{-2} \in H^{1/2}(\Sigma) \\
  (\nab_{\a_0} p(0) - F^1(0)) \cdot \nu = \Delta_{\a_0} u_0 \cdot \nu \in H^{-1/2}(\Sigma_b)
 \end{cases}
\end{equation}
in the sense of \eqref{l_a_poisson_div}. Also, $D_t u(0) = \dt u(0) - R_0 u_0$ satisfies
\begin{equation}\label{l_ss_05}
 D_t u(0) = \Delta_{\a_0} u_0 - \nab_{\a_0} p(0) + F^1(0) - R_0 u_0 \in \y(0).
\end{equation}

Moreover, $D_t u$ satisfies
\begin{equation}\label{l_ss_03}
  \begin{cases}
\dt (D_t u) - \da (D_t u) + \naba (\dt p) = D_t F^1 + G^1 & \text{in }\Omega \\
\diva(D_t u)=0 & \text{in }\Omega\\
\Sa(\dt p, D_t u) \n = \dt F^3 + G^3 & \text{on }\Sigma \\
D_t u =0 & \text{on }\Sigma_b,
 \end{cases}
\end{equation}
in the weak sense of \eqref{l_weak_solution_pressureless}, where $G^1,G^3$ are defined by
\begin{equation}
\begin{split}
  G^1  =  & -(R + \dt J K) \da u - \dt R u + (\dt J K + R + R^T) \naba p \\&+
\diva( \sg_{\a} (Ru) - R \sg_{\a}u + \sg_{\dt \a} u ),
 \\
 G^3 = &\sg_{\a}(R u) \n - (p I - \sg_{\a} u) \dt \n + \sg_{\dt \a} u \n.
 \end{split}
\end{equation}
Here the inclusions \eqref{l_ss_06} guarantee that $G^1$ and $G^3$ satisfy the same inclusions as $F^1, F^3$ listed in \eqref{l_F_assumptions}, whereas \eqref{l_ss_04} guarantees that the initial data $D_t u(0) \in \y(0)$.
\end{thm}

\begin{proof}
The theorem is almost same as Theorem 4.3 of \cite{GT_lwp}, and we shall first give a brief sketch of the proof and then explain the differences and the necessary modifications in our situation.

The result will be established by first solving the pressureless problem in the weak sense of \eqref{l_weak_solution_pressureless} via a time-dependent Galerkin method where the desired countable basis of $\x_T$ may be constructed with aid of Proposition \ref{l_M_iso}. We then use Proposition \ref{l_pressure_decomp} to introduce the pressure as a Lagrange multiplier, which gives a weak solution to \eqref{l_linear_forced}. We then apply Proposition \ref{l_stokes_regularity} to improve the regularity of the solution so that the weak solution is indeed a strong solution.

Along the proof the only difference in the assumptions of the two theorems is that Theorem 4.3 of \cite{GT_lwp} requires that $\mathcal{K}(\eta)$ is sufficiently small, and our theorem only requires that part of $\mathcal{K}(\eta)$, $\mathcal{L}(\eta) $, is sufficiently small but $\mathcal{K}(\eta)$ can be arbitrarily large. Note that the smallness of $\mathcal{L}(\eta) $ is sufficient to guarantee the applications of the lemmas and propositions used in the proof of Theorem 4.3 of \cite{GT_lwp}. So the proof of Theorem 4.3 of \cite{GT_lwp} works in our present case, and the only difference resulting in the conclusions of the two theorems is that the polynomial $ P(1+ \mathcal{K}(\eta))$ appears in \eqref{l_ss_02}; in Theorem 4.3 of \cite{GT_lwp} since $\mathcal{K}(\eta)$ is small, such polynomial is replaced there accordingly by $1+ \mathcal{K}(\eta)$ as $ P(1+ \mathcal{K}(\eta))\ls 1+ \mathcal{K}(\eta)$.
\end{proof}

Now we investigate the higher regularity of the strong solution obtained in Theorem \ref{l_strong_solution}.
In order to define the forcing terms and initial data for the problem that results from temporally differentiating \eqref{l_linear_forced} several times, we define some mappings.
Given $F^1,F^3,v,q$ we define the mappings for forcing terms
\begin{equation}\label{l_G_def}
\begin{array}{ll}
 \dis \mathfrak{G}^1(v,q)  =   -(R + \dt J K) \da v - \dt R v + (\dt J K + R + R^T) \naba q& \\\dis
\qquad\qquad\ \ + \diva( \sg_{\a} (R v) - R \sg_{\a} v + \sg_{\dt \a} v )  &\text{on } \Omega,\\\dis
 \mathfrak{G}^3(v,q) =  \sg_{\a}(R v) \n - (q I - \sg_{\a} v) \dt \n + \sg_{\dt \a} v \n&\text{on } \Sigma,
\end{array}
\end{equation}
and the mappings for the initial data
\begin{equation}\label{l_f_def}
\begin{array}{ll}
\dis\mathfrak{G}^0(F^1,v,q) = \Delta_{\a} v - \nab_{\a} q + F^1 - R v &\text{on } \Omega, \\\dis
 \mathfrak{f}^1(F^1,v) = \diverge_{\a} (F^1 - R v )&\text{on } \Omega, \\\dis
 \mathfrak{f}^2(F^3,v) = ( F^3  + \sg_{\a} v  \n) \cdot  \n \abs{\n}^{-2}&\text{on } \Sigma, \\\dis
 \mathfrak{f}^3(F^1,v) = (F^1+ \Delta_{\a} v) \cdot \nu&\text{on } \Sigma_b.
\end{array}
\end{equation}
In the definitions of $\mathfrak{G}^i$ and $\mathfrak{f}^i$ we assume that $\a,\n,R$, etc are evaluated at the same $t$ as $F^1,F^3, v,q$.  These mappings allow us to define the forcing terms as follows. We write $F^{1,0} = F^1$ and $F^{3,0} = F^3$, and then we recursively define that for $j=1,\dotsc,2N$,
\begin{equation}\label{l_Fj_def}
\begin{array}{ll}
\dis F^{1,j}  := D_t F^{1,j-1} + \mathfrak{G}^1(D_t^{j-1} u,\dt^{j-1} p) = D_t^j F^1 + \sum_{\ell=0}^{j-1} D_t^\ell \mathfrak{G}^1(D_t^{j-\ell-1} u, \dt^{j-\ell-1} p)\text{ on } \Omega,
\\ \dis
 F^{3,j}   := \dt F^{3,j-1} + \mathfrak{G}^3(D_t^{j-1} u,\dt^{j-1} p) = \dt^j F^3 + \sum_{\ell=0}^{j-1} \dt^\ell \mathfrak{G}^3(D_t^{j-\ell-1} u \dt^{j-\ell-1} p)\quad\,\,\,\,\text{on } \Sigma.
\end{array}
\end{equation}
These are the forcing terms that appear when we apply $j$ temporal derivatives to \eqref{l_linear_forced} (see \eqref{l_lwp_01}). In order to estimate these forcing terms, we define
\begin{equation}\label{l_Ffrak_def}
\begin{split}
&\mathfrak{F} (F^1,F^3) := \sum_{j=0}^{2N-1} \ns{\dt^j F^1}_{L^2 H^{4N-2j -1}}+\ns{\dt^{2N} F^1}_{L^2 (\H1)^\ast} +  \sum_{j=0}^{2N } \ns{\dt^j F^3}_{L^2 H^{4N-2j -1/2}}  \\
& \qquad\qquad\quad\ + \sum_{j=0}^{2N-1} \ns{\dt^j F^1}_{L^\infty H^{4N-2j -2}} + \ns{\dt^j F^3}_{L^\infty H^{4N-2j -3/2}}, \\
&\mathfrak{F}_0(F^1,F^3) :=  \sum_{j=0}^{2N-1} \ns{\dt^j F^1(0)}_{4N-2j -2} + \ns{\dt^j F^3(0)}_{4N-2j -3/2}.
\end{split}
\end{equation}
Lemma \ref{l_sobolev_infinity} implies that if $\mathfrak{F}(F^1,F^3) < \infty$, then for $j=0,\dotsc,2N-1$,
\begin{equation}
 \dt^j F^1 \in C ([0,T];H^{4N-2j-2}(\Omega)) \text{ and } \dt^j F^3 \in C ([0,T];H^{4N-2j-3/2}(\Sigma)).
\end{equation}
The same lemma also implies that the sum of the $L^\infty H^{k}$ norms in the definition of $\mathfrak{F}(F^1,F^3)$ can be bounded by $\mathfrak{F}_0(F^1,F^3)$ plus the sum of the $L^2 H^{k}$ norms. For $\eta$ we define
\begin{equation}\label{l_Kfrak_def}
\begin{split}
&\mathfrak{F}  (\eta)    := \ns{ \eta}_{L^2 H^{4N+1/2}}+\sum_{j=1}^{2N+1} \ns{\dt^j \eta}_{L^2 H^{4N-2j+3/2}}
+\ns{ \eta}_{L^\infty H^{4N }}+ \sum_{j=1}^{2N} \ns{\dt^j \eta}_{L^\infty H^{4N-2j+1/2}},
\\&\mathfrak{F}_0 (\eta):=  \ns{ \eta(0)}_{4N }+\sum_{j=1}^{2N} \ns{\dt^j \eta(0)}_{4N-2j+1/2}.
\end{split}
\end{equation}
Again, if $\mathfrak{F} (\eta) <\infty$, then  $ \eta \in C ([0,T];H^{4N }(\Sigma))$ and $\dt^j \eta \in C ([0,T];H^{4N-2j+1/2}(\Sigma))$  for $j=1,\dotsc,2N$. We also define
\begin{equation}\label{1_L0frak_def}
\mathfrak{L}(\eta):= \sup_{0\le t\le T}\norm{\eta}_{H^{4N-1/2}}^2\text{ and }\mathfrak{L}_0(\eta):= \norm{\eta_0}_{ {4N-1/2}}^2.
\end{equation}
Note that $\mathcal{K}(\eta) \le  \mathfrak{F} (\eta)$ and $\mathcal{L}(\eta) \le  \mathfrak{L}(\eta)$, where $\mathcal{K}(\eta)$ and $\mathcal{L}(\eta)$ are defined by \eqref{l_K_def}.

We now record an estimate of $F^{i,j}$ in terms of $\mathfrak{F}(F^1,F^3), \mathfrak{F}(\eta)$ and certain norms of $u,p$.

\begin{lem}\label{l_iteration_estimates_1}
For $m =1,\dotsc,2N-1$ and $j=1,\dotsc,m$,
\begin{equation}\label{l_ie1_01}
\begin{split}
 &\ns{F^{1,j}}_{L^2 H^{2m-2j+1}}  +  \ns{F^{3,j}}_{L^2 H^{2m-2j+3/2}} \le  P(1+ \mathfrak{F}(\eta)) \bigg(  \mathfrak{F}(F^1,F^3)
\\ &\left.\, \,
+ \sum_{\ell=0}^{j-1} \ns{\dt^\ell u}_{L^2 H^{2m-2j+3}}+  \ns{\dt^\ell u}_{L^\infty H^{2m-2j+2}} + \ns{\dt^\ell p}_{L^2 H^{2m-2j+2}} + \ns{\dt^\ell p}_{L^\infty H^{2m-2j+1}}
   \right),
   \end{split}
\end{equation}
\begin{equation}\label{l_ie1_02}
\begin{split}
 \ns{F^{1,j}}_{L^\infty H^{2m-2j}}  + & \ns{F^{3,j}}_{L^\infty H^{2m-2j+1/2}} \le  P(1+ \mathfrak{F}(\eta)) \bigg(  \mathfrak{F}(F^1,F^3)    \\&\qquad\qquad\qquad
\left. +\sum_{\ell=0}^{j-1}  \ns{\dt^\ell u}_{L^\infty H^{2m-2j+2}}+\ns{\dt^\ell p}_{L^\infty H^{2m-2j+1}}    \right),
   \end{split}
\end{equation}
and
\begin{equation}\label{l_ie1_03}
\begin{split}
 &\ns{\dt (F^{1,m} -F^{3,m})  }_{(\x_T)^*}
 \le  P(1+ \mathfrak{F}(\eta)) \bigg(  \mathfrak{F}(F^1,F^3) +\ns{\dt^m u}_{L^2 H^{2}}+ \ns{\dt^{m} p}_{L^2 H^1}        \\
&\qquad\qquad\qquad\qquad\quad\ \,\left.  + \sum_{\ell=0}^{m-1}\ns{\dt^\ell u}_{L^\infty H^{2}} + \ns{\dt^\ell u}_{L^2 H^{3}}+ \ns{\dt^\ell p}_{L^\infty H^{1}} +  \ns{\dt^\ell p}_{L^2 H^{2}}
\right).
   \end{split}
\end{equation}

Similarly, for  $j = 1,\dotsc, 2N-1$,
\begin{equation}\label{l_ie1_04}
\begin{split}
&\ns{F^{1,j}(0)}_{4N-2j-2} + \ns{F^{3,j}(0)}_{4N-2j-3/2}
\\  &\quad\le  P(1+ \mathfrak{F}_0 (\eta)) \left( \mathfrak{F}_0(F^1,F^3) + \sum_{\ell=0}^{j-1}
\ns{\dt^\ell u(0)}_{4N-2\ell} + \ns{\dt^\ell p(0)}_{4N-2\ell-1} \right).
   \end{split}
\end{equation}

\end{lem}
\begin{proof}
The lemma follows from simple but lengthy computations, invoking standard arguments. Although our lemma is similar as Lemma 4.5 of \cite{GT_lwp}, we need to point out the differences and present the necessary modifications of the proof there.
The first point is that we assume $\mathfrak{F}(\eta)$ to be finite but not small, so the polynomial $P$ appears in our estimates. The second point is that some norms in our $\mathfrak{F}(\eta)$ are weaker than those of $\mathfrak{K}(\eta)$ in Lemma 4.5 of \cite{GT_lwp}, but a careful check that we shall explain below would reveal that they are sufficient to guarantee the validity of these estimates. Hence, our estimate is a refined version of Lemma 4.5 in \cite{GT_lwp} in some sense.

We use the definition of $F^{1,j}, F^{3,j}$ given by \eqref{l_Fj_def} and expand all terms using the Leibniz rule and the definition of $D_t$ (defined by \eqref{l_Dt_def}) to rewrite $F^{i,j}$ as a sum of products of two terms: one involving products of various derivatives of $\bar{\eta}$, and one linear in  derivatives of $u$, $p$, $F^1$, or $F^3$. We only focus on the $\bar{\eta}$ part to see that it can be bounded by our polynomial, while the other part follows in the same way as Lemma 4.5 of \cite{GT_lwp}. Note that the highest derivatives of $\bar\eta$ appearing in $F^{1,j}$ are $\dt^j\nab^3\bar{\eta}$ and $\dt^{j+1}\nab \bar{\eta}$, while the one appearing in $F^{3,j}$ is $\dt^j\nab^2\bar{\eta}$. To prove \eqref{l_ie1_01}, by Lemma \ref{p_poisson} and the trace embeddings, we have that for $j=1,\dots,2N-1$,
\begin{equation}\label{obser}
\begin{split}
&\ns{\dt^j\nab^3\bar{\eta}}_{L^2 H^{4N-2j-1}}\ls \ns{\dt^j  \eta }_{L^2 H^{4N-2j+3/2}}\le \mathfrak{F}(\eta) ,   \\
&\ns{\dt^{j+1}\nab \bar{\eta}}_{L^2 H^{4N-2j-1}}\ls \ns{\dt^{j+1}  \eta }_{L^2 H^{4N-2j-1/2}}=\ns{\dt^{j+1}  \eta }_{L^2 H^{4N-2(j+1)+3/2}}\le \mathfrak{F}(\eta),  \text{ and }  \\
&\ns{\dt^j\nab^2\bar{\eta}}_{L^2 H^{4N-2j-1/2}(\Sigma)}\ls \ns{\dt^j  \nab^2\bar{\eta} }_{L^2 H^{4N-2j }}
\ls \ns{\dt^j    {\eta} }_{L^2 H^{4N-2j+3/2 }}\le \mathfrak{F}(\eta).
\end{split}
\end{equation}
With \eqref{obser} in hand, we may conclude \eqref{l_ie1_01} as in Lemma 4.5 of \cite{GT_lwp} by estimating the resulting products by using Lemma \ref{i_sobolev_product_1} in conjunction with the usual Sobolev embeddings, trace embeddings and Lemma \ref{p_poisson}.  The estimates \eqref{l_ie1_02} and \eqref{l_ie1_04} follow from similar arguments.

To prove \eqref{l_ie1_03},
we note that the highest derivatives of $\eta$ are $\dt^{2N}\nab^3\bar{\eta}$ and $\dt^{2N+1}\nab \bar{\eta}$ appearing in $\dt F^{1,m}$ and  $\dt^{2N}\nab^2\bar{\eta}$ appearing in $\dt F^{3,m}$ when $m=2N-1$. We remark that the arguments for (4.78) in Lemma 4.5 of \cite{GT_lwp} or the way of \eqref{obser} can not be applied to estimate \eqref{l_ie1_03} due to our weaker $\mathfrak{F}(\eta)$. We must resort to use the structure of the estimated terms. First, to control $\dt^{2N+1}\nab \bar{\eta}$ in $\Omega$, we have that for all $\varphi\in \H1 $,
\begin{equation}
\br{\dt^{2N+1}\nab \bar{\eta},\varphi}_{*}  =\int_\Sigma\dt^{2N+1}   {\eta} \varphi-\int_\Omega\dt^{2N+1} \bar{\eta} \nab \varphi.
\end{equation}
Then by the duality, the trace embedding and Lemma \ref{p_poisson}, we obtain
\begin{equation}\label{obser1}
\ns{\dt^{2N+1}\nab \bar{\eta}}_{L^2 (\H1)^*}\ls \ns{\dt^{2N+1}   {\eta}}_{L^2 H^{-1/2}}+\ns{\dt^{2N+1}   \bar{\eta}}_{L^2 H^{0}}\ls \ns{\dt^{2N+1}   {\eta}}_{L^2 H^{-1/2}}\le \mathfrak{F}(\eta).
\end{equation}
On the other hand, we have to treat together the terms in the $\dt (F^{1,2N-1} -F^{3,2N-1})$ that leads to $\dt^{2N}\nab^3\bar{\eta}$ in $\Omega$ and $\dt^{2N}\nab^2\bar{\eta}$ on $\Sigma$. In view of \eqref{l_Fj_def} for $j=2N-1$, these terms are $\diva( \sg_{\a} (\dt^{2N-1} R u)) $ in $\Omega$ and $ -\sg_{\a} (\dt^{2N-1}R u) \n $ on $\Sigma$, respectively. We need to estimate $ \ns{  \diva( \sg_{\a} ( \dt^{2N-1} R u)) - \sg_{\a} (\dt^{2N-1} R u) \n }_{(\x_T)^*}$. Indeed, for $\psi\in \x_T$, we observe that
\begin{equation}
\br{  \diva( \sg_{\a} ( \dt^{2N-1} R u)) - \sg_{\a} (\dt^{2N-1} R u) \n  ,\psi}_{*}  =\ip{\dt^{2N-1} R u}{\psi}_{\h^1_T}.
\end{equation}
Then by the duality we have
\begin{equation} \label{obser22}
 \ns{ \diva( \sg_{\a} ( \dt^{2N-1} R u)) - \sg_{\a} (\dt^{2N-1} R u) \n  }_{(\x_T)^*} \ls\ns{\dt^{2N-1} R u}_{L^2H^1}.
\end{equation}
To estimate the right hand side of \eqref{obser22}, we may use the bound, by Lemma \ref{p_poisson},
\begin{equation}\label{obser33}
 \ns{\dt^{2N }\nab\bar{\eta}}_{L^2H^1}\ls  \ns{\dt^{2N } {\eta}}_{L^2H^{3/2}}\le \mathfrak{F}(\eta).
\end{equation}
 Then with \eqref{obser1}, \eqref{obser22} and \eqref{obser33} in hand, we can conclude \eqref{l_ie1_03} as that for \eqref{l_ie1_01}.
\end{proof}

Next we record an estimates for the difference between $\dt v$ and $D_t v$ for a general $v$.
\begin{lem}\label{l_iteration_estimate_3}
If $k = 0,\dotsc,4N-1$ and $v$ is sufficiently regular, then
\begin{equation}\label{l_ie3_01}
\ns{\dt v - D_t v}_{L^2 H^k} \le   P(1+ \mathfrak{F}(\eta))  \ns{v}_{L^2 H^k},
\end{equation}
and if $k=0,\dotsc,4N-2$, then
\begin{equation}\label{l_ie3_06}
\ns{\dt v - D_t v}_{L^\infty H^{k}} \le   P(1+ \mathfrak{F}(\eta))  \ns{v}_{L^\infty H^{k}}.
\end{equation}

If $m=1,\dotsc,2N-1$, $j=1,\dotsc,m$,  and $v$ is sufficiently regular, then
\begin{equation}\label{l_ie3_02}
 \ns{\dt^j v - D_t^j v}_{L^2 H^{2m-2j+3}}  \le   P(1+ \mathfrak{F}(\eta)) \sum_{\ell=0}^{j-1}\left( \ns{\dt^\ell v}_{L^2 H^{2m-2j+3}}  + \ns{\dt^\ell v}_{L^\infty H^{2m-2j+2}} \right),
\end{equation}
\begin{equation}\label{l_ie3_05}
 \ns{\dt^j v - D_t^j v}_{L^\infty H^{2m-2j+2}}  \le   P(1+ \mathfrak{F}(\eta)) \sum_{\ell=0}^{j-1} \ns{\dt^\ell v}_{L^\infty H^{2m-2j+2}},
\end{equation}
and
\begin{equation}\label{l_ie3_03}
\begin{split}
& \ns{\dt D_t^m v - \dt^{m+1} v}_{L^2 H^1} + \ns{\dt^2 D_t^{m} v - \dt^{m+2} v}_{(\x_T)^*} \\
&\quad\le   P(1+ \mathfrak{F}(\eta))  \left( \sum_{\ell=0}^{m} \ns{\dt^\ell v}_{L^2 H^{1}} + \ns{\dt^\ell v}_{L^\infty H^{2}}  + \ns{\dt^{m+1} v}_{(\x_T)^*}  \right).
\end{split}
\end{equation}

Also, if $j=0,\dotsc,2N$, and $v$ is sufficiently regular, then
\begin{equation}\label{l_ie3_04}
 \ns{\dt^j v(0) - D_t^j v(0)}_{4N-2j} \le   P(1+ \mathfrak{F}_0 (\eta))\sum_{\ell=0}^{j-1} \ns{\dt^\ell v(0)}_{4N-2\ell}.
\end{equation}
\end{lem}
\begin{proof}
The proof is similar to that of Lemma \ref{l_iteration_estimates_1}, and is thus omitted.
\end{proof}

\begin{lem}\label{l_iteration_estimates_2}
Suppose that $v,$ $q,$ $G^1,$ $G^3$ are evaluated at $t=0$ and are sufficiently regular.  For $j=0,\dotsc,2N-1$, we have
\begin{equation}\label{l_ie2_02}
\begin{split}
 &\ns{\mathfrak{G}^0(G^1,v,q)}_{4N-2j-2} \\
&\quad\le  P(1+ \ns{\eta(0)}_{4N } + \ns{\dt \eta(0)}_{4N-3/2}) \left(  \ns{v}_{4N-2j} + \ns{q}_{4N-2j-1} + \ns{G^1}_{4N-2j-2}  \right).
\end{split}
\end{equation}
If $j=0,\dotsc,2N-2$, then
\begin{equation}\label{l_ie2_01}
\begin{split}
 &\ns{\mathfrak{f}^1(G^1,v)}_{4N-2j-3} + \ns{\mathfrak{f}^2(G^3,v)}_{4N-2j-3/2} + \ns{\mathfrak{f}^3(G^1,v)}_{4N-2j-5/2} \\
&\quad\le  P(1+ \ns{\eta(0)}_{4N } + \ns{\dt \eta(0)}_{4N-3/2})\left( \ns{G^1}_{4N-2j-2} + \ns{G^3}_{4N-2j-3/2}  + \ns{v}_{4N-2j}\right).
\end{split}
\end{equation}
For $j=2N-1$, if  $\diverge_{\a(0)} v=0$ in $\Omega$, then
\begin{equation}\label{l_ie2_03}
 \ns{\mathfrak{f}^2(G^3,v)}_{1/2} + \ns{\mathfrak{f}^3(G^1,v)}_{-1/2}
\le  P(1+ \ns{\eta(0)}_{4N} )\left( \ns{G^1}_{2} + \ns{ G^3}_{1/2}  + \ns{v}_{2}\right).
\end{equation}
\end{lem}
\begin{proof}
 The proof is similar to that of Lemma \ref{l_iteration_estimates_1}. We note that to derive the $\mathfrak{f}^3$ estimate of \eqref{l_ie2_03}, we may use the bound $\ns{\Delta_{\a(0)} v \cdot \nu }_{H^{-1/2}(\Sigma_b)} \ls \ns{\Delta_{\a(0)} v}_{0}$ provided by Lemma \ref{l_boundary_dual_estimate} since $\diverge_{\a(0)} \Delta_{\a(0)} v=0$, which is implied by that $\diverge_{\a(0)} v=0$.
\end{proof}

Now we assume that $u_0 \in H^{4N}(\Omega)$, $\eta_0 \in
H^{4N}(\Sigma)$, $\mathfrak{F}_0(F^1,F^3) < \infty$ and
$\mathfrak{F}_0 (\eta)<\infty$, and that $ \mathfrak{L}_0(\eta)\le
1$ (as defined in \eqref{1_L0frak_def}) is sufficiently small for
the hypothesis of Proposition \ref{l_a_poisson_regularity} to hold
when $k=4N$.  Note, though, that we do not need $\ns{\eta_0}_{4N}$
and $\ns{u_0}_{4N}$ to be small. We will iteratively construct the
initial data $D_t^j u(0)$ for $j=0,\dotsc,2N$ and $\partial_t^j
p(0)$ for $j=0,\dotsc,2N-1$. To do so, we will first construct all
but the highest order data,  and then we will state some
compatibility conditions for the data.  These are necessary to
construct $D_t^{2N} u(0)$ and $\dt^{2N-1}p(0)$, and to construct
high-regularity solutions in Theorem \ref{l_linear_wp}.

We now turn to the construction of  $D_t^j u(0)$ for $j=0,\dotsc,2N-1$ and $\dt^j p(0)$ for $j=0,\dotsc,2N-2$. For $j=0$ we write $F^{1,0}(0) = F^{1}(0) \in H^{4N-2}$, $F^{3,0}(0) = F^3(0) \in H^{4N-3/2}$, and $D_t^0 u(0) = u_0 \in H^{4N}.$  Suppose now that $F^{1,\ell} \in H^{4N - 2\ell-2}$, $F^{3,\ell} \in H^{4N-2\ell-3/2}$, and $D_t^\ell u(0) \in H^{4N-2\ell}$ are given for $0\le \ell \le j \in [0,2N-2]$; we will define $\dt^j p(0) \in H^{4N-2j-1}$ as well as $D_t^{j+1}u(0) \in H^{4N-2j-2}$, $F^{1,j+1}(0) \in H^{4N-2j-4}$, and $F^{3,j+1}(0) \in H^{4N-2j-7/2}$, which allows us to define all of said data via iteration.  By virtue of estimate \eqref{l_ie2_01}, we know that $f^1 = \mathfrak{f}^1(F^{1,j}(0),D_t^j u(0)) \in H^{4N-2j-3},$ $f^2 = \mathfrak{f}^2(F^{3,j}(0),D_t^j u(0)) \in H^{4N-2j-3/2},$ and $f^3 = \mathfrak{f}^3(F^{1,j}(0),D_t^j u(0)) \in H^{4N-2j-5/2}$.  This allows us to define $\dt^j p (0)$ as the solution to \eqref{l_linear_elliptic_p} with this choice of $f^1, f^2, f^3$, and then Proposition \ref{l_a_poisson_regularity} with $k = 4N$ and $r=4N-2j-1 < k$ implies that $\dt^j p(0) \in H^{4N-2j-1}$.  Now the estimates \eqref{l_ie1_04}, \eqref{l_ie3_04}, and \eqref{l_ie2_02} allow us to define
\begin{equation}
 \begin{split}
D_t^{j+1} u(0) &:= \mathfrak{G}^0(F^{1,j}(0), D_t^j u(0),\dt^j p(0)) \in H^{4N-2j-2},  \\
F^{1,j+1}(0) &:= D_t F^{1,j}(0) + \mathfrak{G}^1(D_t^j u(0) ,\dt^j p(0)) \in H^{4N-2j-4}, \text{ and }\\
F^{3,j+1}(0) &:= \dt F^{3,j}(0) + \mathfrak{G}^3(D_t^j u(0) , \dt^j p(0)) \in H^{4N-2j-7/2}.
 \end{split}
\end{equation}
Hence, we have iteratively constructed all of the desired data except for $D_t^{2N} u(0)$ and $\dt^{2N-1}p(0)$.

By construction, the initial data $D_t^j u(0)$ and $\dt^j p(0)$ are determined in terms of $u_0$ as well as  $\dt^\ell F^1(0)$ and $\dt^\ell F^3(0)$ for $\ell = 0,\dotsc,2N-1$.  In order to use these in Theorem \ref{l_strong_solution} and to construct $D_t^{2N} u(0)$ and $\dt^{2N-1}p(0)$, we must enforce compatibility conditions for $j=0,\dotsc,2N-1$.  For such $j$, we say that the $j^{th}$ compatibility condition is satisfied if
\begin{equation}\label{l_comp_cond}
 \begin{cases}
  D_t^j u(0) \in \x(0) \cap H^2(\Omega) \\
  \Pi_0( F^{3,j}(0) + \sg_{\a_0} D_t^j u(0) \n_0)=0 .
 \end{cases}
\end{equation}
Note that the construction of $D_t^j u(0)$ and $\dt^j p(0)$ ensures that $D_t^j u(0) \in H^2(\Omega)$ and that $\diverge_{\a_0}(D_t^j u(0))=0$, so the condition $D_t^j u(0) \in \x(0) \cap H^2(\Omega)$ may be reduced to the condition $D_t^j u(0) \vert_{\Sigma_b} =0$.

It remains only to define $\dt^{2N-1} p(0)\in H^1$ and $D_t^{2N} u(0) \in H^0$.  According to the $j=2N-1$ compatibility condition \eqref{l_comp_cond}, $\diverge_{\a_0} D_t^{2N-1} u(0)=0$, which means that we can use the estimate  \eqref{l_ie2_03} to see that $f^2 = \mathfrak{f}^2(F^{3,2N-1}(0),D_t^{2N-1} u(0)) \in H^{1/2}$ and $f^3 = \mathfrak{f}^3(F^{1,2N-1}(0),D_t^{2N-1} u(0)) \in H^{-1/2}$.  We also see from \eqref{l_comp_cond} that if we define $g_0 = -\diverge_{\a_0} ( R_0 D_t^{2N-1} u(0))$ then $g_0 \in H^0$.  Then, owing to the fact that $G = -F^{1,2N-1} \in H^0$, we can  define $\dt^{2N-1} p(0) \in H^1$ as a weak solution to \eqref{l_linear_elliptic_p} in the sense of \eqref{l_a_poisson_div}. Then by the estimate \eqref{l_ie2_02} we define
\begin{equation}
D_t^{2N} u(0) = \mathfrak{G}^0( F^{1,2N-1}(0), D_t^{2N-1} u(0), \dt^{2N-1} p(0)) \in H^0.
\end{equation}
In fact, the construction of $\dt^{2N-1} p(0)$ guarantees that $D_t^{2N} u(0) \in \y(0)$.  In addition to providing the above inclusions, the bounds \eqref{l_ie1_04}, \eqref{l_ie2_01}, \eqref{l_ie2_02} also imply the estimate
\begin{equation}\label{l_init_data_estimate}
 \sum_{j=0}^{2N} \ns{ D_t^j u(0) }_{4N-2j} + \sum_{j=0}^{2N-1} \ns{ \dt^j p(0) }_{4N-2j-1}
   \le  P(1 + \mathfrak{F}_0(\eta)) \left(  \ns{u_0}_{4N} + \mathfrak{F}_0(F^1,F^3) \right).
\end{equation}
Note that, owing to \eqref{l_ie3_04},  \eqref{l_init_data_estimate} also holds with $\dt^j u(0)$ replacing $D_t^j u(0)$ on the left.

To state our result on higher regularity of solutions to the problem \eqref{l_linear_forced},
we define
\begin{equation}\label{l_DEfrak_def}
\begin{split}
  \mathfrak{K}(u,p)   :=& \sum_{j=0}^{2N} \ns{\dt^j u}_{L^2 H^{4N-2j +1}} + \ns{\dt^{2N+1} u}_{(\x_T)^*} + \sum_{j=0}^{2N-1}\ns{\dt^j p}_{ L^2 H^{4N-2j}}\\
 & +\sum_{j=0}^{2N} \ns{\dt^j u}_{L^\infty H^{4N-2j }} + \sum_{j=0}^{2N-1}\ns{\dt^j p}_{L^\infty H^{4N-2j-1}} .
\end{split}
\end{equation}

\begin{thm}\label{l_linear_wp}
Suppose that $u_0 \in H^{4N}(\Omega)$, $\eta_0 \in H^{4N}(\Sigma)$,
$\mathfrak{F}(F^1,F^3) < \infty$, $\mathfrak{F}(\eta)<\infty$ and
that $ \mathfrak{L}(\eta)$ (as defined in \eqref{1_L0frak_def}) is
less than the smaller of $\ep_0$ from Theorem
\ref{l_strong_solution} and Proposition
\ref{l_a_poisson_regularity}. Let $D_t^j u(0) \in H^{4N-2j}(\Omega)$
and $\dt^j p(0) \in H^{4N-2j-1}$ for $j=0,\dotsc,2N-1$ along with
$D_t^{2N} u(0) \in \y(0)$ all be determined as above in terms of
$u_0$ and $\dt^j F^1(0)$, $\dt^j F^3(0)$ for $j=0,\dotsc,2N-1$.
Suppose that for $j=0,\dotsc,2N-1$, the initial data satisfy the
$j^{th}$ compatibility condition \eqref{l_comp_cond}.

There exists a  $\bar{T} =\bar{T}(\mathfrak{F}(\eta))>0$ so that if $0 < T \le \bar{T}$, then there exists a unique strong solution $(u,p)$ to \eqref{l_linear_forced} on $[0,T]$ so that
\begin{equation}\label{l_lwp_00}
\begin{split}
 &\dt^j u \in C ([0,T]; H^{4N-2j}(\Omega)) \cap L^2([0,T];H^{4N-2j+1}(\Omega)) \text{ for } j=0,\dotsc,2N, \dt^{2N+1} u  \in (\x_T)^*,\\
 &\dt^j p \in C ([0,T]; H^{4N-2j-1}(\Omega)) \cap L^2([0,T];H^{4N-2j}(\Omega)) \text{ for } j=0,\dotsc,2N-1.
\end{split}
\end{equation}
The pair $(D_t^j u, \dt^j p)$ satisfies the PDE
\begin{equation}\label{l_lwp_01}
\begin{cases}
\dt (D_t^j u) - \da (D_t^j u) + \naba (\dt^j p) = F^{1,j} & \text{in }\Omega \\
\diva(D_t^j u)=0 & \text{in }\Omega\\
\Sa(\dt^j p, D_t^j u) \n =  F^{3,j} & \text{on }\Sigma \\
D_t^j u =0 & \text{on }\Sigma_b
\end{cases}
\end{equation}
in the strong sense with initial data $(D_t^j u(0), \dt^j p(0))$ for $j=0,\dotsc,2N-1$, and in the weak sense of \eqref{l_weak_solution_pressureless} with initial data $D_t^{2N} u(0) \in \y(0)$ for $j=2N$.  Here the vectors $F^{1,j}$ and $F^{3,j}$ are as defined by \eqref{l_Fj_def}.  Moreover, the solution satisfies the estimate
\begin{equation}\label{l_lwp_02}
 \mathfrak{K}(u,p) \le   P(1+\mathfrak{F}_0 (\eta) + \mathfrak{F}(\eta)) \exp\left( P(1+ \mathfrak{F}(\eta)) T \right) \left( \ns{u_0}_{4N} + \mathfrak{F}_0(F^1,F^3) + \mathfrak{F}(F^1,F^3) \right).
\end{equation}
\end{thm}
\begin{proof}
The proof is same to that of Theorem 4.8 of \cite{GT_lwp}.  We will only provide a brief sketch of the idea of the proof.  For full details we refer to \cite{GT_lwp}. The estimate \eqref{l_init_data_estimate} gives control of the initial data, the estimates \eqref{l_ie1_01}--\eqref{l_ie1_03} in Lemma \ref{l_iteration_estimates_1} provide control of the forcing terms $F^{1,j},F^{3,j}$ in terms of $F^1,F^3,\eta,u,p$, and the estimates \eqref{l_ie3_01}--\eqref{l_ie3_03} in Lemma \ref{l_iteration_estimate_3} provide control of the difference between $\dt v$ and $D_t v$. Note that the smallness of $\mathfrak{L}(\eta) $ is sufficient to guarantee the applications of the lemmas and propositions used in the proof of Theorem 4.8 of \cite{GT_lwp} so that the proof of Theorem 4.8 of \cite{GT_lwp} works in our present case. These and the  $j^{th}$ compatibility condition \eqref{l_comp_cond} for $j=0,\dotsc,2N-1$ then allow us to iteratively apply Theorem \ref{l_strong_solution} and Proposition \ref{l_stokes_regularity} to find that $(D_t^ju,\partial_t^jp)$ solve \eqref{l_lwp_01} and satisfy the estimates \eqref{l_lwp_02}.
\end{proof}

\subsection{The linear regularized surface $\kappa$-problem }
In this subsection
we will consider the following regularized surface problem with the artificial viscosity term and the compensator:
\begin{equation}\label{l_k transport_equation}
\begin{cases}
\dt \eta  -\kappa \Delta_\ast \eta  =  \kappa\Psi+ F^4\ \text{ in } \Sigma
\\\eta\mid_{t=0}=\eta_0,
\end{cases}
\end{equation}
where $\kappa>0$ is the artificial viscosity parameter and the function $\Psi$ is the so-called compensator that will be determined below. We define
\begin{equation}\label{l_Kfrak_u2n_def1}
\begin{split}
 &\mathfrak{F}(F^4): = \sum_{j=0}^{2N} \ns{\dt^j F^4}_{L^2 H^{4N-2j}}+\sum_{j=0}^{2N-1} \ns{\dt^j F^4}_{L^\infty H^{4N-2j}},
\\&\mathfrak{F}_0(F^4): = \sum_{j=0}^{2N-1} \ns{\dt^j F^4(0)}_{ {4N-2j-1}}
 .
 \end{split}
\end{equation}
Again, Lemma \ref{l_sobolev_infinity} implies that if $ \mathfrak{F}(F^4)<\infty$, then $\dt^j F^4  \in C ([0,T];H^{4N-2j-1}(\Sigma))$ for $j=0,\dots,2N-1$.
To state our result on the solutions to the problem \eqref{l_k transport_equation}, we define
\begin{equation}\label{l_Ketafrak_def}
 \mathfrak{K} (\eta):=\sum_{j=0}^{2N+1} \ns{\dt^j \eta}_{L^2 H^{4N-2j+2}}+\sum_{j=0}^{2N}\ns{\dt^j \eta}_{L^\infty H^{4N-2j +1}} .
\end{equation}

Now we turn to the construction of the compensator function $\Psi$. In order not to break the compatibility conditions required for the initial data $(u_0,\eta_0)$ of the nonlinear system \eqref{geometric} (i.e. \eqref{kgeometric} with $\kappa=0$), we need to require that the initial conditions for \eqref{l_k transport_equation} are same as the one with $\kappa=0$. This is the reason that we introduce
$\Psi$ here. Formally, we use the equation \eqref{l_k transport_equation} with $k=0$ to construct the initial data $\partial_t^j\eta(0)$, and then choose the appropriate function $\Psi$ so that the equation \eqref{l_k transport_equation} with $k>0$ takes the same initial conditions as \eqref{l_k transport_equation} with $k=0$.
We construct $\Psi$ as follows.  We assume that  $ \eta_0\in H^{4N+1}(\Sigma)$, and we construct the initial data $\partial_t^j\eta(0)$   for $j=1,\dots,2N$ as $\partial_t^{j }\eta(0):= \partial_t^{j-1}F^4(0)$.  Now we define $\eta^0$ to be the extension produced by Lemma \ref{2_sobolev_extension} that achieves the initial data $\dt^j\eta(0)$ for $j=0,\dots,2N$. If we define $\Psi=-\Delta_\ast\eta^0$, then $\dt^j\Psi(0)=-\Delta_\ast\dt^j\eta(0)$. This guarantees that we essentially add nothing in \eqref{l_k transport_equation} at the initial time $t = 0$.   Moreover, we have
\begin{equation}\label{compensator}
\begin{split}
 \sum_{j=0}^{2N} \ns{\dt^j \Psi}_{L^2 H^{4N-2j  }}   +  \sum_{j=0}^{2N}\ns{\dt^j \Psi}_{L^\infty H^{4N-2j -1}}
 &\ls \sum_{j=0}^{2N}\ns{ \dt^j\eta(0) }_{ {4N-2j +1}}
  \\&= \ns{\eta_0}_{4N+1}+\mathfrak{F}_0(F^4).
  \end{split}
\end{equation}
Thanks to this compensator, the constructed data $\dt^j\eta(0)$ for $j=1,\dots,2N$ can serve as the initial data for \eqref{l_k transport_equation} for all $\kappa$.

\begin{thm}\label{l_transport_theorem}
Suppose that $\eta_0\in H^{4N+1}(\Sigma)$ and that $\dt^j\eta(0)$ for $j=1,\dots,2N$ are determined as above. Then the problem  \eqref{l_k transport_equation} admits a unique solution $\eta$ that achieves the initial data $\dt^j\eta(0)$ for $j=0,\dots,2N$ and satisfies
\begin{equation}\label{K eta es}
 \mathfrak{K} (\eta)\le C_\kappa\left( \ns{\eta_0}_{4N+1}+\mathfrak{F}_0(F^4)+  \mathfrak{F}(F^4)\right) .
 \end{equation}
\end{thm}
\begin{proof}
The solvability of \eqref{l_k transport_equation} is classical, and we only focus on the derivation of the estimate \eqref{K eta es}.
The standard energy estimates on \eqref{l_k transport_equation} together with Cauchy's inequality yields
\begin{equation}
 \frac{d}{dt}\ns{  \eta}_{ 4N   +1 }+\kappa  \ns{\nab_\ast \eta}_{ 4N +1}
 \ls     \kappa\ns{ \Psi}_{ {4N }} +  \frac{1}{\kappa }\ns{   F^4}_{  {4N  }}  .
\end{equation}
Integrating the above in time, by \eqref{compensator} we obtain
\begin{equation}\label{eta es 0}
\begin{split}
 \ns{  \eta}_{L^\infty H^{ 4N +1} }+\kappa  \ns{  \eta}_{ L^2 H^{4N  +2 }}
& \ls  \ns{ \eta(0)}_{ { 4N +1} } +\kappa\ns{  \Psi}_{ L^2H^{4N   }} +  \frac{1}{\kappa }\ns{  F^4}_{  L^2H^{4N   }}
\\&\ls \ns{\eta_0}_{4N+1}+\mathfrak{F}_0(F^4)+\frac{1}{\kappa }\mathfrak{F}(F^4).
 \end{split}
 \end{equation}
On the other hand, we directly use the equation \eqref{l_k transport_equation} together with \eqref{compensator} to have that for $j=1,\dots,2N+1$
\begin{equation}\label{eta es12 2}
\begin{split}
   \ns{ \dt^j \eta}_{ L^2 H^{4N -2j+2 }}
 &\le     \kappa^2\ns{ \dt^{j-1}  \eta}_{ L^2 H^{4N -2j+4 }} +  \kappa^2\ns{\dt^{j-1} \Psi}_{ L^2 H^{4N -2j+2 }}+   \ns{\dt^{j-1} F^4}_{ L^2 H^{4N -2j+2 }}
 \\&\ls \kappa^2\ns{ \dt^{j-1} \eta}_{ L^2 H^{4N -2(j-1)+2}}+\ns{\eta_0}_{4N+1}+\mathfrak{F}_0(F^4)+ \mathfrak{F}(F^4) .
 \end{split}
\end{equation}
A simple induction on \eqref{eta es12 2}, together with \eqref{eta
es 0}, yields
\begin{equation}\label{eta es 2}
\begin{split}
  \sum_{j=1}^{2N+1}\ns{\dt^j \eta}_{L^2 H^{ 4N -2j +2} }
 & \ls\kappa^2\ns{  \eta}_{ L^2 H^{4N +2 }}+\ns{\eta_0}_{4N+1}+\mathfrak{F}_0(F^4)+ \mathfrak{F}(F^4)
 \\& \ls\ns{\eta_0}_{4N+1}+\mathfrak{F}_0(F^4)+ \mathfrak{F}(F^4).
\end{split}
\end{equation}
This and \eqref{eta es 0} yield \eqref{K eta es} by further applying Lemma \ref{l_sobolev_infinity} to bound the sum of the $L^\infty H^k$ norms in the definition of $\mathfrak{K} (\eta)$ by $\ns{\eta_0}_{4N+1}+\mathfrak{F}_0(F^4)$ plus the sum of the $L^2 H^{k}$ norms.
\end{proof}

\subsection{Preliminaries for the nonlinear $\kappa$-problem}

We will eventually use our linear theories for the problems \eqref{l_linear_forced} and \eqref{l_k transport_equation}  in order to solve the nonlinear problem \eqref{kgeometric}.  To do so, we define forcing terms $F^1,F^3$ to be used in the linear theory of \eqref{l_linear_forced} and $F^4$ to be used in the linear theory of \eqref{l_k transport_equation} that match the terms in \eqref{kgeometric}.  That is,  given $u,\eta$, we define
\begin{equation}\label{l_F_forcing_def}
 F^1(u,\eta)= \dt \bar{\eta} \tilde{b} K \p_3 u - u \cdot \naba u,\
  F^3( \eta)  =  \eta \n-\sigma H \n, \text{ and }F^4( u,\eta)  = u\cdot\n .
\end{equation}
where $\a,\n,K,H$ are determined as usual by $\eta$.

\subsubsection{Data estimates}\label{l_data_section}

In the construction of the initial data performed in the previous two subsections, when constructing $\dt^j u(0)$ for $j=1,\dots,2N$ it was assumed that $\dt^j \eta(0)$ for $j=0,\dotsc,2N$ and $\dt^j F^1(0)$, $\dt^j F^3(0)$ for $j=0,\dotsc,2N-1$  were all known; when constructing  $\dt^j \eta(0)$ for $j=1,\dots,2N$ it was assumed that $\dt^j F^4(0)$ for $j=0,\dotsc,2N-1$ were known. Since for the full nonlinear problem the functions $u,p,\eta$ are unknown and their evolutions are coupled to each other through $F^1(u,\eta)$, $F^3(\eta)$ and $F^4(u,\eta)$, we must revise the construction of the data to include this coupling, assuming only that $u_0$ and $\eta_0$ are given.  This will also reveal the compatibility conditions that must be satisfied by $u_0$ and $\eta_0$ in order to solve the nonlinear problem \eqref{kgeometric} (and \eqref{geometric}). To this end we define
\begin{equation}\label{l_EF0_def}
 \tilde{\mathcal{E}}_0 := \ns{u_0}_{4N} + \ns{\eta_0}_{4N+1}, \text{ and }  \mathfrak{L}_0 :=  \ns{\eta_0}_{4N-1/2}.
\end{equation}
For our estimates we must also introduce
\begin{equation}\label{e0up}
\mathfrak{E}_0(u,p) : = \sum_{j=0}^{2N} \ns{\dt^j u(0)}_{4N-2j} + \sum_{j=0}^{2N-1} \ns{\dt^j p(0)}_{4N-2j-1}
\end{equation}
and
\begin{equation}
\mathfrak{E}_0(\eta):= \sum_{j=0}^{2N}\ns{\dt^j \eta(0)}_{ 4N-2j +1 }.
\end{equation}

We will also need a more exact enumeration of the terms in
$\mathfrak{E}_0(u,p)$, $\mathfrak{E}_0(\eta)$,
$\mathfrak{F}_0(F^1,F^3)$, and $\mathfrak{F}_0(F^4)$ (as defined in
\eqref{l_Ffrak_def} and \eqref{l_Kfrak_u2n_def1}, respectively).
For $j=0,\dotsc,2N-1$ we define
\begin{equation}\label{l_F0_prec_def}
\mathfrak{F}_0^j(F^1(u,\eta),F^3(\eta)) := \sum_{\ell=0}^j \ns{\dt^\ell F^1(0)}_{4N-2\ell -2} + \ns{\dt^\ell F^3(0) }_{4N-2\ell -3/2},
 \end{equation}
 \begin{equation}
\mathfrak{F}^j_0(F^4(u,\eta)): =  \ns{\dt^j F^4(0)}_{ {4N-2j-1}},
 \end{equation}
and for $j=0,\dotsc,2N $ we define
\begin{equation}\label{l_E0_prec_eta_def}
 \mathfrak{E}_0^j(\eta) :=   \sum_{\ell=0}^j \ns{\dt^\ell \eta(0)}_{4N-2\ell +1}
\end{equation}
and
\begin{equation}\label{l_E0_prec_up_def}
 \mathfrak{E}_0^j(u,p) := \sum_{\ell=0}^j \ns{\dt^\ell u(0)}_{4N-2\ell} + \sum_{\ell=0}^{j-1} \ns{\dt^\ell p(0)}_{4N-2\ell-1}.
\end{equation}
Here in \eqref{l_E0_prec_up_def} for $j=0$ we mean $\mathfrak{E}_0^0(u,p) := \ns{u_0}_{4N}$.

We record the estimates of these quantities that are useful in dealing with the initial data.

\begin{lem}\label{l_full_iteration_estimates}

For $j=0,\dotsc,2N-1$, it holds that
\begin{equation}\label{l_fie_01}
 \mathfrak{F}_0^j(F^1(u,\eta),F^3(\eta)) \le P_j( \mathfrak{E}_0^{j+1}(\eta),\mathfrak{E}_0^j(u,p) ).
\end{equation}

For $j=1,\dotsc,2N-1$ let $F^{1,j}(0)$ and $F^{3,j}(0)$ be determined by \eqref{l_Fj_def} and \eqref{l_F_forcing_def}, using $\dt^\ell \eta(0)$, $\dt^\ell u(0)$, and $\dt^\ell p(0)$ for appropriate values of $\ell$.  Then
\begin{equation}\label{l_fie_02}
 \ns{F^{1,j}(0)}_{4N-2j-2} + \ns{F^{3,j}(0)}_{4N-2j-3/2} \le P_j( \mathfrak{E}_0^{j+1}(\eta),\mathfrak{E}_0^j(u,p) ).
\end{equation}

For $j=0,\dotsc,2N$ it holds that
\begin{equation}\label{l_fie_03}
 \ns{\dt^j u(0) - D_t^j u(0) }_{4N-2j} \le P_j( \mathfrak{E}_0^{j}(\eta),\mathfrak{E}_0^j(u,p) ).
\end{equation}

For $j=0,\dotsc,2N-1$ it holds that
\begin{equation}\label{l_fie_04}
\mathfrak{F}_0^j(F^4(u,\eta))\le P_j( \mathfrak{E}_0^{j}(\eta),\mathfrak{E}_0^j(u,p) ).
\end{equation}
Here $P_j(\cdot,\cdot)$ a polynomial so that $P_j(0,0)=0$.
\end{lem}
\begin{proof}
These bounds may be derived by arguing as in the proof of Lemma \ref{l_iteration_estimates_1}.  As such, we again omit the details.
\end{proof}

This lemma allows us to modify the construction presented in the previous two subsections to construct all of the initial data $\dt^j u(0)$, $\dt^j \eta(0)$ for $j=0,\dotsc,2N$ and $\dt^j p(0)$ for $j=0,\dotsc,2N-1$.  Along the way we will also derive estimates of $\mathfrak{E}_0(u,p) + \mathfrak{E}_0(\eta)$ in terms of $\tilde{\mathcal{E}}_0$ and determine the compatibility conditions for $u_0,\eta_0$ necessary for existence of solutions to \eqref{kgeometric} (and \eqref{geometric}).

We assume that $u_0, \eta_0$ satisfy $\tilde{\mathcal{E}}_0 <\infty$ and that $ \mathfrak{L}_0\le 1$ is sufficiently small for the hypothesis of Proposition  \ref{l_a_poisson_regularity} to hold when $k=4N$.  As before, we will iteratively construct the initial data, but this time we will use the estimates in Lemma \ref{l_full_iteration_estimates}.  Define $\dt \eta(0): =F^4(0)=u_0 \cdot \n_0$, where $u_0 \in H^{4N-1/2}(\Sigma)$ when traced onto $\Sigma$.  Then
$\ns{\dt \eta(0)}_{4N-1/2} \ls \tilde{\mathcal{E}}_0(1+\tilde{\mathcal{E}}_0)$, and hence in particular that $\mathfrak{E}_0^0(u,p) + \mathfrak{E}_0^1(\eta) \le P (\tilde{\mathcal{E}}_0)$.  We may use this bound in \eqref{l_fie_01} with $j=0$ to find
\begin{equation}\label{l_dc_0}
\mathfrak{F}_0^0(F^1(u,\eta),F^3( \eta))  \le P_0(\mathfrak{E}_0^0(u,p) , \mathfrak{E}_0^1(\eta) )\le  P(\tilde{\mathcal{E}}_0).
\end{equation}

Suppose now that $j \in [0,2N-2]$ and that $\dt^\ell u(0)$ are known for $\ell=0,\dotsc,j$, $\dt^\ell \eta(0)$ are known for $\ell = 0,\dotsc,j+1$, and $\dt^\ell p(0)$ are known for $\ell = 0,\dotsc,j-1$ (with the understanding that nothing is known of $p(0)$ when $j=0$), and that
\begin{equation}\label{l_dc_1}
 \mathfrak{E}_0^{j}(u,p) + \mathfrak{E}_0^{j+1}(\eta) + \mathfrak{F}_0^{j}(F^1(u,\eta),F^3(u,\eta))  \le  P (\tilde{\mathcal{E}}_0).
\end{equation}
According to the estimates \eqref{l_fie_02} and \eqref{l_fie_03}, we then know that
\begin{equation}\label{l_dc_2}
 \ns{F^{1,j}(0)}_{4N-2j-2} + \ns{F^{3,j}(0)}_{4N-2j-3/2} + \ns{D_t^j u(0)}_{4N-2j} \le   P (\tilde{\mathcal{E}}_0).
\end{equation}
By virtue of estimates \eqref{l_ie2_01} and \eqref{l_dc_1}, we know that
\begin{equation}\label{l_dc_2_2}
\begin{split}
& \ns{\mathfrak{f}^1(F^{1,j}(0),D_t^j u(0))}_{4N-2j-3} + \ns{\mathfrak{f}^2(F^{3,j}(0),D_t^j u(0))}_{4N-2j-3/2} \\
&\quad+ \ns{\mathfrak{f}^3(F^{1,j}(0),D_t^j u(0))}_{4N-2j-5/2} \le  P (\tilde{\mathcal{E}}_0).
\end{split}
\end{equation}
This allows us to define $\dt^j p (0)$ as the solution to \eqref{l_linear_elliptic_p} with  $f^1, f^2, f^3$ given by $\mathfrak{f}^1$, $\mathfrak{f}^2$, $\mathfrak{f}^3$.  Then Proposition \ref{l_a_poisson_regularity} with $k = 4N$ and $r=4N-2j-1 < k$ implies that
\begin{equation}\label{l_dc_3}
 \ns{\dt^j p(0)}_{4N-2j-1} \le    P (\tilde{\mathcal{E}}_0).
\end{equation}
Now the estimates \eqref{l_ie2_02}, \eqref{l_dc_1}, and \eqref{l_dc_2} allow us to define
\begin{equation}
D_t^{j+1} u(0) := \mathfrak{G}^0(F^{1,j}(0), D_t^j u(0),\dt^j p(0)) \in H^{4N-2j-2},
\end{equation}
and owing to \eqref{l_fie_03}, we have the estimate
\begin{equation}\label{l_dc_4}
\ns{\dt^{j+1} u(0)}_{4N-2(j+1)} \le  P (\tilde{\mathcal{E}}_0).
\end{equation}
Now we define
\begin{equation}
 \partial_t^{j+2}\eta(0):=\dt^{j+1}F^4(0)=  \sum_{\ell=0}^{j+1} C_{j+1}^\ell \dt^{\ell} u(0)\cdot \dt^{j+1-\ell} \n(0) .
\end{equation}
The estimate \eqref{l_fie_04} together with \eqref{l_dc_1} and \eqref{l_dc_4} then imply that
\begin{equation}\label{l_dc_5}
\ns{\dt^{j+2} \eta(0)}_{4N-2(j+2)+1} \le P(\tilde{\mathcal{E}}_0).
\end{equation}
We may combine \eqref{l_dc_1} with \eqref{l_dc_3}--\eqref{l_dc_5} to deduce that
\begin{equation}\label{l_dc_6}
  \mathfrak{E}_0^{j+1}(u,p) + \mathfrak{E}_0^{j+2}(\eta)   \le P(\tilde{\mathcal{E}}_0),
\end{equation}
but then \eqref{l_fie_01} implies that $\mathfrak{F}_0^{j+1}(F^1(u,\eta),F^3(u,\eta))  \le P(\tilde{\mathcal{E}}_0)$ as well, and we deduce that the bound \eqref{l_dc_1} also holds with $j$ replaced by $j+1$.

Using the above analysis, we may iterate from $j=0,\dotsc,2N-2$ to deduce that
\begin{equation}\label{l_dc_7}
  \mathfrak{E}_0^{2N-1}(u,p) + \mathfrak{E}_0^{2N}(\eta) + \mathfrak{F}_0^{2N-1}(F^1(u,\eta),F^3(u,\eta))  \le P(\tilde{\mathcal{E}}_0).
\end{equation}
After this iteration, it remains only to define $\dt^{2N-1} p(0)$ and  $D_t^{2N} u(0)$.  In order to do this, we must first impose the compatibility conditions on $u_0$ and $\eta_0$.  These are the same as in \eqref{l_comp_cond}, but because now the temporal derivatives of $\eta$ have been constructed as well, we restate them in a slightly different way.  Let $\dt^j u(0)$, $F^{1,j}(0)$, $F^{3,j}(0)$ for $j=0,\dots,2N-1$, $\dt^j \eta(0)$ for $j=0,\dotsc,2N$, and $\dt^j p(0)$ for $j=0,\dotsc,2N-2$ be constructed in terms of $\eta_0, u_0$ as above.  Let $\Pi_0$ be the projection defined in terms of $\eta_0$ as in \eqref{l_Pi0_def} and $D_t$ be the operator defined by \eqref{l_Dt_def}.  We say that $u_0,\eta_0$ satisfy the $(2N)^{th}$ order compatibility conditions if
\begin{equation}\label{l_comp_cond_2N}
\begin{cases}
 \diverge_{\a_0} (D_t^j u(0)) =0 & \text{in }\Omega \\
 D_t^j u(0) =0 & \text{on } \Sigma_b \\
 \Pi_0 (F^{3,j}(0) + \sg_{\a_0} D_t^j u(0) \n_0 ) = 0 & \text{on } \Sigma
\end{cases}
\end{equation}
for $j=0,\dotsc,2N-1$.  Note that if $u_0,\eta_0$ satisfy \eqref{l_comp_cond_2N}, then the $j^{th}$ order compatibility condition \eqref{l_comp_cond} is satisfied for $j=0,\dotsc,2N-1$.

Now we define $\dt^{2N-1} p(0)$ and  $D_t^{2N} u(0)$. We use the compatibility conditions \eqref{l_comp_cond_2N} and argue as above and in the derivation of \eqref{l_ie2_03} in Lemma \ref{l_iteration_estimates_2} to estimate
\begin{equation}\label{l_dc_10}
\ns{\mathfrak{f}^2(F^{3,2N-1}(0),D_t^{2N-1} u(0))}_{1/2}
+ \ns{\mathfrak{f}^3(F^{1,2N-1}(0),D_t^{2N-1} u(0))}_{-1/2} \le P(\tilde{\mathcal{E}}_0)
\end{equation}
and
\begin{equation}\label{l_dc_11}
 \ns{F^{1,2N-1}(0)}_{0} + \ns{\diverge_{\a_0}(R_0 D_t^{2N-1} u(0))  }_{0} \le P(\tilde{\mathcal{E}}_0).
\end{equation}
We then  define $\dt^{2N-1} p(0) \in H^1$ as a weak solution to \eqref{l_linear_elliptic_p} in the sense of \eqref{l_a_poisson_div} with this choice of $f^2=\mathfrak{f}^2$,  $f^3= \mathfrak{f}^3,$ $g_0=-\diverge_{\a_0}(R_0 D_t^{2N-1} u(0))$, and $G=-F^{1,2N-1}(0)$.  The estimate \eqref{l_a_poisson_weak_estimate}, when combined with \eqref{l_dc_10}--\eqref{l_dc_11}, allows us to deduce that
\begin{equation}\label{l_dc_8}
 \ns{\dt^{2N-1} p(0)}_{1} \le   P(\tilde{\mathcal{E}}_0).
\end{equation}
Then by \eqref{l_ie2_02} we set $D_t^{2N} u(0) = \mathfrak{G}^0(F^{1,2N-1}(0), D_t^{2N-1} u(0),\dt^{2N-1} p(0)) \in H^0$.  In fact, the construction of $\dt^{2N-1} p(0)$ guarantees that $D_t^{2N} u(0) \in \y(0)$.  Arguing as before, we  also have
\begin{equation}\label{l_dc_9}
\ns{\dt^{2N} u(0)}_{0} \le P(\tilde{\mathcal{E}}_0).
\end{equation}
This completes the construction of the initial data, but we will record a form of the estimates \eqref{l_dc_7}, \eqref{l_dc_8}--\eqref{l_dc_9} in the following proposition.

\begin{prop}\label{l_data_norm_comparison}
Suppose that $u_0, \eta_0$ satisfy $\tilde{\mathcal{E}}_0<\infty$ and that $ \mathfrak{L}_0\le \varepsilon_0$ is sufficiently small for the hypothesis of Proposition  \ref{l_a_poisson_regularity} to hold when $k=4N$.  Let $\dt^j u(0)$, $\dt^j \eta(0)$ for $j=0,\dotsc,2N$ and $\dt^j p(0)$ for $j=0,\dotsc,2N-1$ be given as above.  Then
\begin{equation}\label{l_dnc_0}
\tilde{\mathcal{E}}_0\le \mathfrak{E}_0(u,p) + \mathfrak{E}_0(\eta) \le  P (\tilde{\mathcal{E}}_0).
\end{equation}
\end{prop}
\begin{proof}
The proposition directly follows by summing \eqref{l_dc_7} and \eqref{l_dc_8}--\eqref{l_dc_9}.
\end{proof}
\begin{remark}\label{iii}
Along the arguments above, it is easy to see that $ \eta_0\in H^{4N+1/2}(\Sigma)$ replacing $ \eta_0\in H^{4N+1}(\Sigma)$ is still sufficient to guarantee the construction and that the constructed initial data can enjoy the following improvement, one may also refer to Section 5.2 of \cite{GT_lwp},
\begin{equation}\label{sflja}
\mathfrak{E}_0(u,p)+\ns{\eta(0)}_{4N+1/2} +   \sum_{j=1}^{2N+1} \ns{\dt^j \eta(0)}_{4N-2j+3/2}\le P \left(\ns{u_0}_{4N}+\ns{\eta_0}_{4N+1/2}\right).
 \end{equation}
 Note that we have also included the case $j=2N+1$ in \eqref{sflja}.
Indeed, using the geometric fact: $\dt(u\cdot \n)= D_t u\cdot \n$ on
$\Sigma$, by Lemma \ref{l_boundary_dual_estimate} and since $
D_t^{2N}u(0)\in \mathcal{Y}(0)$, we have
\begin{equation}
\begin{split}
 \ns{\dt^{2N+1}\eta(0)}_{-1/2}&=\snormspace{D_t^{2N}u(0) \cdot \n_0}{-1/2}{\Sigma} \ls \hn{D_t^{2N}u(0)}{0}
 \\&\le P\left(\ns{u_0}_{4N}+\ns{\eta_0}_{4N+1/2}\right).
 \end{split}
\end{equation}
\end{remark}

\subsubsection{Construction of compensator function}\label{sec compensator}

Let $\dt^j \eta(0)$ for $j=0,\dotsc,2N$ be determined as in Section \ref{l_data_section}.
Now we define $\Psi$ to be the extension produced by Lemma \ref{2_sobolev_extension} that achieves the initial data $-\Delta_\ast \dt^j\eta(0)$ for $j=0,\dots,2N$. Then we essentially add nothing on the fourth equation of \eqref{kgeometric} at the initial time. Moreover, by Lemma \ref{2_sobolev_extension} and \eqref{l_dnc_0}, we have
\begin{equation}
 \sum_{j=0}^{2N} \ns{\dt^j \Psi}_{L^2 H^{4N-2j }}   +  \sum_{j=0}^{2N}\ns{\dt^j \Psi}_{L^\infty H^{4N-2j -1}}
 \ls \sum_{j=0}^{2N}\ns{ \dt^j\eta(0) }_{ {4N-2j +1}}
  =  \mathfrak{E}_0(\eta)\le  P (\tilde{\mathcal{E}}_0).
\end{equation}
Thanks to this compensator, the initial data $(u_0,\eta_0)$ and those initial data actually constructed for the problem \eqref{geometric} can serve as the initial data for the problem \eqref{kgeometric}.

\subsubsection{Forcing estimates}
We will now estimate various norms of $F^1(u,\eta)$, $F^3(\eta)$ and
$F^4(u,\eta)$ that are contained in $\mathfrak{F}(F^1,F^3)$ and
$\mathfrak{F}(F^4)$ (as defined in \eqref{l_Ffrak_def} and
\eqref{l_Kfrak_u2n_def1}, respectively) in terms of the certain
norms of $u$ and $\eta$. It is a very key point to notice that the
$\eta$ and $u$ appearing in $\mathfrak{K}(\eta)$ and
$\mathfrak{K}(u,p)$ enjoy the higher regularities than that required
in estimating the norms of the $F^i$ terms. To state the estimates,
we define
\begin{equation}\label{l_Kfrak_u2n_def}
 \mathfrak{F}_{2N}(u) = \sum_{j=0}^{2N} \ns{\dt^j u}_{L^2 H^{4N-2j +3/4}}
+ \sum_{j=0}^{2N-1}\ns{\dt^j u}_{L^\infty H^{4N-2j-1/4}}
\end{equation}
and
\begin{equation}\label{l_Kfrak_eta2n_def}
 \mathfrak{F}_{2N} (\eta)    := \sum_{j=0}^{2N} \ns{\dt^j \eta}_{L^2 H^{4N-2j+1}}
+ \sum_{j=0}^{2N-1} \ns{\dt^j \eta}_{L^\infty H^{4N-2j}} .
\end{equation}
We shall recall the definitions of $\mathfrak{K}_{2N}(u)$, $\mathfrak{K}_{2N}(\eta)$, $\mathfrak{K}(u,p)$, $\mathfrak{F}(\eta)$, and $\mathfrak{K}(\eta)$ from \eqref{kk11}, \eqref{kk12}, \eqref{l_Kfrak_def}, \eqref{l_DEfrak_def} and \eqref{l_Ketafrak_def}, respectively.
Note that $\mathfrak{F}_{2N}(u)\le \mathfrak{K}_{2N}(u)\le \mathfrak{K}(u,0)$ and $\mathfrak{F}_{2N}(\eta)\le \mathfrak{K}_{2N}(\eta)\le \mathfrak{K}(\eta), \mathfrak{F}(\eta)\le \mathfrak{K}(\eta)$.
\begin{lem}\label{l_Ffrak_bound}
Suppose that $  \mathfrak{K} (\eta)<\infty$ and $\mathfrak{F}_{2N}(u)<\infty$. Then
 \begin{equation}\label{f bound}
 \mathfrak{F}(F^1(u,\eta),F^3(\eta)) \le  P(1+ \mathfrak{F}(\eta))\left(   \mathfrak{K} (\eta) +(\mathfrak{F}_{2N}(u))^2\right)
\end{equation}
and
 \begin{equation}\label{f2 bound}
 \mathfrak{F}(F^4(u,\eta) ) \le  P(1+ \mathfrak{F}_{2N}(\eta))\mathfrak{F}_{2N}(u).
\end{equation}
\end{lem}
\begin{proof}
The proof is similar to that of Lemma \ref{l_iteration_estimates_1}. We only estimate the terms $\ns{\dt^{2N} F^1}_{L^2 (\H1)^\ast}$ and $\ns{F^3}_{L^2 H^{4N-1/2}}$ in the definition of $ \mathfrak{F}(F^1,F^3)$ since they are a bit more involved and need more care, and the other estimates follow similarly and simpler.

We note that the highest derivatives appearing in $\dt^{2N} F^1(u,\eta)$ are $\dt^{2N+1} \bar{\eta}$ and $\dt^{2N }\nab u$. The $\dt^{2N+1} \bar{\eta}$ has been already controlled in \eqref{obser1}. For the $\dt^{2N }\nab u$, we observe that for $\varphi\in \H1$,
\begin{equation}
\br{ \dt^{2N}\nab u,\varphi}_{*}=\ip{\dt^{2N}  u}{\varphi}_{0,\Sigma }-\ip{ \dt^{2N}u}{\nab\varphi}_{0,\Omega }.
\end{equation}
Here $\br{\cdot,\cdot}_{*}$ is the dual paring in $\H1$, $ \ip{\cdot}{\cdot}_{0,\Sigma}$ and $ \ip{\cdot}{\cdot}_{0,\Omega}$ is the $L^2$ inner product in $\Sigma$ and $\Omega$, respectively.
By the duality and the trace embedding that $H^{3/4}(\Omega)\hookrightarrow H^{1/4}(\Sigma)\subset H^{-1/2}(\Sigma)$, we obtain
\begin{equation}\label{2nu}
\ns{\dt^{2N}\nab u}_{L^2 (\H1)^*}\ls \ns{\dt^{2N} u }_{L^2H^{3/4}}+\ns{\dt^{2N}u}_{L^2H^0}\ls  \ns{\dt^{2N}u}_{L^2H^{3/4}}\le \mathfrak{F}_{2N}(u).
\end{equation}
With \eqref{obser1} and \eqref{2nu} in hand, we may conclude as in Lemma \ref{l_iteration_estimates_1} that
 \begin{equation}
 \begin{split}
\ns{\dt^{2N} F^1}_{L^2 (\H1)^\ast}& \le  P(1+ \mathfrak{F}(\eta))\left(   \mathfrak{F} (\eta) \mathfrak{F}_{2N}(u)+(\mathfrak{F}_{2N}(u))^2\right)
\\&\le  P(1+ \mathfrak{F}(\eta))\left(   \mathfrak{F} (\eta) +(\mathfrak{F}_{2N}(u))^2\right).
\end{split}
\end{equation}

On the other hand, note that the highest derivative appearing in $F^3(\eta)$ is $\nab_\ast^2\eta$, and we have
\begin{equation}\label{0909}
 \ns{\nab_\ast^2\eta}_{L^2 H^{4N-1/2}}\ls  \ns{\eta}_{L^2 H^{4N+3/2}}\le \mathfrak{K} (\eta).
\end{equation}
We may then conclude as in Lemma \ref{l_iteration_estimates_1} that
 \begin{equation}
 \ns{F^3}_{L^2 H^{4N-1/2}} \le  P(1+ \mathfrak{F}(\eta))  \mathfrak{K} (\eta).
\end{equation}

Consequently, estimating the other terms in the similar way, we deduce \eqref{f bound}--\eqref{f2 bound}.
\end{proof}
\begin{remark}\label{weaker}
Note that at this moment there is no power of $T$ in front of the estimates \eqref{f bound}--\eqref{f2 bound}. The key point is that in the iteration argument in the next subsection we can get better estimates of $ \eta $ and $u$, say, $\mathfrak{K}(\eta)$ and $\mathfrak{K}(u,p)$.
Eventually, this will make the appearance of a power of $T$ in front of some estimates in the iteration argument, which allows us to achieve a local well-posedness without requiring all the norms of initial data be small.
\end{remark}

\subsection{Local well-posedness of the nonlinear $\kappa$-problem}

In order to construct the solution to \eqref{kgeometric}, we will pass to the limit in a sequence of approximate solutions. In this subsection we will explicitly enumerate the various generic polynomials appearing in estimates as $ P_1,\dots, P_{10}$ so that they can be referred, but they will be allowed to depend on $\kappa$. Before proceeding further, we recall from the previous subsection that
\begin{equation}\label{l_dnc_01}
\mathfrak{E}_0(u,p) + \mathfrak{E}_0(\eta) +\mathfrak{F}_0(F^1(u ,\eta),F^3( \eta ))+ \mathfrak{F}_0(F^4(u ,\eta))\le  P_1(\tilde{\mathcal{E}}_0).
\end{equation}

We now construct the sequence of approximate solutions as follows.
First, we extend the initial data $\dt^j u(0) \in H^{4N-2j}(\Omega) $ to the time-dependent functions $u^0 $ so that $\dt^j u^0(0) = \dt^j u(0)$ and extend the initial data $\dt^j \eta(0) \in H^{4N-2j+1}(\Sigma) $ to the time-dependent functions $\eta^0 $ so that $\dt^j \eta^0(0) = \dt^j \eta(0)$.  We do so by applying lemmas \ref{l_sobolev_extension}--\ref{2_sobolev_extension} respectively, and we may combine the estimates of $\mathfrak{K}_{2N}(u^0)$ and $\mathfrak{K}_{2N}(\eta^0)$ provided by lemmas \ref{l_sobolev_extension}--\ref{2_sobolev_extension} with \eqref{l_dnc_01} to bound
\begin{equation}\label{l_it_3}
\mathfrak{K}_{2N}(u^0)+\mathfrak{K}_{2N}(\eta^0) \le
P_2(\tilde{\mathcal{E}}_0).
\end{equation}
This is our start point for constructing $(u^{m},p^{m},\eta^{m})$ for $m\ge 1$.

Now assuming that $(u^m,\eta^m)$ are known and satisfy certain estimates, we then construct $(u^{m+1},p^{m+1},\eta^{m+1})$ as the solution to
\begin{equation}\label{mkgeometric}
 \begin{cases}
 \dt u^{m+1} -\Delta_{\a^{m+1}} u^{m+1} + \nab_{\a^{m+1}} p^{m+1}  = F^1(u^m,\eta^{m+1}) & \text{in } \Omega \\
 \diverge_{\a^{m+1}} u^{m+1} = 0 & \text{in }\Omega \\
 S_{\a^m}(p^{m+1},u^{m+1}) \n^{m+1} =F^3(\eta^{m+1}) & \text{on } \Sigma \\
   \dt \eta^{m+1}-\kappa\Delta_\ast\eta^{m+1} = \kappa \Psi+  F^4(u^m,\eta^{m})  &\text{on } \Sigma\\
 u^{m+1} = 0 & \text{on } \Sigma_b \\
 (u^{m+1} , \eta^{m+1})\mid_{t=0} = (u_0,\eta_0).
 \end{cases}
\end{equation}
Note that our construction is in its special way; we first use $(u^m,\eta^m)$ to construct $\eta^{m+1}$, and then we use $(u^m,\eta^{m+1})$ to construct $(u^{m+1},p^{m+1})$. We shall derive the uniform estimates for the sequence $\{(u^m,p^m,\eta^m)\}_{m=1}^\infty$, which are recorded in the following theorem.

\begin{thm}\label{l_iteration}
Assume that the initial data are given as in Section \ref{l_data_section} and satisfy the $(2N)^{th}$  compatibility conditions \eqref{l_comp_cond_2N}.  Suppose that $ \mathfrak{L}_0 \le \varepsilon_0/2$, where $\varepsilon_0$ is the smaller one from Theorem \ref{l_linear_wp} and Proposition \ref{l_data_norm_comparison}. There exist a $0<M_0<\infty$ and a $0< T_0 <1$ depending on $\tilde{\mathcal{E}}_0$ and $\kappa$ such that if $T\le T_0$, then for each $m\ge 1$ we have the estimates
\begin{equation}\label{l_it_03}
   \mathfrak{K}(u^m,p^m)+\mathfrak{K} (\eta^m)   \le M_0, \text{ and }
  \mathfrak{L}(\eta^m)  \le \varepsilon_0.
\end{equation}
\end{thm}
\begin{proof}
We claim that there exist a (large) $0<M _0<\infty$ and a (small) $ 0<{T}_0 <1$ so that if $T\le T_0$, then the following hold: if $(u^m,\eta^m)$ for $m\ge 0$ are known and satisfy the estimates
\begin{equation}\label{claim0}
  \mathfrak{K}_{2N}(u^m) +\mathfrak{K}_{2N}(\eta^m)  \le  M _0,
\end{equation}
then there exists a unique triple $(u^{m+1},p^{m+1},\eta^{m+1})$ that solves \eqref{mkgeometric} and obeys the estimates
\begin{equation} \label{claim123}
 \mathfrak{K}(u^{m+1},p^{m+1})+ \mathfrak{K} (\eta^{m+1})   \le M_0,  \text{ and }
  \mathfrak{L}(\eta^{m+1})  \le \varepsilon_0.
\end{equation}
We remark that in the following determination $M _0$ depends on
$\tilde{\mathcal{E}}_0$ and ${T}_0$ depends on $M_0$; so these two
constants are determined by the initial data (and $\kappa$).

First, if we take $M _0\ge  P_2(\tilde{\mathcal{E}}_0)$, then \eqref{claim0} holds for $m=0$. Now assume that \eqref{claim0} holds for $m\ge 0$, we prove that \eqref{claim123} holds for $m+1$ for appropriately chosen $M_0$ and $T_0$. We first use $(u^m,\eta^m)$ to construct the $\eta^{m+1}$ via Theorem \ref{l_transport_theorem}. Moreover, by the estimate \eqref{K eta es} of Theorem \ref{l_transport_theorem} together with the estimate \eqref{f2 bound} of Lemma \ref{l_Ffrak_bound} and \eqref{l_dnc_01}, we have
\begin{equation} \label{eses1}
\begin{split}
 \mathfrak{K} (\eta^{m+1})&\le C_\kappa\left( \ns{\eta_0}_{4N+1}+\mathfrak{F}_0(F^4)+     \mathfrak{F}(F^4(u^m,\eta^m))\right)
 \\&\le     P_3(\tilde{\mathcal{E}}_0)+  P_3(1+ \mathfrak{F}_{2N}(\eta^m))\mathfrak{F}_{2N}(u^m) .
 \end{split}
 \end{equation}
We need an appropriate estimate of $\mathfrak{F}_{2N}(\eta^m)$ and $\mathfrak{F}_{2N}(u^m)$, and the key point is to get a power of $T$ in the following estimate. First, Lemma \ref{l_sobolev_infinity} allows us to bound, by \eqref{l_dnc_01} and \eqref{claim0},
 \begin{equation} \label{eses2}
 \begin{split}
 \mathfrak{F}_{2N}(\eta^m)&\le \sum_{j=0}^{2N-1} \ns{\dt^j \eta(0)}_{ {4N-2j}} +2\sum_{j=0}^{2N} \ns{\dt^j \eta^m}_{L^2 H^{4N-2j+1}}
\\& \le  P_1(\tilde{\mathcal{E}}_0)+2T\sum_{j=0}^{2N} \ns{\dt^j \eta^m}_{L^\infty H^{4N-2j+1}}
\\& \le  P_1(\tilde{\mathcal{E}}_0)+ 2 T\mathfrak{K}_{2N}(\eta^m)\le  P_1(\tilde{\mathcal{E}}_0)+2TM _0.
\end{split}
 \end{equation}
On the other hand, using the usual Sobolev interpolation together with Lemma \ref{l_sobolev_infinity}, \eqref{l_dnc_01} and \eqref{claim0} again, we obtain
 \begin{equation}\label{eses3}
 \begin{split}
 \mathfrak{F}_{2N}(u^m) &\le \sum_{j=0}^{2N-1}\ns{\dt^j u(0)}_{ {4N-2j-1/4}}+ 2\sum_{j=0}^{2N} \ns{\dt^j u^m}_{L^2 H^{4N-2j +3/4}}
 \\& \le  P_1 (\tilde{\mathcal{E}}_0)+2 \sum_{j=0}^{2N}T^{1/4} \norm{\dt^j u^m}_{L^\infty H^{4N-2j}}^{1/2}  \norm{\dt^j u^m}_{L^2 H^{4N-2j+1}}^{3/2}
\\& \le P_1(\tilde{\mathcal{E}}_0)+ C T^{1/4}\mathfrak{K}_{2N}(u^m) \le  P_1(\tilde{\mathcal{E}}_0)+C T^{1/4}M _0.
\end{split}
 \end{equation}
If we take $T\le  \min\{\tilde{\mathcal{E}}_0/M _0,(\tilde{\mathcal{E}}_0/M _0)^4\}$, then we deduce from \eqref{eses2}--\eqref{eses3} that
 \begin{equation}\label{eses4}
 \mathfrak{F}_{2N}(\eta^m)+\mathfrak{F}_{2N}(u^m) \ls  P_1(\tilde{\mathcal{E}}_0)+\tilde{\mathcal{E}}_0\le  P_4(\tilde{\mathcal{E}}_0).
 \end{equation}
Plugging the estimate \eqref{eses4} into \eqref{eses1}, we obtain
 \begin{equation}\label{es01}
 \mathfrak{K} (\eta^{m+1})
 \le     P_3 (\tilde{\mathcal{E}}_0)+  P_3(1+  P_4(\tilde{\mathcal{E}}_0) )  P_4(\tilde{\mathcal{E}}_0)  \le  P_5(\tilde{\mathcal{E}}_0) .
 \end{equation}
This proves the bound of $\mathfrak{K} (\eta^{m+1})$. Now, by again Lemma \ref{l_sobolev_infinity} and \eqref{es01}, we have
\begin{equation}
\begin{split}
 \mathfrak{L}(\eta^{m+1})&\le   \mathfrak{L}_0 + \ns{\eta^{m+1} }_{L^2H^{4N}}+\ns{\dt\eta^{m+1} }_{L^2H^{4N-1}}
\\& \le    \frac{\varepsilon_0}{2}   +T \ns{\eta^{m+1} }_{L^\infty H^{4N}}+T \ns{\dt\eta^{m+1} }_{L^\infty H^{4N-1}}
\\& \le    \frac{\varepsilon_0}{2}+ {T \mathfrak{K} (\eta^{m+1})}  \le     \frac{\varepsilon_0}{2}  + \frac{T P_5(\tilde{\mathcal{E}}_0)}{2}
\end{split}
\end{equation}
If we restrict that $
T\le  \varepsilon_0/ P_5(\tilde{\mathcal{E}}_0),$
then we have
\begin{equation} \label{smallness}
 \mathfrak{L}(\eta^{m+1}) \le \varepsilon_0.
\end{equation}

Now the bounds \eqref{es01} and \eqref{claim0} and the smallness \eqref{smallness} allow us to use $(u^m,\eta^{m+1})$ to construct the $(u^{m+1},p^{m+1})$ via Theorem \ref{l_linear_wp}. Moreover, by the estimates \eqref{l_lwp_02} of Theorem \ref{l_linear_wp}, \eqref{l_dnc_01}, the estimate \eqref{f bound} of Lemma \ref{l_Ffrak_bound}, \eqref{es01} and \eqref{claim0}, we have
\begin{equation}
\begin{split}
&\mathfrak{K}(u^{m+1},p^{m+1}) \le    P_6(1+\mathfrak{F}_0 (\eta^{m+1}) + \mathfrak{F}(\eta^{m+1})) \exp\left( P_6(1+ \mathfrak{F}(\eta^{m+1})) T \right)
\\&\qquad\qquad\qquad\quad\ \times \left( \ns{u_0}_{4N} + \mathfrak{F}_0(F^1 ,F^3 ) + \mathfrak{F}(F^1(u^m,\eta^{m+1}),F^3( \eta^{m+1})) \right)
\\& \quad\le  P_6(1+ P_1(\tilde{\mathcal{E}}_0) + P_5(\tilde{\mathcal{E}}_0)) \exp\left( P_6(1+  P_5(\tilde{\mathcal{E}}_0)) T \right)
\\&\qquad\times
  \left(
P_1(\tilde{\mathcal{E}}_0)+ P_7(1+ \mathfrak{F}(\eta^{m+1}))\left(   \mathfrak{K} (\eta^{m+1}) +(\mathfrak{F}_{2N}(u^{m}))^2\right)\right)
\\&\quad \le  P_8(1+ \tilde{\mathcal{E}}_0 ) \exp\left( P_8(1+  \mathcal{E}_0 ) T \right)
  \left(
P_1(\tilde{\mathcal{E}}_0)+ P_7(1+  P_5(\tilde{\mathcal{E}}_0))\left(    P_5(\tilde{\mathcal{E}}_0) +( P_4(\tilde{\mathcal{E}}_0))^2\right)\right)
\\&\quad\le    P_9(\tilde{\mathcal{E}}_0)  \exp\left( P_8(1+ \tilde{\mathcal{E}}_0)T \right)
.\end{split}
\end{equation}
If we restrict that $T \le 1/ P_8(1+ \tilde{\mathcal{E}}_0)$, then
\begin{equation}\label{upup}
\mathfrak{K}(u^{m+1},p^{m+1}) \le    P_{10}(\tilde{\mathcal{E}}_0)
.
\end{equation}
W can now fix $ M_0:= P_5(\tilde{\mathcal{E}}_0)+ P_{10}(\tilde{\mathcal{E}}_0)$ and then take
\begin{equation}
T_0:=\min\left\{\frac{\tilde{\mathcal{E}}_0}{P_5(\tilde{\mathcal{E}}_0
)+ P_{10}(\tilde{\mathcal{E}}_0 )}
,\left(\frac{\tilde{\mathcal{E}}_0}{P_5(\tilde{\mathcal{E}}_0 )+
P_{10}(\tilde{\mathcal{E}}_0 )}\right)^4,
\frac{\varepsilon_0}{P_5(\tilde{\mathcal{E}}_0)} ,\frac{1}{P_8(1+
\tilde{\mathcal{E}}_0 )} \right\}.
\end{equation}
Hence if $0<T\le T_0$, then we deduce the estimates \eqref{claim123}
from \eqref{es01}, \eqref{smallness} and \eqref{upup}. We thus
conclude the claim and complete the proof of the theorem by the
induction.
\end{proof}

Now we can produce a unique local strong solution to the nonlinear $\kappa$-problem \eqref{kgeometric} with $\kappa>0$, which is the main result of this section.
\begin{thm}\label{l_knwp}
Let $\kappa>0$. Assume that $u_0\in H^{4N}(\Omega),\eta_0\in H^{4N+1}(\Sigma)$ with $\mathfrak{L}_0\le \varepsilon_0/2$ for  the universal $\varepsilon_0>0$ from Theorem \ref{l_iteration}, and that the initial data $\dt^j u(0)$, etc are as constructed in Section \ref{l_data_section} and satisfy the $(2N)^{th}$  compatibility conditions \eqref{l_comp_cond_2N}. There exists a $0< T_0^\kappa\le 1$ such that there exists a unique solution triple $(u,p,\eta)$ to the  problem \eqref{kgeometric} on the time interval $[0,T_0^\kappa]$ that  achieves the initial data  and  satisfies
\begin{equation}
 \mathfrak{K}(u,p)+\mathfrak{K}(\eta)  \le P_\kappa\left(\ns{u_0}_{4N}+\ns{\eta_0}_{4N+1}\right) \text{ and } \mathfrak{L} (\eta) \le \varepsilon_0.
\end{equation}
\end{thm}

\begin{proof}
The uniform estimates of $\mathfrak{K}(u^{m+1},p^{m+1})+
\mathfrak{K} (\eta^{m+1})$ in \eqref{l_it_03} allows us to take weak
and weak-$\ast$ limits, up to extraction of a subsequence, so that
the sequence $(u^m,p^m,\eta^m)$ converges to a limit $(u,p,\eta)$ in
the weak sense of the norms in the definition of
$\mathfrak{K}(u,p)+\mathfrak{K}(\eta)$. If we recover the dependence
of $M_0$ and $T_0$ on the initial data and $\kappa$ as $M_0=P_\kappa
( \tilde{\mathcal{E}}_0 )=P_\kappa (\ns{u_0}_{4N}+\ns{\eta_0}_{4N+1}
)$ and $T_0=T_0^\kappa$, then we have from \eqref{l_it_03} that for
$0<T\le T_0^\kappa$, according to the weak and weak-$*$ lower
semicontinuity of the norms in
$\mathfrak{K}(u,p)+\mathfrak{K}(\eta)$,
\begin{equation}\label{kkes}
 \mathfrak{K}(u,p)+\mathfrak{K}(\eta)  \le P_\kappa\left(\ns{u_0}_{4N}+\ns{\eta_0}_{4N+1}\right).
\end{equation}
Similarly, the uniform smallness of $\mathfrak{L} (\eta^{m+1})$ in
\eqref{l_it_03} implies
\begin{equation}
\mathfrak{L} (\eta) \le \varepsilon_0.
\end{equation}

The convergence of a subsequence of $(u^m,p^m,\eta^m)$ does not
guarantee that we can simply pass to the limit in the system
\eqref{mkgeometric} in order to  produce the desired solution to
\eqref{kgeometric}. However, by restrict $T_0^\kappa$ to be smaller
if necessary, the uniform estimates \eqref{l_it_03} will allow us to
show that $(u^m,p^m,\eta^m)$ is contractive in certain norms of
lower order regularity; the proof is very similar to Theorem 6.2 of
\cite{GT_lwp} and we omit the details. This implies that the whole
sequence of $(u^m,p^m,\eta^m)$ converge strongly to the limit
$(u,p,\eta)$. This strong convergence together with the higher order
regularity weak convergence above allow us to interpolate to derive
a strong convergence result that is more than sufficient for us to
pass to the limit in the system \eqref{mkgeometric} for each $t\in
[0,T_0^\kappa]$. Then we find that the limit $(u,p,\eta)$ is a
strong solution of the $\kappa$-problem \eqref{kgeometric} on the
time interval $[0,T_0^\kappa]$ that achieves the initial data. The
uniqueness of the solution to \eqref{kgeometric} with $\kappa>0$
satisfying \eqref{kkes} follows also from the contractive argument.
\end{proof}

\section{Local in time theory}\label{sec local}

The aim of this section is that first for each fixed $\sigma>0$ we will pass to the limit as $\kappa\rightarrow0$ in the regularized $\kappa$-problem \eqref{kgeometric} to produce a local unique strong solution to the original problem \eqref{geometric}, and then we will pass to the limit as $\sigma\rightarrow0$ in \eqref{geometric} to rigorously justify the zero surface tension limit within a local time interval.

\subsection{Preliminaries}

To estimate $(\dt^ju,\dt^jp)$ for $j\ge 1$, we will recall that $(D_t^j u, \dt^j p)$ solves
\begin{equation}\label{l_lwp_0111}
\begin{cases}
\dt (D_t^j u) - \da (D_t^j u) + \naba (\dt^j p) = F^{1,j} & \text{in }\Omega \\
\diva(D_t^j u)=0 & \text{in }\Omega\\
\Sa(\dt^j p, D_t^j u) \n =  F^{3,j} & \text{on }\Sigma \\
D_t^j u =0 & \text{on }\Sigma_b
\end{cases}
\end{equation}
in the strong sense for $j=1,\dotsc,2N-1$, and in the weak sense of \eqref{l_weak_solution_pressureless} for $j=2N$.  Here the vectors $F^{1,j}$ and $F^{3,j}$ are as defined by \eqref{l_Fj_def} in which $F^1$ and $F^3$ are as defined by \eqref{l_F_forcing_def}.

We will restate the estimates of $F^{1,j}$ and $F^{3,j}$ with a slight modification from Lemma \ref{l_iteration_estimates_1}. To this end, we define
\begin{equation}\label{2l_Ffrak_def}
 \mathfrak{F}_\ast(F^1,F^3) := \sum_{j=1}^{2N-1} \ns{D_t^j F^1}_{L^2 H^{4N-2j -1}}+\ns{\dt D_t^{2N-1}F^1}_{L^2 (\H1)^\ast} +  \sum_{j=1}^{2N } \ns{\dt^j F^3}_{L^2 H^{4N-2j -1/2}}
\end{equation}
and
\begin{equation}
\begin{split}
 \mathfrak{F}(u,p) := &\ns{\dt^{2N-1} u}_{L^2 H^{2}} +\sum_{\ell=0}^{2N-2} \ns{\dt^\ell u}_{L^2 H^{4N-2\ell-1}}+ \sum_{\ell=0}^{2N-1} \ns{\dt^\ell u}_{L^\infty H^{4N-2\ell-2}}
\\ &  +
\ns{\dt^{2N-1} p}_{L^2 H^1}+\sum_{\ell=0}^{2N-2} \ns{\dt^\ell p}_{L^2 H^{4N-2\ell-2}} + \ns{\dt^\ell p}_{L^\infty H^{4N-2\ell-3}}.
\end{split}
\end{equation}

\begin{remark}
We can deduce an estimate of $\mathfrak{F}(u,p)$ in terms of $T$. Indeed, we apply Lemma \ref{l_sobolev_infinity} to obtain that for $\ell=0,\dots,2N-1$
\begin{equation}
\ns{\dt^\ell u}_{L^\infty H^{4N-2\ell-2}}\le \ns{\dt^\ell u(0)}_{ {4N-2\ell-2}}+\ns{\dt^\ell u}_{L^2 H^{4N-2\ell-2}}+\ns{\dt^{\ell+1} u}_{L^2 H^{4N-2\ell-2}},
\end{equation}
and for $\ell=0,\dots,2N-2$
\begin{equation}
\ns{\dt^\ell p}_{L^\infty H^{4N-2\ell-3}}\le \ns{\dt^\ell p(0)}_{ {4N-2\ell-3}}+\ns{\dt^\ell p}_{L^2 H^{4N-2\ell-3}}+\ns{\dt^{\ell+1} p}_{L^2 H^{4N-2\ell-3}}.
\end{equation}
Then we have
\begin{equation}\label{upast}
\begin{split}
 \mathfrak{F}(u,p)& \le \mathfrak{E}_0(u,p)+\sum_{\ell=0}^{2N-1} \ns{\dt^\ell u}_{L^2 H^{4N-2\ell}}+   \ns{\dt^\ell p}_{L^2 H^{4N-2\ell-1}}
 \\& \le \mathfrak{E}_0(u,p)+T\sum_{\ell=0}^{2N-1} \ns{\dt^\ell u}_{L^\infty H^{4N-2\ell}}+   \ns{\dt^\ell p}_{L^\infty H^{4N-2\ell-1}}
 \le \mathfrak{E}_0(u,p)+T\mathfrak{K}(u,p),
\end{split}
\end{equation}
where $\mathfrak{E}_0(u,p)$ and $\mathfrak{K}(u,p)$ are defined by \eqref{e0up} and \eqref{l_DEfrak_def}, respectively.
\end{remark}

We now record a version of estimates of $F^{1,j}$ and $F^{3,j}$, recalling $\mathfrak{F}(\eta)$ as defined in \eqref{l_Kfrak_def},

 \begin{lem}\label{fstar}
It holds that
\begin{equation}\label{2l_Fe1}
\begin{split}
  &\sum_{j=1}^{2N-1}\ns{F^{1,j}}_{L^2 H^{4N-2j-1}}  +  \ns{F^{3,j}}_{L^2 H^{4N-2j-1/2}}+\ns{\dt (F^{1,2N-1} -F^{3,2N-1})  }_{(\x_T)^*}
   \\&\quad\le  \mathfrak{F}_\ast(F^1,F^3) +  P(1+ \mathfrak{F}(\eta))
\mathfrak{F}(u,p).
\end{split}
\end{equation}
\end{lem}
\begin{proof}
The estimate \eqref{2l_Fe1} follows by taking $m=2N-1$ in \eqref{l_ie1_01} and \eqref{l_ie1_03} of Lemma \ref{l_iteration_estimates_1}; the term $\mathfrak{F}(F^1,F^3)$ therein can be replaced by $\mathfrak{F}_\ast(F^1,F^3)$ is due to the definition \eqref{l_Fj_def}.
\end{proof}

We also record here a version of estimates for the difference between $\dt v$ and $D_t v$.
\begin{lem}\label{l_iteration_estimate_66}
For $j=1,\dotsc,2N-1$,
\begin{equation}\label{sss1}
 \ns{\dt^j u - D_t^j u}_{L^2 H^{4N-2j+1}}  \le   P(1+ \mathfrak{F}(\eta))
\mathfrak{F}(u,0),
\end{equation}
\begin{equation}\label{sss2}
 \ns{\dt^j u - D_t^j u}_{L^\infty H^{4N-2j}}  \le   P(1+ \mathfrak{F}(\eta))
\mathfrak{F}(u,0),
\end{equation}
\begin{equation}\label{sss4}
\norm{\dt^{2N}u-\dt D_t^{2N-1}u}_{L^\infty H^0}^2  \le   P(1+ \mathfrak{F}(\eta)) \mathfrak{F}(u,0),
\end{equation}
and
\begin{equation}\label{sss3}
\begin{split}
& \ns{\dt D_t^{2N-1} u - \dt^{2N} u}_{L^2 H^1} + \ns{\dt^2 D_t^{2N-1} u - \dt^{2N+1} u}_{(\x_T)^*} \\
&\quad\le   P(1+ \mathfrak{F}(\eta))  \left( \mathfrak{F}(u,0)+  \ns{\dt^{2N-1} u}_{L^\infty H^{2}}  + \ns{\dt^{2N} u}_{(\x_T)^*}  \right).
\end{split}
\end{equation}
\end{lem}
\begin{proof}
These estimates follow similarly as Lemma \ref{fstar}.
\end{proof}

Next, to estimate $\dt^j\eta$ for $j\ge 1$, we will recall that
\begin{equation}\label{kkk}
\dt \eta    = \kappa \Delta \eta+ \kappa\Psi+F^4  \text{ on } \Sigma.
\end{equation}

Finally, to estimate $(u,p,\eta)$ itself without temporal derivatives, we shall write the system \eqref{kgeometric} in the perturbed form:
\begin{equation}\label{perturb2}
\begin{cases}
 \partial_t u- \Delta u+\nabla p=G^1\quad&\text{in }\Omega
\\ \diverge{u}=G^2&\text{in }\Omega
\\ ( pI-\mathbb{D}(u)) e_3= (\eta  -\sigma \Delta_\ast\eta) e_3+G^3&\text{on }\Sigma
\\ \partial_t\eta-\kappa\Delta_\ast\eta-u_3=\kappa\Psi+G^4&\text{on }\Sigma
\\ u=0 &\text{on }\Sigma_b.
\end{cases}
\end{equation}
Here we have written the vector $G^1 $ for
\begin{equation}\label{G1_def}
G_i^1 =\dt \bar{\eta} \tilde{b} K \p_3 u_i-
u_j \a_{jk}\p_k u_i+  (\a_{jk}\a_{jl}-\delta_{jk}\delta_{jl})\p_{kl} u_i+\a_{jk}\p_k\a_{jl}\p_l u_i-(\a_{ij}-\delta_{ij})\p_j p,
\end{equation}
the function $G^2$ for
\begin{equation}\label{G2_def}
G^2=-(\a_{ij}-\delta_{ij})\p_j u_i,
\end{equation}
the vector $G^3$ for
\begin{equation}\label{G3_def}
\begin{split}
G^3 = & \p_1 \eta
\begin{pmatrix}
 p-\eta -2(\p_1 u_1 -AK \p_3 u_1  ) \\
 -\p_2 u_1 - \p_1 u_2  + BK \p_3 u_1 + AK \p_3 u_2 \\
 -\p_1 u_3 - K \p_3 u_1 + AK \p_3 u_3
\end{pmatrix}
\\ &+
\p_2 \eta
\begin{pmatrix}
  -\p_2 u_1 - \p_1 u_2  + BK \p_3 u_1 + AK \p_3 u_2  \\
  p-\eta -2(\p_2 u_2 -BK \p_3 u_2  )  \\
 -\p_2 u_3 - K \p_3 u_2 + BK \p_3 u_3
\end{pmatrix}
+\begin{pmatrix}
  (K-1) \p_3 u_1  +AK \p_3 u_3 \\
  (K-1) \p_3 u_2  +BK \p_3 u_3  \\
  2(K-1)\p_3 u_3
\end{pmatrix}
\\& -\sigma \diverge_\ast(((1+|\nab_\ast\eta|^2)^{-1/2}-1)\nab_\ast\eta)\n
,
\end{split}
\end{equation}
and the function $G^4$ for
\begin{equation}\label{G4_def}
 G^4 =-u\cdot\nab_\ast\eta.
\end{equation}

We want to apply the horizontal spatial derivatives $\pa$ to \eqref{perturb2} with $\al\in \mathbb{N}^2$ so that $|\al|\le 4N$. However, the lower boundary $\Sigma_b$ may not be flat, so we are not free to apply such derivatives to \eqref{perturb2}. The idea then is to localize away from the lower boundary. To this end, we introduce the cutoff function $\chi \in C^\infty_c(\mathbb{R})$ so that
\begin{equation}\label{chi_properties}
\supp \chi  \subset [-3b_-/4, 1] \text{ and }\chi(x_3) = 1 \text{ for }x_3 \in [-b_-/2, 1/2].
\end{equation}
Multiplying the equations in \eqref{perturb2} by $\chi$, we find that $( \chi u, \chi p, \eta)$ solve
\begin{equation}\label{p_localized_equations}
 \begin{cases}
  \dt (\chi  u) + \nab (\chi  p) - \Delta (\chi  u) = \chi  G^1 + H^{1 } & \text{in }\Omega \\
  \diverge(\chi  u) = \chi  G^2 + H^{2 } & \text{in }\Omega \\
  ((\chi  p) I - \sg (\chi  u) )e_3 =   (\eta  -\sigma\Delta_\ast \eta) e_3 + G^3   & \text{on }\Sigma \\
   \dt \eta -\kappa\Delta_\ast\eta- (\chi  u_3) =\kappa\Psi+ G^4  &\text{on } \Sigma \\
  \chi  u =0 & \text{on }\Sigma_b,
 \end{cases}
\end{equation}
where
\begin{equation}\label{p_H_def}
H^{1 } = \p_3 \chi  (  p e_3 - 2  \p_3 u) - \p_3^2 \chi   u
\text{ and }
 H^{2 } = \p_3 \chi   u_3.
\end{equation}
Since now $\chi u$ and $\chi p$ have the support  away from $\Sigma_b$, we could apply $\partial^\alpha$ ($\alpha \in \mathbb{N}^{2}$) to \eqref{p_localized_equations}.

\subsection{Local well-posedness of \eqref{geometric} with surface tension}
In this subsection we fix $\sigma>0$, and our aim is to pass to the limit as $\kappa\rightarrow0$ in the $\kappa$-problem \eqref{kgeometric}. We define a special quantity for this subsection as
\begin{equation}\label{dnk}
\begin{split}
  \mathfrak{K}^{\kappa}=\mathfrak{K}^{\kappa}(u,p,\eta):&
  =\sup_{0\le t\le T}\mathfrak{E}(t) +\int_0^T\mathfrak{D}^{\kappa}(t)\,dt+ \ns{\dt^{2N+1} u}_{(\x_T)^*},
 \end{split}
\end{equation}
where
\begin{equation}
\mathfrak{E}= \sum_{j=0}^{2N} \ns{\dt^j u}_{ {4N-2j }} + \sum_{j=0}^{2N-1}\ns{\dt^j p}_{ {4N-2j-1}}+ \sum_{j=0}^{2N}   \ns{\dt^j \eta}_{ {4N-2j+1}}
\end{equation}
and
\begin{equation}
\begin{split}
\mathfrak{D}^{\kappa}= &\sum_{j=0}^{2N} \ns{\dt^j u}_{ {4N-2j +1}} + \sum_{j=0}^{2N-1}\ns{\dt^j p}_{{4N-2j}}
\\&+\kappa  \ns{ \eta}_{ {4N+2}}  +  \ns{ \eta}_{ {4N+3/2}}+  \sum_{j=1}^{2N+1}   \ns{\dt^j \eta}_{ {4N-2j+2}}.
\end{split}
\end{equation}

For each $\kappa>0$, we let $(u,p,\eta)$ (dropping the
$\kappa$-dependence temporally) be the solution to the problem
\eqref{kgeometric} on $[0,T_0^\kappa]$ produced by Theorem
\ref{l_knwp}. Moreover, Theorem \ref{l_knwp} implies that for
$0<T\le T_0^\kappa$, we have $\mathfrak{K}^{\kappa}<\infty$ and the
following smallness of $\eta$:
\begin{equation}\label{smallsmall}
  \mathfrak{L}(\eta):=\sup_{0\le t\le T} \ns{ \eta(t)}_{4N-1/2}\le \varepsilon_0
\end{equation}
for the universal constant $\varepsilon_0\ll1$. Our goal is to derive the $\kappa$-independent estimates of the solution on $[0,T_0^\kappa]$, and the key step is to derive the uniform bound of $\mathfrak{K}^{\kappa}$. We will see that the smallness \eqref{smallsmall} is crucial for closing the $\kappa$-independent energy estimates since we can not get a power of $T$ in front of the estimates of some terms stemming from the mean curvature.

We recall from the construction of the initial data in Section \ref{l_data_section} that
\begin{equation}\label{2l_Fe2}
\mathfrak{E}(0)\ls \mathfrak{E}_0(u,p)+\ns{\eta(0)}_{4N+1} +   \sum_{j=1}^{2N+1} \ns{\dt^j \eta(0)}_{4N-2j+3/2}\ls\mathcal{E}_{2N}^1(0)\le P(\tilde{\mathcal{E}}_0),
 \end{equation}
where $\tilde{\mathcal{E}}_0:=\ns{u_0}_{4N}+\ns{\eta_0}_{4N+1}$, and $\mathcal{E}_{2N}^1(0)$ is as defined by \eqref{p_energy_def} with that $\sigma=1$, $n=2N$ and $t=0$. We also recall from the construction of $\Psi$ in Section \ref{sec compensator} that, by \eqref{2l_Fe2},
\begin{equation}\label{2l_Fe2123}
 \sum_{j=0}^{2N} \ns{\dt^j \Psi}_{L^2 H^{4N-2j }}   +  \sum_{j=0}^{2N}\ns{\dt^j \Psi}_{L^\infty H^{4N-2j -1}}
 \ls \sum_{j=0}^{2N}\ns{ \dt^j\eta(0) }_{ {4N-2j +1}}
  \le  P (\tilde{\mathcal{E}}_0).
\end{equation}

\subsubsection{The $\kappa$-independent energy estimates with temporal derivatives}\label{ksec}
We first record an estimate of the nonlinear terms $\mathfrak{F}_\ast(F^1,F^3)$ (defined by \eqref{2l_Ffrak_def}) and an estimate of $\mathfrak{F}(F^4)$ (as defined in \eqref{l_Kfrak_u2n_def1}) in terms of $\mathfrak{K}^{\kappa}$ (defined by \eqref{dnk}).
\begin{lem}\label{2l_Ffrak_bound}
Let $F^1(u,\eta),F^3(\eta),F^4(u,\eta)$ be defined by \eqref{l_F_forcing_def}. Then
 \begin{equation}\label{fast}
 \mathfrak{F}_\ast(F^1(u,\eta),F^3(\eta)) \le  P(\tilde{\mathcal{E}}_0)+ P( T^{1/4} \mathfrak{K}^{\kappa})
\end{equation}
and
\begin{equation}\label{fast2}
 \mathfrak{F} (F^4(u,\eta) ) \le  P(\tilde{\mathcal{E}}_0)+ P( T^{1/4} \mathfrak{K}^{\kappa}).
\end{equation}
 \end{lem}
 \begin{proof}
We first prove \eqref{fast2}. We recall from the estimate \eqref{f2 bound} of Lemma \ref{l_Ffrak_bound} that
  \begin{equation}\label{jaja0}
 \mathfrak{F}(F^4(u,\eta) ) \le  P(1+ \mathfrak{F}_{2N}(\eta))\mathfrak{F}_{2N}(u),
\end{equation}
where $\mathfrak{F}_{2N}(u)$, $\mathfrak{F}_{2N}(\eta)$ are defined by \eqref{l_Kfrak_u2n_def} and \eqref{l_Kfrak_eta2n_def}, respectively.
We then estimate $\mathfrak{F}_{2N}(\eta)$ and $\mathfrak{F}_{2N}(u)$. We recall from the estimates \eqref{eses2}--\eqref{eses3} that, by the estimates \eqref{2l_Fe2} and the definition of $\mathfrak{K}^{\kappa}$ \eqref{dnk},
 \begin{equation} \label{jaja1}
 \mathfrak{F}_{2N}(\eta) \le \sum_{j=0}^{2N-1} \ns{\dt^j \eta(0)}_{ {4N-2j}}+2T\sum_{j=0}^{2N} \ns{\dt^j \eta}_{L^\infty H^{4N-2j+1}}
 \le  P(\tilde{\mathcal{E}}_0)+ 2 T \mathfrak{K}^{\kappa}
 \end{equation}
 and
  \begin{equation} \label{jaja2}
 \mathfrak{F}_{2N}(u )  \le \sum_{j=0}^{2N-1}\ns{\dt^j u(0)}_{ {4N-2j-1/4}}+C T^{1/4}\mathfrak{K}_{2N}(u) \le  P(\tilde{\mathcal{E}}_0)+ C T^{1/4} \mathfrak{K}^{\kappa}.
 \end{equation}
Plugging the estimates \eqref{jaja1}--\eqref{jaja2} into \eqref{jaja0}, then we have
  \begin{equation}
 \mathfrak{F}(F^4(u,\eta) )  \le  P(1+  P(\tilde{\mathcal{E}}_0)+ 2  T \mathfrak{K}^{\kappa}) ( P(\tilde{\mathcal{E}}_0)+ C T^{1/4} \mathfrak{K}^{\kappa})
 \le  P(\tilde{\mathcal{E}}_0)+  P( T^{1/4} \mathfrak{K}^{\kappa}).
\end{equation}
This gives \eqref{fast2}.

Now we turn to prove \eqref{fast}. We refine the estimate \eqref{f bound} from the proof of Lemma \ref{l_Ffrak_bound} as
\begin{equation}\label{jaja00}
 \mathfrak{F}_\ast(F^1(u,\eta),F^3(\eta)) \le  P(1+ \mathfrak{F}(\eta))\left(   \mathfrak{F} (\eta) +(\mathfrak{F}_{2N}(u))^2\right)
\end{equation}
Here the reason that we are able to use $\mathfrak{F}(\eta)$ (as defined in \eqref{l_Kfrak_def}) to replace $\mathfrak{K}(\eta)$ (defined by \eqref{l_Ketafrak_def}) in \eqref{f bound} is due to that there is at least one temporal derivative in the definition of $\mathfrak{F}_\ast(F^1,F^3)$; the appearance of $\mathfrak{K}(\eta)$ in \eqref{f bound} is due to the estimate \eqref{0909}.
We then appeal to an estimate of $ \mathfrak{F}(\eta)$. First we trivially have
\begin{equation}\label{l_tt_812}
 \ns{\eta}_{L^2 H^{4N+1/2}} \le  T\ns{\eta}_{L^\infty H^{4N+1/2}}  \le T {\mathfrak{K}^\kappa }.
\end{equation}
Next for $j=1,\dots,2N$, we use the Sobolev interpolation to obtain
\begin{equation}\label{keta1}
  \ns{\dt^j \eta}_{L^2 H^{4N-2j+3/2}}
   \le T^{1/2}\norm{\dt^j  \eta}_{L^\infty H^{4N-2j +1}}\norm{\dt^j  \eta}_{L^2 H^{4N-2j +2}}  \le T^{1/2} \mathfrak{K}^{\kappa} .
\end{equation}
Now for $j= 2N+1$, we must resort to that $\eta$ is the solution to the regularized surface $\kappa$-problem \eqref{l_k transport_equation}; then we have, by \eqref{keta1} and \eqref{fast2},
\begin{equation}\label{l_tt_81233}
\begin{split}
  \ns{\dt^{2N+1} \eta}_{L^2 H^{-1/2}} &\le     \kappa^2\ns{\dt^{2N} \eta}_{L^2 H^{3/ 2}} + \kappa^2\ns{\dt^{2N} \Psi}_{L^2 H^{-1/2}}  +\ns{ \dt^{2N} F^4(u,\eta)}_{L^2 H^{-1/2}}
  \\&\ls   {T}^{1/2}\mathfrak{K}^{\kappa}+ P(\tilde{\mathcal{E}}_0)+T^{1/4} \mathfrak{K}^{\kappa}+  P(  T^{1/4} \mathfrak{K}^{\kappa})
\le  P(\tilde{\mathcal{E}}_0)+ P( T^{1/4} \mathfrak{K}^{\kappa}).
  \end{split}
\end{equation}
Hence, we conclude from \eqref{l_tt_812}--\eqref{l_tt_81233}, by further applying Lemma \ref{l_sobolev_infinity},
\begin{equation}\label{jaja3}
\begin{split}
 \mathfrak{F} ( \eta  )
 &\ls \ns{ \eta(0)}_{4N }+\sum_{j=1}^{2N} \ns{\dt^j \eta(0)}_{4N-2j+1/2}+ \ns{ \eta}_{L^2 H^{4N+1/2}}+ \sum_{j=1}^{2N+1} \ns{\dt^j \eta}_{L^2 H^{4N-2j+3/2}}
 \\&\le   P(\tilde{\mathcal{E}}_0)+T {\mathfrak{K}^\kappa }+ T^{1/2} \mathfrak{K}^{\kappa}+ P(  T^{1/4} \mathfrak{K}^{\kappa})\le  P(\tilde{\mathcal{E}}_0) + P(  T^{1/4} \mathfrak{K}^{\kappa}).
 \end{split}
\end{equation}
Plugging the estimates \eqref{jaja2} and \eqref{jaja3} into \eqref{jaja00}, we deduce
\begin{equation}
\begin{split}
 \mathfrak{F}_\ast(F^1(u,\eta),F^3(\eta)) &\le  P(1+ P(\tilde{\mathcal{E}}_0)+ P(  T^{1/4} \mathfrak{K}^{\kappa})) \left( P(\tilde{\mathcal{E}}_0)+ P(  T^{1/4} \mathfrak{K}^{\kappa})+
( 2 T^{1/4} \mathfrak{K}^{\kappa})^2\right)
\\&\le   P(\tilde{\mathcal{E}}_0)+ P(  T^{1/4} \mathfrak{K}^{\kappa}).
 \end{split}
\end{equation}
This gives \eqref{fast}, and we then conclude the lemma.
\end{proof}

We shall now use the equations \eqref{l_lwp_0111} to estimate $(\dt^ju,\dt^jp)$ for $j\ge 1$.
\begin{prop}\label{k11}
Suppose that $(u,p,\eta)$ is the solution to \eqref{kgeometric}. Then
\begin{equation}\label{kestimate1}
\begin{split}
   & \sum_{j=1}^{2N} \ns{\dt^j u}_{L^2 H^{4N-2j +1}} + \ns{\dt^{2N+1} u}_{(\x_T)^*} + \sum_{j=1}^{2N-1}\ns{\dt^j p}_{ L^2 H^{4N-2j}}  \\
 &\quad+\sum_{j=1}^{2N} \ns{\dt^j u}_{L^\infty H^{4N-2j }} + \sum_{j=1}^{2N-1}\ns{\dt^j p}_{L^\infty H^{4N-2j-1}}
\\& \qquad\le  \left(  P(\tilde{\mathcal{E}}_0)  + P(  T^{1/4} \mathfrak{K}^{\kappa})
  \right)  \exp\left( \left(  P(\tilde{\mathcal{E}}_0)  + P(  T^{1/4} \mathfrak{K}^{\kappa})
  \right) T \right).
   \end{split}
\end{equation}
\end{prop}
\begin{proof}
First, taking $j=2N-1$ in \eqref{l_lwp_0111}, we then apply the estimate \eqref{l_ss_02} of Theorem \ref{l_strong_solution} to find, using the estimate \eqref{2l_Fe1} of Lemma \ref{fstar},
\begin{equation}\label{2n1}
\begin{split}
   & \norm{D_t^{2N-1}u}_{L^\infty H^2}^2+\norm{D_t^{2N-1}u}_{L^2 H^3}^2 + \norm{\dt D_t^{2N-1}u}_{L^\infty H^0}^2 + \norm{\dt D_t^{2N-1}u}_{L^2 H^1}^2
   \\&\quad + \norm{\dt^2 D_t^{2N-1}u}_{(\x_T)^*}^2 + \norm{\dt^{2N-1}p}_{L^\infty H^1}^2  + \norm{\dt^{2N-1}p}_{L^2 H^2}^2
\\&\quad\quad\le   P(1+ \mathcal{K}(\eta)) \exp\left( P(1+ \mathcal{K}(\eta)) T \right)
\left(  P(\tilde{\mathcal{E}}_0)+\ns{F^{1,2N-1}}_{L^2H^1}\right. \\
&\qquad\quad \left. +\ns{F^{3,2N-1}}_{L^2H^{3/2}}+\ns{\dt (F^{1,2N-1} -F^{3,2N-1})  }_{(\x_T)^*}
  \right)
  \\&\quad\quad\le   P(1+ \mathfrak{F}(\eta)) \exp\left( P(1+ \mathfrak{F}(\eta)) T \right)
\left(  P(\tilde{\mathcal{E}}_0) +\mathfrak{F}_\ast(F^1,F^3) +
\mathfrak{F}(u,p)
  \right):=\mathcal{Z}.
   \end{split}
\end{equation}
We will allow $\mathcal{Z}$ to change in the sense that the polynomial $P$ changes from line to line.

While for $j=1,\dots,2N-2$, we rewrite \eqref{l_lwp_0111} as
\begin{equation}\label{l_lwp_0122}
\begin{cases}
- \da (D_t^j u) + \naba (\dt^j p) =- \dt (D_t^j u) +F^{1,j} & \text{in }\Omega \\
\diva(D_t^j u)=0 & \text{in }\Omega\\
\Sa(\dt^j p, D_t^j u) \n =  F^{3,j} & \text{on }\Sigma \\
D_t^j u =0 & \text{on }\Sigma_b.
\end{cases}
\end{equation}
We apply the $\a$--Stokes elliptic regularity theory of Proposition \ref{l_stokes_regularity} with $k=4N$ and $r=4N-2j+1\le 4N-1$ to \eqref{l_lwp_0122} to obtain, using the estimate \eqref{2l_Fe1} of Lemma \ref{fstar},
\begin{equation}\label{idid}
\begin{split}
 &\ns{D_t^j u}_{L^2H^{4N-2j+1}}+\ns{ \dt^j p}_{L^2H^{4N-2j} }
 \\&\quad \ls  \ns{\dt (D_t^j u) }_{L^2H^{4N-2j-1}}
+\ns{F^{1,j}}_{L^2H^{4N-2j-1}}+\ns{F^{3,j}}_{L^2H^{4N-2j-1/2}}
\\& \quad \ls  \ns{ D_t^{j+1} u  }_{L^2H^{4N-2j-1}}+\ns{R D_t^{j } u  }_{L^2H^{4N-2j-1}}+\mathfrak{F}_\ast(F^1,F^3) +  P(1+ \mathfrak{F}(\eta))
\mathfrak{F}(u,p)
\\& \quad \ls  \ns{ D_t^{2N-1} u  }_{L^2H^{3}}+ \mathfrak{F}_\ast(F^1,F^3) +  P(1+ \mathfrak{F}(\eta))
\mathfrak{F}(u,p) .
\end{split}
\end{equation}
A simple induction on \eqref{idid} yields
\begin{equation}\label{2n12}
\begin{split}
 &\sum_{j=1}^{2N-2}\ns{D_t^j u}_{L^2H^{4N-2j+1}}+\ns{ \dt^j p}_{L^2H^{4N-2j} }
\\& \quad \ls  \ns{ D_t^{2N-1} u  }_{L^2H^{3}}
+\mathfrak{F}_\ast(F^1,F^3) +  P(1+ \mathfrak{F}(\eta))
\mathfrak{F}(u,p) .
\end{split}
\end{equation}
Hence, we combine \eqref{2n1} and \eqref{2n12} to have
\begin{equation}\label{ee1}
\begin{split}
 &\sum_{j=1}^{2N-1}\ns{D_t^j u}_{L^2H^{4N-2j+1}}+\ns{ \dt^j p}_{L^2H^{4N-2j} }+ \norm{\dt D_t^{2N-1}u}_{L^2 H^1}^2  + \norm{\dt^2 D_t^{2N-1}u}_{(\x_T)^*}^2
 \\&\quad+\norm{D_t^{2N-1}u}_{L^\infty H^2}^2+ \norm{\dt D_t^{2N-1}u}_{L^\infty H^0}^2   + \norm{\dt^{2N-1}p}_{L^\infty H^1}^2
 \le   \mathcal{Z}.
\end{split}
\end{equation}

Now we want to derive more estimates based on the estimate \eqref{ee1}. By the estimates \eqref{sss1}--\eqref{sss4} of Lemma \ref{l_iteration_estimate_66} and \eqref{ee1}, we have
\begin{equation}\label{tot1}
\begin{split}
 \sum_{j=1}^{2N-1}\ns{\dt^j u}_{L^2H^{4N-2j+1}} &\le \sum_{j=1}^{2N-1}\ns{D_t^j u}_{L^2H^{4N-2j+1}} +
 \sum_{j=1}^{2N-1}\ns{\dt^j u-D_t^j u}_{L^2H^{4N-2j+1}}
 \\&  \le \mathcal{Z}+ P(1+ \mathfrak{F}(\eta))
\mathfrak{F}(u,0)\ls \mathcal{Z}
\end{split}
\end{equation}
and
\begin{equation}\label{sssss}
\begin{split}
  \norm{\dt^{2N-1}u}_{L^\infty H^2}^2+\norm{\dt^{2N}u}_{L^\infty H^2}^2&\le  \norm{D_t^{2N-1}u}_{L^\infty H^2}^2+\norm{\dt^{2N-1}-D_t^{2N-1}u}_{L^\infty H^0}^2
  \\&\quad+\norm{ \dt D_t^{2N-1}u}_{L^\infty H^0}^2
  +\norm{\dt^{2N}u-\dt D_t^{2N-1}u}_{L^\infty H^0}^2
 \\&  \le \mathcal{Z}+ P(1+ \mathfrak{F}(\eta))
\mathfrak{F}(u,0)\ls \mathcal{Z}.
\end{split}
\end{equation}
By the estimate \eqref{sss3} of Lemma \ref{l_iteration_estimate_66}, \eqref{ee1} and \eqref{sssss}, we have
\begin{equation}
\begin{split}
& \ns{ \dt^{2N} u}_{L^2 H^1} + \ns{ \dt^{2N+1} u}_{(\x_T)^*}\le
 \ns{\dt D_t^{2N-1} u }_{L^2 H^1} + \ns{\dt^2 D_t^{2N-1} u }_{(\x_T)^*}
 \\&\qquad\qquad\qquad\qquad+\ns{\dt D_t^{2N-1} u - \dt^{2N} u}_{L^2 H^1} + \ns{\dt^2 D_t^{2N-1} u - \dt^{2N+1} u}_{(\x_T)^*} \\
&\quad\le  \mathcal{Z}+ P(1+ \mathfrak{F}(\eta))  \left( \mathfrak{F}(u,0)+  \ns{\dt^{2N-1} u}_{L^\infty H^{2}}  + T\ns{\dt^{2N} u}_{L^\infty H^0}  \right)
 \le  \mathcal{Z} .
\end{split}
\end{equation}
On the other hand, we use Lemma \ref{l_sobolev_infinity} and \eqref{ee1}--\eqref{tot1} to deduce
\begin{equation}\label{321}
\begin{split}
 &\sum_{j=1}^{2N-2}\ns{\dt^j u}_{L^\infty H^{4N-2j}}+\ns{ \dt^j p}_{L^\infty H^{4N-2j-1} }
 \\&\ \ \le \sum_{j=1}^{2N-2}\ns{\dt^j u(0)}_{{4N-2j}}+\ns{ \dt^j p(0)}_{{4N-2j-1} } +\sum_{j=1}^{2N-1}\ns{\dt^j u}_{L^2H^{4N-2j+1}}+\ns{ \dt^j p}_{L^2H^{4N-2j} }
\\&\ \ \le   P(\tilde{\mathcal{E}}_0)+\mathcal{Z}\ls \mathcal{Z}.
\end{split}
\end{equation}

We can now conclude the estimates for $(\dt^ju,\dt^jp)$ for $j\ge 1$ from \eqref{ee1}--\eqref{321} as
\begin{equation}\label{upes}
\begin{split}
   & \sum_{j=1}^{2N} \ns{\dt^j u}_{L^2 H^{4N-2j +1}} + \ns{\dt^{2N+1} u}_{(\x_T)^*} + \sum_{j=1}^{2N-1}\ns{\dt^j p}_{ L^2 H^{4N-2j}} \\
 &\quad+\sum_{j=1}^{2N} \ns{\dt^j u}_{L^\infty H^{4N-2j }} + \sum_{j=1}^{2N-1}\ns{\dt^j p}_{L^\infty H^{4N-2j-1}}
 \le   \mathcal{Z}.
   \end{split}
\end{equation}
Turning back to estimate $\mathcal{Z}$, we employ the estimates \eqref{jaja3}, \eqref{fast}, \eqref{upast} and \eqref{2l_Fe2} to deduce, since $\mathfrak{K}(u,p)\le \mathfrak{K}^\kappa$,
\begin{equation} \label{zestimate}
\begin{split}
   \mathcal{Z} &= P(1+ \mathfrak{F}(\eta)) \exp\left( P(1+ \mathfrak{F}(\eta)) T \right)
\left(  P(\tilde{\mathcal{E}}_0) +\mathfrak{F}_\ast(F^1,F^3) +
\mathfrak{F}(u,p)
  \right)
  \\&\le  P(1+  P(\tilde{\mathcal{E}}_0)+ P(  T^{1/4} \mathfrak{K}^{\kappa})) \exp\left( P(1+  P(\tilde{\mathcal{E}}_0)+ P(  T^{1/4} \mathfrak{K}^{\kappa})) T \right)
  \\&\quad\times  \left(  P(\tilde{\mathcal{E}}_0)  + P(  T^{1/4} \mathfrak{K}^{\kappa}) +
T \mathfrak{K}^{\kappa}
  \right)
  \\&\le \left(  P(\tilde{\mathcal{E}}_0)  + P(  T^{1/4} \mathfrak{K}^{\kappa})
  \right)  \exp\left( P(1+ \mathcal{E}_0+ P(  T^{1/4} \mathfrak{K}^{\kappa})) T \right)
  \\&\le \left(  P(\tilde{\mathcal{E}}_0)  + P(  T^{1/4} \mathfrak{K}^{\kappa})
  \right)  \exp\left( \left(  P(\tilde{\mathcal{E}}_0)  + P(  T^{1/4} \mathfrak{K}^{\kappa})
  \right) T \right)  .
   \end{split}
\end{equation}
Plugging the estimate \eqref{zestimate} into \eqref{upes}, we then deduce \eqref{kestimate1}.
\end{proof}

We then use the equation \eqref{kkk} to estimate $\dt^j\eta$ for $j\ge 1$.
\begin{prop}\label{k33}
Suppose that $(u,p,\eta)$ is the solution to \eqref{kgeometric}. Then
\begin{equation} \label{etaall}
  \sum_{j=1}^{2N+1}\ns{\dt^j \eta}_{L^2 H^{ 4N -2j +2} }+\sum_{j=1}^{2N}\ns{\dt^j \eta}_{L^\infty H^{ 4N -2j +1} }
 \ls P(\tilde{\mathcal{E}}_0)+ P( T^{1/4} \mathfrak{K}^{\kappa})+\kappa^2\ns{  \eta}_{ L^2 H^{4N +2 }}.
\end{equation}
\end{prop}
\begin{proof}
We recall from \eqref{eta es12 2}--\eqref{eta es 2} that, by \eqref{2l_Fe2}--\eqref{2l_Fe2123},
\begin{equation}  \label{eett}
\begin{split}
 & \sum_{j=1}^{2N+1}\ns{\dt^j \eta}_{L^2 H^{ 4N -2j +2} }+\sum_{j=1}^{2N}\ns{\dt^j \eta}_{L^\infty H^{ 4N -2j +1} }
 \\&\quad\ls\kappa^2\ns{  \eta}_{ L^2 H^{4N +2 }}+\sum_{j=0}^{2N} \ns{\dt^j \eta(0)}_{4N-2j+1} + \mathfrak{F}(F^4(u,\eta))
 \\&\quad\ls\kappa^2\ns{  \eta}_{ L^2 H^{4N +2 }}+ P(\tilde{\mathcal{E}}_0)+ \mathfrak{F}(F^4(u,\eta)).
 \end{split}
\end{equation}
Plugging the estimate \eqref{fast2} into \eqref{eett}, we obtain \eqref{etaall}.
\end{proof}

\subsubsection{The $\kappa$-independent energy estimates without temporal derivatives}\label{ksec2}

To simplify notation, we define the energy and dissipation without temporal derivatives by, where we temporally keep the dependence of $\sigma$,
\begin{equation}
  {\mathcal{E}}  =  \ns{   u  }_{4N}  +  \ns{ p}_{4N-1} +  \ns{ \eta}_{4N} + \sigma \ns{ \eta}_{4N+1}
\end{equation}
and
\begin{equation}
     \mathcal{D}
   = \ns{   u  }_{4N+1}+\ns{   p  }_{4N}  +\kappa\sigma \ns{ \eta}_{4N +2}+ \sigma^2 \ns{ \eta}_{4N +3/2}+\ns{   \eta}_{4N-1/2}.
\end{equation}
We also define the ``horizontal" energy and dissipation with localization as
\begin{equation}
 \bar{\mathcal{E}}  =  \ns{  \nab_{\ast 0}^{\  4N}(\chi u) }_{0}  +   \ns{ \eta}_{4N} + \sigma \ns{ \eta}_{4N+1}\text{ and } \bar{\mathcal{D}}  = \ns{ \nab_{\ast 0}^{\  4N}  \mathbb{D}(\chi u) }_{0} +\kappa\sigma \ns{ \eta}_{4N +2},
\end{equation}
where $\chi$ is as defined in \eqref{chi_properties}.

We first record an estimate of the nonlinear terms $G^i$, defined by \eqref{G1_def}--\eqref{G4_def}.
\begin{lem}\label{gle}
Let $\varepsilon_0$ be the small universal constant in \eqref{smallsmall}. Then we have
\begin{equation}\label{Gesti}
\begin{split}
 \norm{ G^{1}}_{4N-1}+\norm{ G^{2}}_{4N} +\norm{ G^{3}}_{4N-1/2} +\norm{ G^{4}}_{4N-1/2}\ls   {\varepsilon_0}\sqrt{ \mathcal{D}}+  P(1+ \mathfrak{E} ).
  \end{split}
\end{equation}
\end{lem}
\begin{proof}
We note that all terms are at least quadratic. Then
we apply the spatial derivatives $\pa$ (for the appropriate $\al$ depending on $G^i$) and expand using the Leibniz rule; each term in the resulting
sum is also at least quadratic and can be written in the form $X Y$ or $ZZ$. Here $Y$ is the factor involving the highest derivatives and $X$ is a product of factors involving very low derivatives; $Z$ is the factor involving fewer derivatives than $Y$ so that the various norms of $Z$ can be bounded by $ P(1+ \mathfrak{E} ) \sqrt{\mathfrak{E} }$.  We then focus on how to estimate the terms of form $XY$.

We first estimate the $\ns{\pa G^{1}}_0$ for $\al\in \mathbb{N}^3$ with $|\al|\le 4N-1$. Note that those factors appearing as $Y$ of $\pa G^1$ are
$ \p^{\al}\nab^2 u,\p^{\al}\nab  p, \p^{\al}\dt\bar\eta,\p^{\al}\nab^2\bar\eta
$. Then we have
\begin{equation}\label{1alpha1}
\norm{\p^{\al}\nab^2 u}_{0} \le  \sqrt{\mathcal{D}}, \norm{\p^{\al}\nab  p}_{0} \le  \sqrt{\mathcal{D}},
\end{equation}
\begin{equation}
\norm{\p^{\al}\dt\bar\eta}_{(\H1)^\ast}\ls \norm{ \dt \bar \eta}_{4N -1}\ls  \norm{ \dt   \eta}_{4N-3/2} \ls   \sqrt{\mathfrak{E} },
\end{equation}
\begin{equation}\label{1alpha7}
\norm{\p^{\al}\nab^2\bar\eta}_{(\H1)^\ast}\ls \norm{ \nab^2\bar \eta}_{4N -1}\ls \norm{   \eta}_{4N +1/2}\ls  \sqrt{\mathfrak{E} }.
\end{equation}
We can bound certain norms of $X$ related to the last two factors by $ P(1+ \mathfrak{E} ) \sqrt{\mathfrak{E} }$. However,
to bound the $X$ related to the first two factors we need to resort to the factor that such $X$ only involves the spatial derivatives of $\bar{\eta}$ (up to second order); hence we can bound the certain norms of such $X$ by $  P(1+ \mathfrak{L}(\eta))\sqrt{\mathfrak{L}(\eta)}\ls {\varepsilon_0}$, using \eqref{smallsmall}.  In light of the analysis above and using that $ P(1+ \mathfrak{E} ) {\mathfrak{E} }\le  P(1+ \mathfrak{E} )$, we conclude the estimate of $G^1$ as in Lemma \ref{l_iteration_estimates_1}.

We now  estimate the $\ns{\pa G^{3}}_{1/2}$ for $\al\in \mathbb{N}^2$ with $|\al|\le 4N-1$. Note that the factors appearing as $Y$ of $\pa G^{3}$ are $ \p^{\al}p,  \p^{\al}\nab  u$, $\p^{\al}\nab \bar\eta$ and $\sigma\pa \nab_\ast^2\eta$. For the first three factors, we have
\begin{equation}\label{3alpha1}
 \norm{\p^{\al}p}_{H^{ 1/2}(\Sigma)}\ls \norm{ p}_{4N}\le \sqrt{\mathcal{D}},\ \norm{\p^{\al}\nab  u}_{H^{ 1/2}(\Sigma)} \ls \norm{ u}_{4N+1} \le \sqrt{\mathcal{D}} .
\end{equation}
\begin{equation}
 \norm{\p^{\al}\nab \bar\eta}_{H^{-1/2}(\Sigma)}\ls\norm{ \nab \bar\eta}_{H^{4N-1/2}(\Sigma)}
 \ls\norm{   \bar\eta}_{ {4N+1} }\ls\norm{    \eta}_{ 4N+1/2 } \ls  \sqrt{\mathfrak{E} } .
\end{equation}
To estimate the last factor we need to use the fact that there is a $\sigma$. Hence we can bound by
\begin{equation}\label{3alpha7}
 \sigma \norm{\p^{\al}\nab_\ast^2\eta}_{-1/2} \ls  \sigma \norm{ \eta}_{4N+3/2}\le  \sqrt{\mathcal{D}}.
\end{equation}
Then arguing as for the $G^1$ term, we conclude the estimate of $G^3$.

Finally, the estimates of $G^2$ and $G^4$ follows similarly as that for $G^1$ and $G^3$ and we then conclude the estimate \eqref{Gesti}.
\end{proof}

We now state the estimates of $(u,p,\eta)$ itself without temporal derivatives as follows.
\begin{prop}\label{up2222}
Suppose that $(u,p,\eta)$ is the solution to \eqref{kgeometric}. Then
\begin{equation}\label{upspatial}
  {\mathcal{E}}+\int_0^t  {\mathcal{D}} \ls P(\tilde{\mathcal{E}}_0)+ P (1+ \mathfrak{K}^{\kappa})T^{1/2}+ \ns{\partial_t u}_{L^2H^{4N-1}}+ \ns{\partial_t p}_{L^2H^{4N-2}}.
\end{equation}
\end{prop}
\begin{proof}
We first derive the evolution of $\bar{\mathcal{E}}$. We claim that
\begin{equation}\label{0ese}
 \bar{\mathcal{E}}+\int_0^t \bar{\mathcal{D}} \ls P(\tilde{\mathcal{E}}_0)+ P (1+ \mathfrak{K}^{\kappa}) T^{1/2}  + {\varepsilon_0} \int_0^t \mathcal{D}.
\end{equation}

To prove the claim, applying $\partial^\alpha$ with $\al\in \mathbb{N}^2$ so that $|\al|\le 4N$ to the first equation of \eqref{p_localized_equations} and then taking the dot product with $\pa u$, using the other equations, we obtain
\begin{equation} \label{p_u_e_1}
\begin{split}
 &\hal \frac{d}{dt}\left(  \int_\Omega \abs{\pa (\chi u)}^2  +    \int_\Sigma \abs{\pa \eta}^2 +  \sigma \abs{\pa\nab_\ast \eta}^2\right)
+ \hal \int_\Omega \abs{\sg \pa (\chi u)}^2
 \\&\quad+ \int_\Sigma\kappa\abs{\nab_\ast\pa\eta}^2+\kappa\sigma  \abs{\Delta_\ast\pa \eta}^2= \int_\Omega \chi  \pa u  \cdot (\chi\pa G^1+\pa H^{1 })
\\&\quad\ \
+\int_\Omega  \chi \pa p (\chi\pa G^2+\pa H^{2 })   + \int_\Sigma -\pa u \cdot \pa G^3 + (\pa \eta-\sigma\Delta_\ast\pa \eta) (\kappa\pa\Psi+\pa G^4).
\end{split}
\end{equation}

We will estimate the terms on the right side of \eqref{p_u_e_1}, beginning with the terms involving $H^{1}$ and $H^{2}$.
Since $\chi$ is only a function of $x_3$, we have that
\begin{equation}
 \pa H^{1} = \p_3 \chi  ( \pa p e_3 - 2 \pa \p_3 u) - \p_3^2 \chi  \pa u \text{ and } \pa H^{2} = \p_3 \chi  \pa u_3.
\end{equation}
This allows us to estimate
\begin{equation}\label{p_u_e_2}
\begin{split}
 \int_\Omega \chi  \pa u  \cdot \pa H^{1 }  + \chi  \pa p \pa H^{2 } &\ls
\norm{\pa u}_{0}\left(\norm{\pa p}_{0} + \norm{\pa u}_{1} \right)
+ \norm{\pa p}_{0} \norm{\pa u}_{0} \\
&\ls \norm{ u}_{4N}\left(\norm{ p}_{4N} + \norm{ u}_{4N+1} \right) \le
  \sqrt{\mathfrak{E} }\sqrt{\mathcal{D}} .
  \end{split}
\end{equation}

We now turn to estimate the terms involving $G^{i},1\le i\le 4$.  Using Lemma \ref{gle} and the fact that $\pa$ only involves the horizontal spatial derivatives, we may bound
\begin{equation}\label{kg1}
\begin{split}
&\int_\Omega  \chi^2(\pa u \cdot \pa G^{1}+\pa p   \pa G^{2})
\\&\ \  \le   \norm{\pa u}_{1} \norm{\pa G^{1}}_{(\H1)^\ast}+\norm{\pa p}_{0}\norm{\pa G^{2}}_{0}
 \le   \norm{\pa u}_{1} \norm{ G^{1}}_{4N-1}+\norm{\pa p}_{0}\norm{ G^{2}}_{4N}
 \\&\ \ \ls   \sqrt{ \mathcal{D}}\left(  {\varepsilon_0}\sqrt{\mathcal{D}}+  P(1+ \mathfrak{E} )   \right)
\le      P(1+ \mathfrak{E} )\sqrt{\mathcal{D}}+C {\varepsilon_0} {\mathcal{D}} .
  \end{split}
\end{equation}
Similarly, we use Lemma \ref{gle} along with the trace embedding to find that
\begin{equation}\label{kg3}
\begin{split}
 \int_\Sigma -\pa u \cdot \pa G^{3} &\le    \snormspace{\pa u}{1/2}{\Sigma}  \norm{\pa G^{3}}_{-1/2}\ls    \norm{\pa u}_{1}   \norm{  G^{3}}_{4N-1/2}
  \\&\ls  \norm{\pa u}_{1}  \left(  {\varepsilon_0}\sqrt{ \mathcal{D}}+  P(1+ \mathfrak{E} )  \right)
\le      P (1+ \mathfrak{E} )\sqrt{\mathcal{D}}+C {\varepsilon_0} {\mathcal{D}} .
\end{split}
\end{equation}

Now we estimate the $G^4$ term.  For the first term for $G^4$, we use Lemma \ref{gle} to have
\begin{equation}\label{4alpha1}
\begin{split}
  \int_\Sigma   \pa \eta  \pa G^{4}   &\le     \norm{ \pa\eta}_{1/2}  \norm{ \pa G^{4}}_{-1/2}\le     \norm{ \eta}_{4N+1/2}  \norm{ \pa G^{4}}_{-1/2}
  \\& \ls   \sqrt{ \mathfrak{E}}
  \left(  {\varepsilon_0}\sqrt{\mathcal{D}}+  P(1+ \mathfrak{E} )   \right) \le
 P (1+ \mathfrak{E} )  + {\varepsilon_0} \mathcal{D} .
\end{split}
\end{equation}
For the other term, we need to do the splitting:
\begin{equation}
 \pa G^{4} =  \sum_{0 \le  \beta < \alpha} C_{\alpha,\beta}  \p^\beta \nab_\ast \eta \cdot \p^{\alpha-\beta} u +\p^\al \nab_\ast \eta \cdot   u:=I+\p^\al \nab_\ast \eta \cdot   u.
\end{equation}
Note that the highest derivatives appearing as $Y$ of $I$ are $\p^{\al}u$ and $\p^{\al-\beta}\nab_\ast\eta$ for some $\beta\in \mathbb{N}^2$ with $|\beta|=1$, and we have
\begin{equation}
 \norm{\p^{\al}u}_{H^{1/2}(\Sigma)}\ls \norm{ u}_{4N+1}\le \sqrt{\mathcal{D}} ,\ \norm{\p^{\al-\beta}\nab_\ast  \eta}_{ {1/2} } \le \norm{\eta}_{4N+1/2} \le  \sqrt{\mathfrak{E} } .
\end{equation}
Then arguing as in Lemma \ref{gle}, we may bound
\begin{equation}
\norm{I}_{1/2}
\ls\sqrt{ \mathcal{D}}\left(  {\varepsilon_0}\sqrt{ \mathcal{D}}+  P(1+ \mathfrak{E} )     \right).
\end{equation}
Hence, we have
\begin{equation}\label{4alpha2}
\begin{split}
 \int_\Sigma -\sigma\Delta_\ast\p^\al\eta I&\le  \sigma  \norm{\Delta_{\ast}\p^\al \eta}_{-1/2}\norm{I}_{1/2}
\le  \sigma  \norm{ \eta}_{4N+3/2}\norm{I}_{1/2}
  \\&\ls\sqrt{ \mathcal{D}}\left(  {\varepsilon_0}\sqrt{ \mathcal{D}}+  P(1+ \mathfrak{E} )     \right)
\le      P(1+ \mathfrak{E} )\sqrt{ \mathcal{D}} + {\varepsilon_0} \mathcal{D}
  .
\end{split}
\end{equation}
On the other hand, we integrate by parts to obtain, by Lemma \ref{i_sobolev_product_1},
\begin{equation}\label{4alpha7}
\begin{split}
& \int_\Sigma -\sigma\Delta_\ast \p^\al \eta \p^\al \nab_\ast \eta \cdot   u =   \int_\Sigma  \sigma   \p_i\p^\al \eta   ( \nab_\ast\p_i\p^\al\eta \cdot   u +\nab_\ast\p^\al\eta \cdot  \p_i u )
\\&\quad
= \int_\Sigma  -\sigma \diverge_\ast  u \frac{|\p_i\p^\al \eta |^2}{2}   +\sigma   \p_i\p^\al \eta  \nab_\ast\p^\al\eta \cdot  \p_i u
 \\&\quad
\le \norm{\nab_\ast u}_{H^2(\Sigma)}\norm{\nab_\ast\p^\al\eta}_{-1/2}\norm{\sigma \nab_\ast\p^\al\eta}_{1/2}\le  \mathfrak{E} \sqrt{\mathcal{D}} .
\end{split}
\end{equation}
Hence, in light of \eqref{4alpha1}, \eqref{4alpha2} and \eqref{4alpha7}, we may bound
\begin{equation}\label{kg4}
\int_\Sigma (\pa \eta-\sigma\Delta_\ast\pa \eta) \pa G^4 \le    P (1+ \mathfrak{E} )\sqrt{\mathcal{D}}  +C {\varepsilon_0}\mathcal{D} .
\end{equation}

Finally, using Cauchy's inequality, we have
\begin{equation}\label{kgk}
\int_\Sigma  (\pa \eta-\sigma\Delta_\ast\pa \eta)  \kappa\pa\Psi\le \kappa \ns{\pa \eta}_0+\frac{1}{2}\kappa\sigma\ns{\Delta_\ast\pa \eta}_0 + C\ns{\Psi}_{4N}.
\end{equation}

Now, in light of \eqref{p_u_e_2}--\eqref{kg3}, \eqref{kg4} and \eqref{kgk}, we deduce from \eqref{p_u_e_1} that
\begin{equation}\label{hihi}
\begin{split}
 &   \frac{d}{dt}\left( \int_\Omega \abs{\pa(\chi u)}^2  +    \int_\Sigma \abs{\pa \eta}^2 +  \sigma \abs{\pa\nab_\ast \eta}^2\right)
+    \int_\Omega \abs{\sg \pa (\chi u)}^2
 + \int_\Sigma {\kappa\sigma }  \abs{\Delta_\ast\pa \eta}^2
\\&\quad\le      P (1+\mathfrak{E} ) \sqrt{\mathcal{D}}  +C {\varepsilon_0}  \mathcal{D} +C\ns{ \Psi}_{4N}
\end{split}
\end{equation}
for all $\al\in \mathbb{N}^2$ with $|\al|\le 4N$. We then integrate \eqref{hihi} in time to have, by \eqref{2l_Fe2123},
\begin{equation}
\begin{split}
 \bar{\mathcal{E}}+\int_0^t \bar{\mathcal{D}} &\ls  P(\tilde{\mathcal{E}}_0)+\int_0^t P (1+ \mathfrak{E} )(1+\sqrt{\mathcal{D}})   + {\varepsilon_0} \mathcal{D} +\ns{ \Psi}_{4N}
 \\&\ls  P(\tilde{\mathcal{E}}_0)+ P (1+ \mathfrak{K}^{\kappa}) T^{1/2} + {\varepsilon_0}\int_0^t \mathcal{D}.
 \end{split}
\end{equation}
This completes the proof of the claim \eqref{0ese}.

Now we will show that $\mathcal{D}$ is comparable to $\bar{\mathcal{D}}$  so that we can improve the energy evolution obtained in \eqref{0ese}. We will employ the elliptic theory to improve the estimates of $(u,p)$. Notice that $G^3$ involves the high spatial derivative of $\eta$, say, $\nab_\ast^2\eta$, neither $\bar{\mathcal{E}}$ nor $\bar{\mathcal{D}}$ contain an enough $\kappa$-independent regularity estimate of $\eta$ that allows us to apply Lemma \ref{i_linear_elliptic}. But it is crucial to observe that we can get higher regularity estimates of $u$ on the boundary $\Sigma$ from $\bar{\mathcal{D}}$. Indeed, since $\Sigma $ is flat, we may use the definition of Sobolev norm on $\Sigma$, the trace
theory and Korn's inequality of Lemma \ref{i_korn} to obtain, by the definition of $\bar{\mathcal{D}}$,
\begin{equation}
\begin{split}
\norm{  u}_{H^{4N +1/2}(\Sigma)}^2&=\norm{ \chi u}_{H^{4N +1/2}(\Sigma)}^2 \lesssim \norm{ \chi u }_{L^2(\Sigma)}^2
+\norm{\nab_\ast^{4N } (\chi u)}_{H^{1/2}(\Sigma )}^2
 \\&\lesssim \norm{   \chi u }_{1}^2+\norm{\nab_\ast^{4N }   (\chi u)  }_{1}^2
 \le  \bar{\mathcal{D}} .
 \end{split}
\end{equation}
This motivates us to use Lemma \ref{i_linear_elliptic2}. Observe that $(u,p)$ solves the problem
\begin{equation}
\begin{cases}
 - \Delta u+\nabla p=-\partial_t u+G^1\quad&\text{in }\Omega
\\ \diverge{u}=G^2&\text{in }\Omega
 \\ u=u&\text{on }\Sigma
\\ u=0 &\text{on }\Sigma_b.
\end{cases}
\end{equation}
We then apply Lemma \ref{i_linear_elliptic2} with $r=4N+1$ to have, by Lemma \ref{gle},
\begin{equation}\label{0ese22}
\begin{split}
 \ns{u}_{4N+1}+\ns{  p}_{4N} &\ls \ns{ p}_{0}+\ns{\partial_t u}_{4N-1}+\ns{G^1}_{4N-1}+\ns{G^2}_{4N}+\norm{  u}_{H^{4N +1/2}(\Sigma)}^2
 \\&\le   P(1+ \mathfrak{E} )+{\varepsilon_0} \mathcal{D}+\ns{\partial_t u}_{4N-1}+\bar{\mathcal{D}}.
 \end{split}
\end{equation}
Combining \eqref{0ese} and \eqref{0ese22},  we obtain
\begin{equation}\label{0ese2122}
 \bar{\mathcal{E}}+\int_0^t \ns{   u  }_{4N+1}+\ns{   p  }_{4N}  +\kappa\sigma \ns{ \eta}_{4N +2} \ls P(\tilde{\mathcal{E}}_0)+ P (1+ \mathfrak{K}^{\kappa}) T^{1/2}  + {\varepsilon_0} \int_0^t \mathcal{D}.
\end{equation}

We now improve the estimates of $\eta$; we use the boundary condition
\begin{equation} \label{etaboundary1}
 \eta-\sigma\Delta_\ast  \eta =p-2\p_3u_3 - G_3^3 \text{ on } \Sigma.
\end{equation}
We use the standard elliptic theory on \eqref{etaboundary1} to obtain
 \begin{equation}\label{00eta0}
 \begin{split}
 \int_0^t\sigma^2\ns{   \eta}_{4N+3/2}+\ns{   \eta}_{4N-1/2}\le &\int_0^t
 \ns{  p}_{H^{4N-1/2}(\Sigma)}+\ns{\nab u}_{H^{4N-1/2}(\Sigma)} +\ns{G^3}_{ {4N-1/2} }
  \\\ls & \int_0^t
 \ns{ p}_{ {4N } } +  \ns{ u}_{ {4N+1} }+{\varepsilon_0} \mathcal{D}+ P (1+ \mathfrak{E} )
 \\\ls &  P (1+ \mathfrak{K}^{\kappa}) T^{1/2}+\int_0^t
 \ns{ p}_{ {4N } } +  \ns{ u}_{ {4N+1} } + {\varepsilon_0} \mathcal{D}   .
  \end{split}
\end{equation}

Now, in light of \eqref{0ese2122} and \eqref{00eta0}, we have
\begin{equation}
 \bar{\mathcal{E}}+\int_0^t  {\mathcal{D}} \ls P(\tilde{\mathcal{E}}_0)+ P (1+ \mathfrak{K}^{\kappa})T^{1/2} + {\varepsilon_0} \int_0^t \mathcal{D}.
\end{equation}
If we have made $\varepsilon_0$ be smaller (we can achieve this by assuming $\mathfrak{L}_0$ be smaller in Theorem \ref{l_knwp}), then we can conclude \eqref{upspatial} by further applying Lemma \ref{l_sobolev_infinity} to control the energy.
\end{proof}

\subsubsection{The limit as $\kappa\rightarrow0$: local well-posedness of \eqref{geometric} with $\sigma>0$}

Now we can produce a unique local strong solution to the original problem \eqref{geometric} for each fixed $\sigma>0$, which is the main result of this subsection. In this subsubsection, we will allow the generic polynomial $P$ to depend on $\sigma$.
\begin{thm}\label{l_sigmanwp}
Let $\sigma>0$. Assume that $u_0\in H^{4N}(\Omega),\eta_0\in H^{4N+1}(\Sigma)$ with $\mathfrak{L}_0\le \varepsilon_0/2$ for  the universal $\varepsilon_0>0$ from Theorem \ref{l_knwp}, and that the initial data $\dt^j u(0)$, etc are as constructed in Section \ref{l_data_section} and satisfy the $(2N)^{th}$  compatibility conditions \eqref{l_comp_cond_2N}. There exists a $0< T_0^\sigma\le 1$ such that there exists a unique solution triple $(u,p,\eta)$ to the  problem \eqref{geometric} on the time interval $[0,T_0^\sigma]$ that  achieves the initial data  and  satisfies
\begin{equation}\label{kindependent}
    \mathcal{K}_{2N}^1(u,p,\eta) \le P_\sigma\left(\ns{u_0}_{4N}+\ns{\eta_0}_{4N+1}\right) \text{ and } \mathfrak{L} (\eta) \le \varepsilon_0,
\end{equation}
where the energy functional $\mathcal{K}_{2N}^1$ is the one setting $\sigma=1$ in \eqref{energysigma}.
Moreover, $\eta$ is such that the mapping $\Phi(\cdot,t)$, defined by \eqref{mapping_def}, is a $C^{4N-3/2}$ diffeomorphism for each $t \in [0,T_0^\sigma]$.
\end{thm}
\begin{proof}
A suitable linear combination of the estimate \eqref{kestimate1} of Proposition \ref{k11}, the estimate \eqref{etaall} of Proposition \ref{k33} and the estimate \eqref{upspatial} of Proposition \ref{up2222} yields
\begin{equation}\label{ee123}
\begin{split}
    \mathfrak{K}^{\kappa}
&\le  \left(  P(\tilde{\mathcal{E}}_0)  + P(  T^{1/4} \mathfrak{K}^{\kappa})
  \right)  \exp\left( \left(  P(\tilde{\mathcal{E}}_0)  + P(  T^{1/4} \mathfrak{K}^{\kappa})
  \right) T \right)
 \\& \quad+  P(\tilde{\mathcal{E}}_0)+ P (1+ \mathfrak{K}^{\kappa})T^{1/2}  + P(  T^{1/4} \mathfrak{K}^{\kappa})
 \\& \le  \left(  P_1(\tilde{\mathcal{E}}_0)  + P (1+ \mathfrak{K}^{\kappa}) T^{1/4}
  \right)  \exp\left( \left(  P(\tilde{\mathcal{E}}_0)  + P(  T^{1/4} \mathfrak{K}^{\kappa})
  \right) T \right) .
   \end{split}
\end{equation}
Note that here we have transferred the $\sigma$ dependence of $\mathcal{E}$ and $\mathcal{D}$ to the polynomials $P$, which we have allowed to depend on $\sigma$. We name the polynomial $ P_1$ so that it can be referred later.

 We can now use a standard continuity argument  (cf. Section 9 of \cite{coutand_shkoller_1} for instance) to infer from \eqref{ee123} that there exists a $T_0^\sigma>0$ that depends on the initial data and $\sigma$ but does not depend on $\kappa$ so that for $0<T\le T_0^\sigma$, recovering the dependence of the solution on $\kappa$,
 \begin{equation}
 \mathfrak{K}^{\kappa}(u^\kappa,p^\kappa,\eta^\kappa)\le  2  P_1(\tilde{\mathcal{E}}_0)  .
\end{equation}
This $\kappa$-independent estimates in turn imply that $(u^\kappa,p^\kappa,\eta^\kappa)$ is indeed a solution of \eqref{kgeometric} on the time interval $[0,T_0^\sigma]$ and yield a strong convergence of $(u^\kappa,p^\kappa,\eta^\kappa)$ to a limit $(u,p,\eta)$, up to extraction of a subsequence, which is more than sufficient for us to pass to the limit as $\kappa\rightarrow0$ in the system \eqref{kgeometric} for each $t\in[0,T_0^\sigma]$. We then find that $(u,p,\eta)$ is a strong solution of \eqref{geometric} on $[0,T_0^\sigma]$ and satisfies the estimates
\begin{equation} \label{keses}
 \mathfrak{K}^{0}(u,p,\eta) \le P_\sigma\left(\ns{u_0}_{4N}+\ns{\eta_0}_{4N+1}\right)
\end{equation}
if we recover the dependence of $ P_1$ on $\sigma$, where the quantity $\mathfrak{K}^{0}$ is the one setting $\kappa=0$ in \eqref{dnk}. Moreover, if we restrict $T_0^\sigma$ to be smaller if necessary, then we have
\begin{equation}\label{11ere1}
 \mathfrak{L}(\eta ) \le   \mathfrak{L}_0   + \ns{\eta  }_{L^2H^{4N}}+\ns{\dt\eta  }_{L^2H^{4N-1}}
 \le    \frac{\varepsilon_0}{2}+ {T \mathfrak{K}^{0}}  \le      \varepsilon_0.
\end{equation}
The uniqueness of the solution to \eqref{geometric} with $\sigma>0$
satisfying \eqref{keses} follows as Theorem \ref{l_knwp}.

To complete the theorem, we need to improve the estimates \eqref{keses} to be \eqref{kindependent}. Comparing the definitions of $\mathfrak{K}^{0}$ (as defined by \eqref{dnk}) and $\mathcal{K}_{2N}^1$ (as defined by \eqref{energysigma}), it means to improve the estimates of $\eta$ with temporal derivatives. For $j=1,\dots,2N+1$, we use the transport equation $\dt \eta=F^4(u,\eta)$ to have
\begin{equation}\label{tt923}
  \ns{\dt^j \eta}_{L^2 H^{4N-2j+5/2}}  =\ns{\dt^{j-1} F^4}_{L^2 H^{4N-2j+5/2}}
  = \ns{\dt^{j-1} F^4}_{L^\infty H^{4N-2(j-1)+1/2} }.
\end{equation}
But as in Lemma \ref{l_iteration_estimates_1} we easily have
\begin{equation}
 \ns{\dt^{j-1} F^4}_{L^2 H^{4N-2(j-1)+1/2} } \ls  \mathfrak{K}^{0}+(\mathfrak{K}^{0})^2 .
 \end{equation}
These two imply, by  \eqref{keses},
 \begin{equation}\label{siginf}
\sum_{j=1}^{2N+1} \ns{\dt^j \eta}_{L^2 H^{4N-2j+5/2}}\ls \mathfrak{K}^{0}+(\mathfrak{K}^{0})^2\le P_\sigma\left(\ns{u_0}_{4N}+\ns{\eta_0}_{4N+1}\right).
  \end{equation}
This completes the improvement of the $L^2H^k$ estimates of $\dt^j\eta$ for $j=1,\dots,2N+1$. We may further apply Lemma \ref{l_sobolev_infinity} to improve the $L^\infty H^k$ estimates of $\dt^j\eta$ for $j=1,\dots,2N$. For $j=2N+1$, we shall use the geometric fact: $\dt^{2N+1}\eta=D_t^{2N}u\cdot \n$. Then by Lemma \ref{l_boundary_dual_estimate} and since $\diverge_\a (D_t^{2N}u)=0$, we have
\begin{equation}
 \ns{\dt^{2N+1}\eta}_{L^\infty H^{-1/2}}=\ns{D_t^{2N}u \cdot \n}_{L^\infty H^{-1/2}}  \ls \ns{D_t^{2N}u}_{L^\infty \mathcal{H}^0}\le P_\sigma\left(\ns{u_0}_{4N}+\ns{\eta_0}_{4N+1}\right).
\end{equation}
We thus conclude the estimate \eqref{kindependent} in view of the analysis above. The $C^{4N-3/2}$ diffeomorphism of $\Phi(\cdot,t)$ for each $t\in [0,T_0^\sigma]$ follows from the regularity of $\eta$ and the smallness \eqref{11ere1}.
\end{proof}

\subsection{Local zero surface tension limit of \eqref{geometric}}
In this subsection our aim is to pass to the limit as $\sigma\rightarrow0$ in the problem \eqref{geometric} within a local time interval, where $\sigma>0$ is the surface tension coefficient. We define a special quantity for this subsection as
\begin{equation}\label{ensigma}
\begin{split}
  \mathfrak{K}^{\sigma}=\mathfrak{K}^{\sigma}(u,p,\eta):&
  =\sup_{0\le t\le T}\mathfrak{E}^{\sigma}(t) +\int_0^T\mathfrak{D}^{\sigma}(t)\,dt+ \ns{\dt^{2N+1} u}_{(\x_T)^*},
 \end{split}
\end{equation}
where
\begin{equation}
\mathfrak{E}^{\sigma}= \sum_{j=0}^{2N} \ns{\dt^j u}_{ {4N-2j }} + \sum_{j=0}^{2N-1}\ns{\dt^j p}_{ {4N-2j-1}}
+\sigma\ns{ \eta}_{ {4N+1}}+ \sum_{j=0}^{2N} \ns{\dt^j \eta}_{ {4N-2j+1/2}}
\end{equation}
and
\begin{equation}
\begin{split}
\mathfrak{D}^{\sigma}= &\sum_{j=0}^{2N} \ns{\dt^j u}_{ {4N-2j +1}} + \sum_{j=0}^{2N-1}\ns{\dt^j p}_{{4N-2j}}
\\&  + \sigma^2\ns{ \eta}_{ {4N+3/2}}+\ns{ \eta}_{ {4N+1/2}}+  \sum_{j=1}^{2N+1}   \ns{\dt^j \eta}_{ {4N-2j+3/2}}.
\end{split}
\end{equation}
We note that by the definition,
\begin{equation}\label{viewc}
 \mathfrak{K}^{\sigma}=\mathfrak{K}(u,p)+\mathfrak{F}(\eta)+ \ns{ \eta}_{L^\infty H^{4N+1/2}}+\sigma\ns{ \eta}_{L^\infty H^{4N+1}}+\sigma^2\ns{ \eta}_{L^2 H^{4N+3/2}},
\end{equation}
where $\mathfrak{K}(u,p)$, $\mathfrak{F}(\eta)$ are defined by \eqref{l_DEfrak_def} and \eqref{l_Kfrak_def}, respectively.

For each $\sigma>0$, we let $(u,p,\eta)$ (dropping the
$\sigma$-dependence temporally) be the solution to \eqref{geometric}
on $[0,T_0^\sigma]$ produced by Theorem \ref{l_sigmanwp}. Moreover,
Theorem \ref{l_sigmanwp} implies that for $0<T\le T_0^\sigma$, we
have $\mathfrak{K}^{\sigma}<\infty$ and the following smallness of
$\eta$:
\begin{equation}\label{smallsmallsigma}
  \mathfrak{L}(\eta):=\sup_{0\le t\le T} \ns{ \eta(t)}_{4N-1/2}\le \varepsilon_0
\end{equation}
for the universal constant $\varepsilon_0\ll1$. Our goal is to
derive the $\sigma$-independent estimates of the solution on
$[0,T_0^\sigma]$, and the key step is to derive the uniform bound of
$\mathfrak{K}^{\sigma}$.  The smallness \eqref{smallsmallsigma} is
crucial for closing the $\sigma$-independent energy estimates again
since we can not get a power of $T$ in front of some terms stemming
from the mean curvature.

We recall from the construction of the initial data in Section
\ref{l_data_section} that
\begin{equation}\label{2l_Fe211f}
\mathfrak{E}^\sigma(0) \le\mathcal{K}_{2N}^\sigma(0):=\mathcal{E}_{2N}^\sigma(0)+\f(0) \le P ( \mathcal{E}_0^\sigma ).
 \end{equation}
Here we recall that $\mathcal{E}_0^\sigma$, $\mathcal{E}_{2N}^\sigma(0)$, $\f(0)$ and $\mathcal{K}_{2N}^\sigma(0)$ are as defined by \eqref{00eta}, \eqref{p_energy_def}, \eqref{fff} and \eqref{energysigma}, respectively.

\subsubsection{The $\sigma$-independent transport estimates}
We shall first use the transport equation to derive the estimates for $\eta$:
\begin{equation}
 \dt \eta + u \cdot \nab_\ast \eta = u_3   \text{ in } \Sigma
\end{equation}
\begin{prop}\label{proe1}
We have the following estimates for $\eta$:
\begin{equation}\label{2etaes1}
 \ns{\eta}_{L^\infty H^{4N+1/2}} \le \exp\left( C T^{1/2}\sqrt{\mathfrak{K}^\sigma }  \right) \left( \ns{\eta_0}_{4N+1/2} + T {\mathfrak{K}^\sigma }\right)
\end{equation}
and
\begin{equation}\label{2etaes2}
 \mathfrak{F}(\eta)\le   P(\mathcal{E}_0^\sigma)+T^{1/4}  P(\mathfrak{K}^\sigma ).
\end{equation}

\end{prop}
\begin{proof}
Basing on the transport theory of Lemma \ref{i_sobolev_transport}, we may obtain
\begin{equation}\label{l_tt_1}
 \norm{\eta}_{L^\infty H^{4N+1/2}} \le \exp\left( C \int_0^T \norm{u(t)}_{H^{4N+1/2}(\Sigma)}dt  \right) \left( \sqrt{\mathcal{F}_0} + \int_0^T \norm{u_3(t)}_{H^{4N+1/2}(\Sigma)}dt \right).
\end{equation}
The trace theory and the Cauchy-Schwarz inequality implies
\begin{equation}\label{l_tt_12}
 \int_0^T \norm{u(t)}_{H^{4N+1/2}(\Sigma)}dt \ls \int_0^T \norm{u(t)}_{4N+1}dt \ls {T}^{1/2} \sqrt{\mathfrak{K}^\sigma }.
\end{equation}
Then the estimate \eqref{2etaes1} easily follows from \eqref{l_tt_1} and \eqref{l_tt_12}.

To estimate $\mathfrak{F}(\eta)$, we first trivially have
\begin{equation}\label{l_tt_8}
 \ns{\eta}_{L^2 H^{4N+1/2}} \le  T\ns{\eta}_{L^\infty H^{4N+1/2}}  \le T {\mathfrak{K}^\sigma }.
\end{equation}
For $j=1,\dots,2N+1$, we use the transport equation directly as $\dt\eta=u\cdot\n$ to have
\begin{equation}\label{tt9}
  \ns{\dt^j \eta}_{L^2 H^{4N-2j+3/2}} =\ns{\dt^{j-1}(u\cdot\n)}_{L^2 H^{4N-2j+3/2}}= \ns{\dt^{j-1}(u\cdot\n)}_{L^2H^{4N-2(j-1)-1/2}} .
\end{equation}
By the trace theory and the Sobolev interpolation, we obtain that for $\ell=0,\dots,2N$,
\begin{equation}\label{l_tt_1100}
\begin{split}
\ns{\dt^\ell u_3}_{L^2H^{4N-2\ell-1/2}(\Sigma)}&\ls \ns{\dt^\ell u_3}_{L^2H^{4N-2\ell+3/4}}
\le  T^{1/4}\norm{\dt^\ell u }_{L^\infty H^{4N-2\ell}}^{1/2}  \norm{\dt^\ell u }_{L^2 H^{4N-2\ell+1}}^{3/2}
\\&\le T^{1/4}{\mathfrak{K}^\sigma }.
 \end{split}
\end{equation}
Trivially again, for $\ell=0,\dots,2N$,
\begin{equation}\label{l_tt_1111}
\ns{\dt^\ell \eta}_{L^2H^{4N-2\ell+1/2}} \le T\ns{\dt^\ell \eta}_{L^\infty H^{4N-2\ell+1/2}} \le T {\mathfrak{K}^\sigma }.
\end{equation}
Hence, we deduce from \eqref{l_tt_1100} and \eqref{l_tt_1111} that for $j=1,\dots,2N-1$
\begin{equation}
\begin{split}
 \ns{\dt^{j-1}(u\cdot\n)}_{L^2H^{4N-2(j-1)-1/2}} &\ls \left(\sum_{\ell=0}^{j-1}\ns{\dt^\ell u}_{L^2H^{4N-2\ell-1/2}(\Sigma)}+\ns{\dt^\ell \eta}_{L^2H^{4N-2\ell+1/2}}\right) \mathfrak{K}^\sigma
 \\& \ls \left(T^{1/4}{\mathfrak{K}^\sigma } +T {\mathfrak{K}^\sigma } \right) \mathfrak{K}^\sigma \ls T^{1/4} ({\mathfrak{K}^\sigma })^2.
 \end{split}
 \end{equation}
This together with \eqref{tt9} implies
\begin{equation}\label{tt91}
  \ns{\dt^j \eta}_{L^2 H^{4N-2j+3/2}}  \ls  T^{1/4}{\mathfrak{K}^\sigma }+T^{1/4} ({\mathfrak{K}^\sigma })^2.
\end{equation}
We thus conclude \eqref{2etaes2} by further applying Lemma \ref{l_sobolev_infinity} to control the $L^2H^k$ norms in $\mathfrak{F}(\eta)$.
\end{proof}

\subsubsection{The remaining $\sigma$-independent estimates}
To derive the remaining $\sigma$-independent estimates for the solution, our strategy follows Sections \ref{ksec}--\ref{ksec2}. We first record an estimate of the nonlinear terms $\mathfrak{F}_\ast(F^1,F^3)$ (defined by \eqref{2l_Ffrak_def}) in terms of $\mathfrak{K}^{\sigma}$ (defined by \eqref{ensigma}).
\begin{lem}\label{2l_Ffrak_boundsigma}
Let $F^1(u,\eta),F^3(\eta)$ be defined by \eqref{l_F_forcing_def}. Then
 \begin{equation}\label{fast222}
 \mathfrak{F}_\ast(F^1(u,\eta),F^3(\eta)) \le  P(\mathcal{E}_0^\sigma)+ T^{1/4} P( \mathfrak{K}^\sigma).
\end{equation}
 \end{lem}
 \begin{proof}
We also recall the estimate \eqref{jaja00} from the proof of Lemma \ref{2l_Ffrak_bound} that
\begin{equation} \label{jaja002}
 \mathfrak{F}_\ast(F^1(u,\eta),F^3(\eta)) \le  P(1+ \mathfrak{F}(\eta))\left(   \mathfrak{F} (\eta) +(\mathfrak{F}_{2N}(u))^2\right)
\end{equation}
We recall from the estimate \eqref{jaja2} from the proof of Lemma \ref{2l_Ffrak_bound} that, since $\mathfrak{K}_{2N}(u)\le  \mathfrak{K}^\sigma$,
 \begin{equation}\label{jaja222}
 \mathfrak{F}_{2N}(u )  \le \sum_{j=0}^{2N-1}\ns{\dt^j u(0)}_{ {4N-2j-1/4}}+C T^{1/4}\mathfrak{K}_{2N}(u) \le  P(\mathcal{E}_0^\sigma)+ 2 T^{1/4} \mathfrak{K}^\sigma .
 \end{equation}
Plugging the estimates \eqref{jaja222} and \eqref{2etaes2} into \eqref{jaja002}, we have
\begin{equation}
\begin{split}
 \mathfrak{F}_\ast(F^1(u,\eta),F^3(\eta)) &\le  P(1+ P(\mathcal{E}_0^\sigma)+T^{1/4} P( \mathfrak{K}^\sigma)) \left( P(\mathcal{E}_0^\sigma)+T^{1/4} P( \mathfrak{K}^\sigma)+
( P(\mathcal{E}_0^\sigma)+ 2 T^{1/4} \mathfrak{K}^\sigma )^2\right)
\\&\le   P(\mathcal{E}_0^\sigma)+ T^{1/4} P( \mathfrak{K}^\sigma ).
 \end{split}
\end{equation}
We thus conclude the lemma.
\end{proof}

For the nonlinear terms $G^i$ (defined by \eqref{G1_def}--\eqref{G4_def}), we have the following estimates.
\begin{lem}\label{gle22}
Let $\varepsilon_0$ be the small universal constant in \eqref{smallsmall}. Then we have
\begin{equation}\label{Gesti22}
\begin{split}
 \norm{ G^{1}}_{4N-1}+\norm{ G^{2}}_{4N} +\norm{ G^{3}}_{4N-1/2} +\norm{ G^{4}}_{4N-1/2}\ls   {\varepsilon_0}\sqrt{  \mathcal{D}}+ P(1+\mathfrak{E}^{\sigma}).
  \end{split}
\end{equation}
\end{lem}
\begin{proof}
The estimate follows in the same way as Lemma \ref{gle}.
\end{proof}

\begin{prop}\label{proe2}
Suppose that $(u,p,\eta)$ is the solution to \eqref{geometric}. Then
\begin{equation}\label{sigmaestimate1}
\begin{split}
 &\mathfrak{K}(u,p)+   \sigma\ns{ \eta}_{L^\infty H^{4N+1}}+\sigma^2\ns{ \eta}_{L^2H^{4N+3/2}}
 \\&\quad\le  \left(  P(\mathcal{E}_0^\sigma)  + P (1+ \mathfrak{K}^\sigma ) T^{1/4}
  \right)  \exp\left( \left(  P(\mathcal{E}_0^\sigma)  + T^{1/4} P( \mathfrak{K}^\sigma )
  \right) T \right).
  \end{split}
  \end{equation}
\end{prop}
\begin{proof}
We recall the estimates for $(\dt^ju,\dt^jp)$ for $j\ge 1$  from \eqref{upes} that, using the estimates \eqref{fast222} of Lemma \ref{2l_Ffrak_boundsigma}, \eqref{2l_Fe211f} and \eqref{upast},
\begin{equation}\label{11111}
\begin{split}
   & \sum_{j=1}^{2N} \ns{\dt^j u}_{L^2 H^{4N-2j +1}} + \ns{\dt^{2N+1} u}_{(\x_T)^*} + \sum_{j=1}^{2N-1}\ns{\dt^j p}_{ L^2 H^{4N-2j}}\\
 &\quad+\sum_{j=1}^{2N} \ns{\dt^j u}_{L^\infty H^{4N-2j }} + \sum_{j=1}^{2N-1}\ns{\dt^j p}_{L^\infty H^{4N-2j-1}}
\\&\qquad\le   P(1+ \mathfrak{F}(\eta)) \exp\left( P(1+ \mathfrak{F}(\eta)) T \right)
\left(  P(\mathcal{E}_0^\sigma) +\mathfrak{F}_\ast(F^1,F^3) +
\mathfrak{F}(u,p)
  \right)
  \\&\qquad\le  P(1+  P(\mathcal{E}_0^\sigma)+ T^{1/4} P( \mathfrak{K}^\sigma )) \exp\left( P(1+  P(\mathcal{E}_0^\sigma)+ T^{1/4} P( \mathfrak{K}^\sigma )) T \right)
  \\&\qquad\quad\times  \left(  P(\mathcal{E}_0^\sigma)  + P(\mathcal{E}_0^\sigma)+ T^{1/4} P( \mathfrak{K}^\sigma ) +
T \mathfrak{K}^\sigma
  \right)
  \\&\qquad\le \left(  P(\mathcal{E}_0^\sigma)+ T^{1/4} P( \mathfrak{K}^\sigma )
  \right)  \exp\left( P(1+ \mathcal{E}_0^\sigma+  T^{1/4} P( \mathfrak{K}^\sigma )) T \right)
  \\&\qquad\le \left(  P(\mathcal{E}_0^\sigma)  +T^{1/4} P( \mathfrak{K}^\sigma )
  \right)  \exp\left( \left(  P(\mathcal{E}_0^\sigma)  +T^{1/4} P( \mathfrak{K}^\sigma )
  \right) T \right).
   \end{split}
\end{equation}
In view of the estimate \eqref{Gesti22} of  Lemma \ref{gle22} and  Proposition \ref{up2222} (setting $\kappa=0$), we obtain
\begin{equation}\label{33333}
\begin{split}
&\ns{u}_{L^2 H^{4N+1}} +\ns{  p}_{L^2 H^{4N}}+\ns{u}_{L^\infty H^{4N}}+\ns{  p}_{L^\infty H^{4N-1}}
\\&\quad+\sigma^2\ns{ \eta}_{L^2 H^{4N+3/2}}+  \ns{ \eta}_{L^2 H^{4N -1/2}} +\sigma \ns{ \eta}_{L^\infty H^{4N +1}}+  \ns{ \eta}_{L^\infty H^{4N }}
\\&\qquad\ls P(\mathcal{E}_0^\sigma)+ P (1+  \mathfrak{K}^\sigma ) T^{1/2}  +
\ns{\partial_t u}_{L^2 H^{4N-1}}+\ns{\partial_t p}_{L^2 H^{4N-2}}.
\end{split}
\end{equation}
 Suitably summing up the estimate \eqref{11111}--\eqref{33333}, we get
\begin{equation}
\begin{split}
   & \mathfrak{K}(u,p)
+\sigma^2\ns{ \eta}_{L^2 H^{4N+3/2}}+  \ns{ \eta}_{L^2 H^{4N -1/2}} +\sigma \ns{ \eta}_{L^\infty H^{4N +1}}+  \ns{ \eta}_{L^\infty H^{4N }}
\\&  \le  \left(  P(\mathcal{E}_0^\sigma)  +T^{1/4} P( \mathfrak{K}^\sigma )
  \right)  \exp\left( \left(  P(\mathcal{E}_0^\sigma)  +T^{1/4} P( \mathfrak{K}^\sigma )
  \right) T \right)
 +  P(\mathcal{E}_0^\sigma)+ P (1+ \mathfrak{K}^\sigma ) T^{1/2}
 \\&  \le  \left(  P(\mathcal{E}_0^\sigma)  + P (1+ \mathfrak{K}^\sigma ) T^{1/4}
  \right)  \exp\left( \left(  P(\mathcal{E}_0^\sigma)  + T^{1/4} P( \mathfrak{K}^\sigma )
  \right) T \right)  .
   \end{split}
\end{equation}
We thus conclude the proposition.
\end{proof}

\subsubsection{The limit as $\sigma\rightarrow0$: the proof of Theorem \ref{lwp}}

We can now prove a $\sigma$-independent local well-posedness of \eqref{geometric} which allow us to pass to the limit as $\sigma\rightarrow0$ in the local time interval.
 \begin{proof}[Proof of Theorem \ref{lwp}]
 Assume that the initial data $ u_0^\sigma$ and $\eta_0^\sigma$ satisfy the assumptions of Theorem \ref{lwp} for $0<\sigma\le 1$. For each $\sigma$, we let $(u^\sigma,p^\sigma,\eta^\sigma)$ be the solution to \eqref{geometric} on $[0,T_0^\sigma]$ produced by Theorem \ref{l_sigmanwp}.
By the estimate \eqref{2etaes1}--\eqref{2etaes2} of Proposition \ref{proe1} and the estimate \eqref{sigmaestimate1} of Proposition \ref{proe2}, in light of \eqref{viewc},
we conclude that
 \begin{equation}\label{ttts}
 \begin{split}
 \mathfrak{K}^\sigma &\ls  \left(  P(\mathcal{E}_0^\sigma)  + P (1+ \mathfrak{K}^\sigma ) T^{1/4}
  \right)  \exp\left( \left(  P(\mathcal{E}_0^\sigma)  +T^{1/4}  P( \mathfrak{K}^\sigma )
  \right) T \right)
  \\&\quad+\exp\left( C T^{1/2}\sqrt{\mathfrak{K}^\sigma }  \right) \left( \ns{\eta_0}_{4N+1/2} + T {\mathfrak{K}^\sigma }\right)
  \\&\le  \left(  P_1(\mathcal{E}_0^\sigma)  + P (1+\mathfrak{K}^\sigma ) T^{1/4}
  \right)  \exp\left( \left(  P(\mathcal{E}_0^\sigma)  +T^{1/4} P(\mathfrak{K}^\sigma )
  \right) T \right) .
  \end{split}
\end{equation}
Here we name the polynomial $ P_1$ so that it can be referred later.

Since $\mathcal{E}_0^\sigma(u_0^\sigma,
\eta_0^\sigma)\rightarrow\mathcal{E}_0^0(u_0, \eta_0)$ as
$\sigma\rightarrow0$, we can now use a standard continuity argument
to infer from \eqref{ttts} that there exists a $ {T}_0>0$ that
depends only on the initial data but does not depend $\sigma$ so
that for $0<T\le T_0$
 \begin{equation}\label{whole}
 \mathfrak{K}^\sigma (u^\sigma ,p^\sigma ,\eta^\sigma )\le  2  P_1(\mathcal{E}_0^\sigma)= P(\mathcal{E}_0^\sigma).
\end{equation}
This $\sigma$-independent estimates in turn imply that $(u^\sigma,p^\sigma,\eta^\sigma)$ is indeed a solution of \eqref{geometric} on the time interval $[0,T_0]$.
Moreover, we can improve the estimate \eqref{whole} to be \eqref{zero_1} by improving the estimates of $\eta$ with temporal derivatives as in Theorem \ref{l_sigmanwp}.

The $\sigma$-independent estimates \eqref{zero_1} yield  a strong
convergence of $(u^\sigma,p^\sigma,\eta^\sigma)$ to a limit
$(u,p,\eta)$ in any functional space that $\mathcal{K}_{2N}^0$ can
compactly embed into, up to extraction of a subsequence, which is
more than sufficient for us to pass to the limit as
$\sigma\rightarrow0$ in \eqref{geometric} for each $t\in[0,T_0]$. We
find that $(u,p,\eta)$ is a strong solution of \eqref{geometric} for
$\sigma=0$ on $[0,T_0]$ and satisfies the estimates
\eqref{zero_123}. It can be shown as in Theorem \ref{l_knwp} that
the solution of \eqref{geometric} with $\sigma=0$ satisfying
\eqref{zero_123} is unique. This in turn implies that indeed the
whole family of the solutions $(u^\sigma,p^\sigma,\eta^\sigma)$
converges to the unique limit $(u,p,\eta)$. We thus conclude the
zero surface tension limit within the local time interval $[0,T_0]$
and complete the proof of Theorem \ref{lwp}.
\end{proof}

\section{Global in time theory}\label{sec global}

The aim of this section is to prove a $\sigma$-independent global well-posedness of the problem \eqref{geometric} and hence justify the global-in-time zero surface tension limit in \eqref{geometric}. Note that Theorem \ref{lwp} in particular implies that if $\mathcal{E}_0^\sigma\le \delta_0\ll1$ ($\delta_0$ will be determined later but is smaller than $\varepsilon_0/2$ in Theorem \ref{lwp}), then there exists a unique solution $(u,p,\eta)$ (dropping the $\sigma$ dependence) to \eqref{geometric} on $[0,T_0]$ for $T_0=T_0(\mathcal{E}_0^\sigma)>0$. Moreover, the solution satisfies the estimates
\begin{equation}
\mathcal{G}_{2N}^\sigma(T_0)\ls \mathcal{K}_{2N}^\sigma(T_0)
\le   P ( \mathcal{E}_0^\sigma)\ls  \mathcal{E}_0^\sigma.
\end{equation}
Hence for $\delta>0$ (sufficiently small that will be determined later), if $\mathcal{E}_0^\sigma$ is sufficiently small, then
\begin{equation}
\mathcal{G}_{2N}^\sigma(T_0) \le  \delta.
\end{equation}
This smallness will be used to derive the global energy estimates throughout the rest of this section. Since $\eta_0^\sigma$ satisfies the zero average condition \eqref{z_avg}, we have
\begin{equation}\label{avege}
\int_\Sigma \eta(t)=0\text{ for }t\ge 0.
\end{equation}
Indeed, from the second and fourth equations of \eqref{geometric}, we obtain
\begin{equation}
\frac{d}{dt} \int_\Sigma \eta=\int_\Sigma \partial_t\eta =\int_\Sigma u\cdot\n =\int_\Omega \diverge_\a u=0.
\end{equation}
The zero average \eqref{avege} will allow us to use Poincar\'e's inequality on $\Sigma$.

\subsection{Preliminaries}

To derive the global energy evolution of the pure temporal derivatives of the solution to \eqref{geometric} we shall use the following geometric formulation. Applying the temporal differential operator $\dt^j$ for $j=0,\dots,2N$ to \eqref{geometric}, we find that
\begin{equation}\label{linear_geometric}
\begin{cases}
  \dt (\dt^j u) - \dt \bar{\eta} \tilde{b} K \p_3 (\dt^j u) + u \cdot \naba (\dt^j u) + \diva S_{\a}(\dt^j p,\dt^j u) = F^{1,j}  & \text{in } \Omega \\
 \diva (\dt^j u) = F^{2,j}  & \text{in } \Omega \\
 S_{\a}(\dt^j p,\dt^j u) \n = (\dt^j \eta-\sigma\Delta_\ast \dt^j \eta) \n + F^{3,j} & \text{on } \Sigma \\
 \dt (\dt^j \eta)   = \dt^j u\cdot \n+F^{4,j} & \text{on } \Sigma\\
\dt^j u =0 & \text{on } \Sigma_b,
\end{cases}
\end{equation}
where
\begin{equation}\label{F_def_start}
\begin{split}
& F_i^{1,j}  = \sum_{0 < \ell \le j }  C_j^\ell\left\{  \dt^\ell ( \dt \bar{\eta} \tilde{b} K)      \dt^{j - \ell} \p_3 u_i+
\mathcal{A}_{lk} \p_k (\dt^{ \ell} \mathcal{A}_{im} \dt^{j - \ell} \p_m u_l  + \dt^{\ell} \mathcal{A}_{lm} \dt^{j - \ell} \p_m u_i)
  \right.
\\
 &\left.\qquad\qquad +   \dt^{ \ell}\mathcal{A}_{lk}\dt^{j - \ell}\p_k (\mathcal{A}_{im} \p_m u_l + \mathcal{A}_{lm}\p_m u_i)- \dt^{\ell} ( u_l  \mathcal{A}_{lk} ) \dt^{j - \ell} \p_k u_i -  \dt^{\ell} \mathcal{A}_{ik} \dt^{j - \ell} \p_k p\right\},
 \end{split}
\end{equation}
 \begin{equation}\label{i_F2_def}
 F^{2,j} = - \sum_{0 < \ell \le j}  C_j^\ell\dt^{\ell} \mathcal{A}_{lk} \dt^{j - \ell} \p_k u_l
,
\end{equation}
\begin{equation}\label{i_F3_def}
\begin{split}
 F_i^{3,j} =& \sum_{0 < \ell \le j} C_j^\ell \left\{ \dt^{\ell} \n_i \dt^{j - \ell}( \eta -  p)  +
  \dt^{ \ell} ( \n_l \mathcal{A}_{ik} ) \dt^{j - \ell} \p_k u_l + \dt^{ \ell} ( \n_l \mathcal{A}_{lk} ) \dt^{j - \ell} \p_k u_i\right\}
\\
& -\sigma\sum_{0 < \ell \le j} C_j^\ell  \dt^{ \ell} \n \dt^{j -
\ell}   \Delta_\ast\eta  -\sigma\dt^{j
}\left(\diverge_\ast(((1+|\nab_\ast\eta|^2)^{-1/2}-1)\nab_\ast\eta)\n\right),
 \end{split}
 \end{equation}
 and
\begin{equation}\label{F_def_end}
 F^{4,j} =  \sum_{0 < \ell \le j}  C_j^\ell \dt^{ \ell} \n\cdot \dt^{j - \ell}  u.
\end{equation}

We shall present the estimates of these nonlinear terms $F^{1,j},F^{2,j},F^{3,j}$ and $F^{4,j}$ in the following lemma, at both $2N$ and $N+2$ levels.

\begin{lem}\label{p_F_estimates}
Let $F^{i,j}$ be defined by \eqref{F_def_start}--\eqref{F_def_end}, then the following estimates hold.
\begin{enumerate}
\item For $0\le j\le 2N$, we have
\begin{equation}\label{p_F_e_01}
 \ns{F^{1,j} }_{0}+ \ns{  F^{2,j} }_{0}  + \ns{\dt (J F^{2,j} ) }_{0} + \ns{F^{3,j}}_{0} + \norm{F^{4,j}}_{0} \ls \se{2N}^0 \sd{2N}^0
\end{equation}
and
\begin{equation}\label{p_F_e_02}
 \ns{F^{2,j}}_{0} \ls (\se{2N}^0)^2.
\end{equation}
\item For $0\le j\le N+2$, we have
\begin{equation}\label{p_F_e_h_01}
 \ns{F^{1,j} }_{0}+ \ns{  F^{2,j} }_{0}  + \ns{\dt (J F^{2,j} ) }_{0} + \ns{F^{3,j}}_{0} + \norm{F^{4,j}}_{0} \ls \se{2N}^0  \sd{N+2}^0
\end{equation}
and
\begin{equation}\label{p_F_e_h_02}
 \ns{F^{2,j}}_{0}\ls \se{2N}^0   \se{N+2}^0 .
\end{equation}
\end{enumerate}
\end{lem}
\begin{proof}
All these estimates, except the estimates of the $\sigma$-related
terms in the definition \eqref{i_F3_def} of $F^{3,j}$, are recorded
in Theorems 4.1--4.2 of \cite{GT_per}. The highest derivative
appearing in these $\sigma$-related terms is
$\dt^{j}\nab_\ast^2\eta$. Since $\norm{\dt^{j}\nab_\ast^2\eta}_0\le
\sqrt{\sd{j}^0}$, we can deduce \eqref{p_F_e_01} and
\eqref{p_F_e_h_01}.
\end{proof}

To derive the global energy evolution of the mixed time-horizontal derivatives of the solution to \eqref{geometric} we shall use the following perturbed formulation, i.e., setting $\kappa=0$ in \eqref{perturb2},
\begin{equation}\label{perturb f}
\begin{cases}
 \partial_t u- \Delta u+\nabla p=G^1\quad&\text{in }\Omega
\\ \diverge{u}=G^2&\text{in }\Omega
\\ ( pI-\mathbb{D}(u)) e_3= (\eta -\sigma \Delta_\ast\eta )e_3+G^3&\text{on }\Sigma
\\ \partial_t\eta -u_3= G^4&\text{on }\Sigma
\\ u=0 &\text{on }\Sigma_b.
\end{cases}
\end{equation}
We shall again need also the localized equations of \eqref{perturb f}, i.e.,
setting $\kappa=0$ in \eqref{p_localized_equations},
\begin{equation}\label{perturb f local}
 \begin{cases}
  \dt (\chi  u) - \Delta (\chi  u)+ \nab (\chi  p) = \chi  G^1 + H^{1 } & \text{in }\Omega \\
  \diverge(\chi  u) = \chi  G^2 + H^{2 } & \text{in }\Omega \\
  ((\chi  p) I - \sg (\chi  u) )e_3 =   (\eta -\sigma \Delta_\ast\eta ) e_3 + G^3   & \text{on }\Sigma \\
   \dt \eta  - (\chi  u_3) = G^4  &\text{on } \Sigma \\
  \chi  u =0 & \text{on }\Sigma_b,
 \end{cases}
\end{equation}
where we recall $\chi$ and $H^1,H^2$ from \eqref{chi_properties} and
\eqref{p_H_def}, respectively.

 We shall present the estimates of the nonlinear terms $G^1,G^2,G^3$ and $G^4$ in the following lemma, at both $2N$ and $N+2$ levels. For this, we define one more specialized energy term by
\begin{equation}
 \k := \pnorm{\nab u}{\infty}^2 + \pnorm{\nab^2 u}{\infty}^2+  \snormspace{\nab_\ast u }{2}{\Sigma}^2.
\end{equation}
Note that $\k\ls \mathcal{E}_{N+2}$.
\begin{lem}\label{p_G_estimates}
Let $G^i$ be defined by \eqref{G1_def}--\eqref{G4_def}. Then
\begin{enumerate}
\item It holds that
\begin{equation}\label{p_G_e_0}
\begin{split}
& \ns{ \bar{\nab}_0^{4N-2} G^1}_{0}  +  \ns{ \bar{\nab}_0^{4N-2}  G^2}_{1} +
 \ns{ \bar{\nab}_{\ast 0}^{\  4N-2}  G^3}_{1/2}
 \\&\quad+
 \ns{\bar{\nab}_{\ast 0}^{\  4N-1} G^4}_{1/2} \ls(\se{2N}^\sigma)^{2} + \k \f,
 \end{split}
\end{equation}
and
\begin{equation}\label{p_G_e_00}
\begin{split}
&\ns{ \bar{\nab}_{0}^{4N-1} G^1}_{0}  +  \ns{ \bar{\nab}_0^{4N-1}  G^2}_{1}  +
 \ns{ \bar{\nab}_{\ast 0}^{\  4N-1} G^3}_{1/2} \\&\quad + \ns{\bar{\nab}_{\ast 0}^{\  4N-1} G^4}_{1/2}
   + \ns{\bar{\nab}_{\ast }^{\, 4N-2} \dt G^4}_{1/2}+\sigma^2\ns{\nab_{\ast }^{\,  4N} G^4}_{1/2}
\ls \se{2N}^\sigma \sd{2N}^\sigma + \k \f.
\end{split}
\end{equation}
 \item It holds that
\begin{equation}\label{p_G_e_h_0}
\begin{split}
&\ns{ \bar{\nab}_0^{2(N+2)-2} G^1}_{0} +  \ns{ \bar{\nab}_0^{2(N+2)-2}  G^2}_{1} +
 \ns{ \bar{\nab}_{\ast 0}^{\  2(N+2)-2} G^3}_{1/2}
\\&\quad+ \ns{\bar{\nab}_{\ast 0}^{\  2(N+2)-1} G^4}_{1/2}\ls \se{2N}^\sigma\se{N+2}^\sigma
\end{split}
\end{equation}
and
\begin{equation}\label{p_G_e_h_00}
\begin{split}
&\ns{ \bar{\nab}_0^{2(N+2)-1} G^1}_{0} +  \ns{ \bar{\nab}_0^{2(N+2)-1}  G^2}_{1}  +
 \ns{ \bar{\nab}_{\ast 0}^{\  2(N+2)-1} G^3}_{1/2}
\\&\quad + \ns{\bar{\nab}_{\ast 0}^{\  2(N+2) } G^4}_{1/2}
\ls
\se{2N}^\sigma \sd{N+2}^\sigma.
\end{split}
\end{equation}
\end{enumerate}
\end{lem}
\begin{proof}
These estimates are recorded with the slight differences in Theorems 3.1--3.2 of \cite{GT_per}. These differences result from the estimates of the $G^3$ and $G^4$ terms. For the $G^3$ term, the difference results from estimating the $\sigma$-related terms in the definition \eqref{G3_def} of $G^3$. The highest derivative appearing in these $\sigma$-related terms is $\sigma\nab_\ast^2\eta$. We need to make use of this $\sigma$ factor so that these $\sigma$-related terms can be bounded in terms of $\se{n}^\sigma$ and $\sd{n}^\sigma$. For example, to bound $\ns{\nab_{\ast }^{\, 4N-1} G^3}_{1/2}$, we estimate $\sigma\norm{\nab_{\ast }^{\, 4N-1} \nab_\ast^2\eta}_{1/2}\le\sigma\norm{  \eta}_{4N+3/2} \le \sqrt{\sd{2N}^\sigma}$. Hence, we can deduce the estimates of $G^3$ in \eqref{p_G_e_0}--\eqref{p_G_e_h_00}. For the $G^4$ term, the difference results from that we control the higher regularities of $G^4$ than those in Theorems 3.1--3.2 of \cite{GT_per}, but the proof is similar.
\end{proof}

\subsection{Global energy evolution}\label{sss}
In this subsection, we will derive the global energy evolution of the solution to \eqref{geometric}. We define the ``horizontal'' energies and dissipations with localization as follows. We define those only involving temporal derivatives by
\begin{equation}\label{p_temporal_energies_def}
 \seb{n}^{\sigma,t} = \sum_{j=0}^{n} \ns{\sqrt{J}\dt^j u}_{0}  + \sum_{j=0}^{n} \ns{\dt^j \eta}_{0}+ \sigma\sum_{j=0}^{n} \ns{\dt^j \eta}_{1}
\text{ and }
 \sdb{n}^{ t} = \sum_{j=0}^{n} \ns{ \sg \dt^j u}_{0}.
\end{equation}
We define those localized away from the lower boundary $\Sigma_b$ involving both the horizontal spatial derivatives and temporal derivatives by
\begin{equation}\label{p_upper_energy_def}
 \seb{n}^{\sigma,+} =  \ns{ \bar{\nab}_{\ast 0}^{\ 2n-1}  (\chi u)}_{0} + \ns{ \nab_\ast \bar{\nab}_{\ast 0}^{\ 2n-1}  (\chi u)}_{0} + \sum_{j=0}^{n-1} \ns{\dt^j \eta}_{2n-2j}+ \sigma\sum_{j=0}^{n-1} \ns{\dt^j \eta}_{2n-2j+1}
\end{equation}
and
\begin{equation}\label{p_upper_dissipation_def}
  \sdb{n}^{ +} =   \ns{ \bar{\nab}_{\ast 0}^{\ 2n-1} \sg (\chi u)}_{0} + \ns{ \nab_\ast\bar{\nab}_{\ast 0}^{\ 2n-1} \sg (\chi u)}_{0}.
\end{equation}
We now define
\begin{equation}
  \seb{n}^\sigma  = \seb{n}^{\sigma,t}+\seb{n}^{\sigma,+}\text{ and } \sdb{n}  = \sdb{n}^{ t}+\sdb{n}^{ +}.
\end{equation}

\begin{remark}
Note that we only consider the energy evolution of localized terms away from $\Sigma_b$.  The method employed in \cite{GT_per} involves another energy and dissipation pair for terms localized near $\Sigma_b$.   Here, our modification of the method of \cite{GT_per} frees us from the need to introduce such a lower localization.
\end{remark}

We first derive the global energy evolution of the pure temporal derivatives of the solution, that is, the evolution of $\bar{\mathcal{E}}_{n}^{\sigma,t}$. We first present the result at the $2N$ level.
\begin{prop}\label{i_temporal_evolution 2N}
It holds that
\begin{equation} \label{tem en 2N}
   \frac{d}{dt} \left( \bar{\mathcal{E}}_{2N}^{\sigma,t}+2\int_\Omega J   \dt^{2N-1} p  F^{2,2N}\right)
+  \bar{\mathcal{D}}_{2N}^{t}
  \ls  \sqrt{\se{2N}^0} \sd{2N}^0
\end{equation}
and
\begin{equation} \label{tem en 2N00}
     \int_\Omega J   \dt^{2N-1} p  F^{2,2N}\ls (\se{2N}^0)^{3/2}.
\end{equation}
\end{prop}

\begin{proof}
Taking the dot product of the first equation of $\eqref{linear_geometric}$ with $\dt^j u$ and then integrating by parts, using the second, third and fifth equations, we obtain
\begin{equation}\label{i_ge_ev_0}
\begin{split}
&\hal  \frac{d}{dt} \int_\Omega  J \abs{\dt^j u}^2
+ \hal \int_\Omega J \abs{ \sg_{\mathcal{A}} \dt^j u}^2
\\&\quad= \int_\Omega J (  \dt^j u\cdot F^{1,j}+  \dt^j p  \diverge_\a (\dt^j u))
- \int_\Sigma S_{ik}(\dt^j p,\dt^j u)\n_k \dt^j u_i
\\&\quad= \int_\Omega J (   \dt^j u\cdot F^{1,j}+  \dt^j p  F^{2,j})
- \int_\Sigma (\dt^j \eta-\sigma\Delta_\ast \dt^j \eta) \n\cdot \dt^j u + F^{3,j}\cdot \dt^j u.
\end{split}
\end{equation}
The fourth equation further implies
\begin{equation}\label{i_ge_ev_1}
\begin{split}
- \int_\Sigma (\dt^j \eta-\sigma\Delta_\ast \dt^j \eta) \n\cdot \dt^j u & = - \int_\Sigma(\dt^j \eta-\sigma\Delta_\ast \dt^j \eta)(\dt (\dt^j \eta) -F^{4,j})
\\&=-\hal \frac{d}{dt}\int_\Sigma \abs{\dt^j \eta}^2+\sigma\abs{\nab_\ast \dt^j \eta}^2+\int_\Sigma(\dt^j \eta-\sigma\Delta_\ast \dt^j \eta)F^{4,j}.
\end{split}
\end{equation}
Hence, we have
\begin{equation} \label{identity1}
\begin{split}
&\hal  \frac{d}{dt} \left(\int_\Omega  J \abs{\dt^j u}^2+\int_\Sigma \abs{\dt^j \eta}^2+\sigma\abs{\nab_\ast \dt^j \eta}^2\right)
+ \hal \int_\Omega J \abs{ \sg_{\mathcal{A}} \dt^j u}^2
\\&\quad= \int_\Omega J (   \dt^j u\cdot F^{1,j}+  \dt^j p  F^{2,j})
+\int_\Sigma -  \dt^j u\cdot F^{3,j}+(\dt^j \eta-\sigma\Delta_\ast \dt^j \eta)F^{4,j}.
\end{split}
\end{equation}

We now estimate the right hand side of \eqref{identity1}. For the $F^{1,j}$ term, by \eqref{p_F_e_01}, we may bound
\begin{equation}\label{i_te_2}
\int_\Omega J \dt^j u\cdot F^{1,j} \le   \pnorm{J}{\infty}\norm{\dt^j u}_{0}   \norm{F^{1,j}}_0 \ls  \sqrt{\sd{2N}^0 } \sqrt{\se{2N}^0 \sd{2N}^0}
 .
\end{equation}
For the $F^{3,j}$ and $F^{4,j}$ terms, by \eqref{p_F_e_01} and the trace theory, we have
\begin{equation}\label{i_te_3}
\begin{split}
 &\int_\Sigma -  \dt^j u\cdot F^{3,j}+(\dt^j \eta-\sigma\Delta_\ast \dt^j \eta)F^{4,j}\ls  \snormspace{\dt^{j} u}{0}{\Sigma} \norm{F^{3,j}}_{0} + \norm{\dt^{j} \eta}_{2} \norm{F^{4,j}}_{0} \\&\quad
\ls  \left( \norm{\dt^{j} u}_{1} + \norm{\dt^{j} \eta}_{2} \right)\sqrt{\se{2N}^0\sd{2N}^0} \le  \sqrt{\sd{2N}^0 } \sqrt{\se{2N}^0 \sd{2N}^0}
.
\end{split}
\end{equation}

For the $F^{2,j}$ term, we must consider the case $j<2N$ and $j=2N$ separately. For $j<2N$, by \eqref{p_F_e_01} we have
\begin{equation}\label{i_te_4}
\int_\Omega J   \dt^j p  F^{2,j}  \le  \pnorm{J}{\infty} \norm{\dt^j p}_{0}   \norm{F^{2,j}}_0 \ls  \sqrt{\sd{2N}^0} \sqrt{\se{2N}^0 \sd{2N}^0}.
\end{equation}
The case $j=2N$ is much more involved since we can not control $\dt^{2N}p$. We are then forced to integrate by parts in time:
\begin{equation}
 \int_\Omega J   \dt^{2N} p  F^{2,2N}= \int_\Omega \dt(J \dt^{2N-1} p  F^{2,2N})-\int_\Omega     \dt^{2N-1} p \dt(J F^{2,2N}) .
\end{equation}
By \eqref{p_F_e_01}, we may bound
\begin{equation}\label{i_te_5}
 -  \int_\Omega  \dt^{2N-1} p  \dt(J F^{2} ) \ls   \norm{\dt^{2N-1} p }_{0} \norm{\dt(J F^{2} )}_{0} \ls   \sqrt{\sd{2N}^0} \sqrt{\se{2N}^0 \sd{2N}^0} .
\end{equation}

Now we combine \eqref{i_te_2}--\eqref{i_te_5} to deduce from \eqref{identity1} that, summing over $j$,
\begin{equation} \label{identity2}
   \frac{d}{dt} \left( \bar{\mathcal{E}}_{2N}^{\sigma,t}+2\int_\Omega J   \dt^{2N-1} p  F^{2,2N}\right)
+  \sum_{j=0}^{2N} \int_\Omega J \abs{ \sg_{\mathcal{A}} \dt^j u}^2
  \ls \sqrt{\se{2N}^0} \sd{2N}^0.
\end{equation}
We follow the estimates (4.17)--(4.21) of \cite{GT_per} to have
\begin{equation}\label{identity222}
\int_\Omega J \abs{ \sg_{\mathcal{A}} \dt^j u}^2\ge  \ns{ \sg \dt^j u}_{0}-C\sqrt{\se{2N}^0} \sd{2N}^0.
\end{equation}
We then deduce \eqref{tem en 2N} from \eqref{identity2}--\eqref{identity222}. On the other hand, we may use \eqref{p_F_e_02} to obtain
\begin{equation}
     \int_\Omega J   \dt^{2N-1} p  F^{2,2N}\ls \norm{J}_{L^\infty}\norm{\dt^{2N-1} p}_{0}\norm{F^{2,2N}}_{0}\ls(\se{2N}^0)^{3/2}.
\end{equation}
This implies \eqref{tem en 2N00} and we thus conclude the proposition.
\end{proof}

We then record a similar result at the $N+2$ level.
\begin{prop}\label{i_temporal_evolution N+2}
It holds that
\begin{equation} \label{tem en N+2}
   \frac{d}{dt} \left( \bar{\mathcal{E}}_{N+2}^{\sigma,t}+2\int_\Omega J   \dt^{N+1} p  F^{2,N+2}\right)
+  \bar{\mathcal{D}}_{N+2}^{t}
  \le  \sqrt{\se{2N}^0} \sd{N+2}^0
\end{equation}
and
\begin{equation} \label{tem en N+200}
     \int_\Omega J   \dt^{N+1} p  F^{2,N+2}\ls \sqrt{\se{2N}^0}\se{N+2}^0.
\end{equation}
\end{prop}

\begin{proof}
The proof is similar to Proposition \ref{i_temporal_evolution 2N}, except using \eqref{p_F_e_h_01}--\eqref{p_F_e_h_02} in place of \eqref{p_F_e_01}--\eqref{p_F_e_02}. We thus omit the details.
\end{proof}

We now derive the global energy evolution of $\bar{\mathcal{E}}_{2N}^{\sigma,+}$.
\begin{prop}\label{p_upper_evolution 2N}
For any $\ep \in (0,1)$ there exists a constant $C(\varepsilon)>0$ such that
\begin{equation}\label{p_u_e_00}
 \frac{d}{dt}\seb{2N}^{\sigma,+} +  \sdb{2N}^{+} \ls   \sqrt{\se{2N}^\sigma} \sd{2N}^\sigma + \sqrt{ \sd{2N}^\sigma \k \f } + \ep \sd{2N}^\sigma +  C(\varepsilon) \sdb{2N}^t.
\end{equation}
\end{prop}
\begin{proof}
Applying $\partial^\alpha$ with $\al\in \mathbb{N}^{1+2}$ so that $\al_0\le 2N-1$ and $|\al|\le 4N$ to the first equation of \eqref{p_localized_equations} and then taking the dot product with $\pa u$, using the other equations as in Proposition \ref{i_temporal_evolution 2N}, we find that
\begin{equation} \label{p_u_e_111}
\begin{split}
 &\hal \frac{d}{dt}\left(  \int_\Omega \abs{\pa (\chi u)}^2  +    \int_\Sigma \abs{\pa \eta}^2 +  \sigma \abs{\pa\nab_\ast \eta}^2\right)
+ \hal \int_\Omega \abs{\sg \pa (\chi u)}^2
 \\&\quad= \int_\Omega \chi  \pa u  \cdot (\chi\pa G^1+\pa H^{1 })+\int_\Omega  \chi \pa p (\chi\pa G^2+\pa H^{2 })
\\&\quad\ \
  + \int_\Sigma -\pa u \cdot \pa G^3 + (\pa \eta-\sigma\Delta_\ast\pa \eta) \pa G^4.
\end{split}
\end{equation}

We will estimate the terms on the right side of \eqref{p_u_e_111}, beginning with the terms involving $H^{1}$ and $H^{2}$.
By the expression of $H^1$ and $H^2$ \eqref{p_H_def}, we have
\begin{equation}\label{intt}
\begin{split}
 \int_\Omega \chi  \pa u  \cdot \pa H^{1 }  + \chi  \pa p \pa H^{2 }  &\ls
\norm{\pa u}_{0}\left(\norm{\pa p}_{0} + \norm{\pa u}_{1} \right)
+ \norm{\pa p}_{0} \norm{\pa u}_{0}
\\
&\ls \norm{\pa u}_{0}\left(\norm{\pa p}_{0} + \norm{\pa u}_{1} \right)
  \ls \norm{\dt^{\alpha_0} u}_{4N-2\alpha_0}\sqrt{\sd{2N}^\sigma} .
  \end{split}
\end{equation}
We use the standard Sobolev interpolation to obtain
\begin{equation}
\|\partial^{\alpha_0} u\|_{4N-2\alpha_0} \lesssim \|\partial^{\alpha_0} u\|_{0}^{\theta}\| \partial^{\alpha_0} u\|_{4N-2\alpha_0+1}^{1-\theta} \lesssim (\bar{\mathcal{D}}_{2N}^{ t})^{\theta/2} (\mathcal{D}_{2N}^\sigma)^{(1-\theta)/2},
\end{equation}
where $\theta = (4N-2\alpha_0 +1)^{-1} \in (0,1)$. This and \eqref{intt} together with Young's inequality imply
\begin{equation}\label{m8}
\int_\Omega \chi \partial^\alpha u\cdot \partial^\alpha H^1 + \chi \partial^\alpha p \partial^\alpha H^2 \lesssim (\bar{\mathcal{D}}_{2N}^{ t})^{\theta/2} (\mathcal{D}_{2N}^\sigma)^{1-\theta/2}
\lesssim \varepsilon \mathcal{D}_{2N}^\sigma + C(\varepsilon) \bar{\mathcal{D}}_{2N}^t.
\end{equation}

We now turn to estimate the terms involving $G^{i},1\le i\le 4$.  Assume initially that   $\abs{\alpha}\le 4N-1$.  Then by \eqref{p_G_e_00} along with the trace theory, we have
\begin{equation}\label{i_de_2}
\begin{split}
\abs{ \int_\Omega   \chi^2\left(  \pa  u \cdot   \pa G^1 +  \pa p \pa G^2 \right)} &\le \norm{\pa  u}_{0}  \norm{ \pa G^1 }_{0} +
\norm{\pa  p}_{0}  \norm{ \pa G^2 }_{0} \\
&\ls \sqrt{\sd{2N}^\sigma} \sqrt{ \se{2N}^\sigma  \sd{2N}^\sigma + \k \f}
\end{split}
\end{equation}
and
\begin{equation}\label{i_de_3}
\begin{split}
&\abs{ \int_\Sigma -\pa u \cdot \pa G^3 + (\pa \eta-\sigma\Delta_\ast\pa \eta) \pa G^4 }
\\&\quad\le
\snormspace{\pa  u}{0}{\Sigma}  \norm{ \pa G^3 }_{0} +
(\norm{\pa  \eta}_{0} +\sigma\norm{\pa  \eta}_{2}) \norm{ \pa G^4 }_{0}
\\&\quad\ls \sqrt{\sd{2N}^\sigma} \sqrt{ \se{2N}^\sigma  \sd{2N}^\sigma + \k \f}  .
\end{split}
\end{equation}

Now assume that $\abs{\alpha}=4N$.  Since $\alpha_0 \le 2N-1$, we have that $\al_1+\al_2\ge 2$, then we may write $\alpha = \beta +(\alpha-\beta)$ for some $\beta \in \mathbb{N}^2$ with $\abs{\beta}=1$.  Since $\abs{\alpha-\beta} = 4N-1$, we can then integrate by parts and use \eqref{p_G_e_00} to have
\begin{equation}\label{i_de_4}
\begin{split}
\abs{ \int_\Omega   \chi^2  \pa  u \cdot   \pa G^1 } &= \abs{ \int_\Omega    \chi^2 \p^{\alpha+\beta}  u \cdot   \p^{\alpha-\beta} G^1 }
 \le \norm{\p^{\alpha+\beta}  u}_{0}  \norm{ \p^{\alpha-\beta} G^1 }_{0} \\&
\le \norm{\p^{\alpha}  u}_{1}  \norm{ \bar{\nab}^{4N-1} G^1 }_{0}
\ls \sqrt{\sd{2N}^\sigma} \sqrt{ \se{2N}^\sigma  \sd{2N}^\sigma + \k \f} .
\end{split}
\end{equation}
For the $G^2$ term we do not need to integrate by parts:
\begin{equation}\label{i_de_5}
\begin{split}
\abs{ \int_\Omega \chi^2  \pa p \pa G^2 }
&\le \norm{\pa p}_{0}  \norm{ \p^{\alpha-\beta}\p^\beta G^2 }_{0}
\le \norm{\pa p}_{0}  \norm{ \bar{\nab}^{4N-1}   G^2 }_{1}
  \\&\ls \sqrt{\sd{2N}^\sigma} \sqrt{ \se{2N}^\sigma  \sd{2N}^\sigma + \k \f}.
\end{split}
\end{equation}
For the $G^3$ term we integrate by parts and use the trace estimate to see that
\begin{equation}\label{i_de_6}
\begin{split}
\abs{ \int_\Sigma     \pa  u \cdot   \pa G^3 } &= \abs{ \int_\Sigma  \p^{\alpha+\beta}  u \cdot   \p^{\alpha-\beta} G^3 }
 \le \snormspace{\p^{\alpha+\beta}  u}{-1/2}{\Sigma}  \norm{ \p^{\alpha-\beta} G^3 }_{1/2} \\
&\le \snormspace{\p^{\alpha}  u}{1/2}{\Sigma}  \norm{ \bar{\nab}_{\ast}^{\, 4N-1} G^3 }_{1/2}
 \le \norm{\p^{\alpha}  u}_{1}  \norm{\bar{\nab}_{\ast}^{\, 4N-1} G^3 }_{1/2}
  \\&\ls \sqrt{\sd{2N}^\sigma} \sqrt{ \se{2N}^\sigma  \sd{2N}^\sigma + \k \f}.
\end{split}
\end{equation}
For the $G^4$ term we must split to two cases: $\alpha_0\ge 1$ and $\alpha_0 =0$.  In the former case, there is at least one temporal derivative in $\pa$, so by \eqref{p_G_e_00} we have
\begin{equation}\label{i_de_7}
\begin{split}
\abs{ \int_\Sigma    ( \pa \eta-\sigma\Delta_\ast\pa \eta )   \pa G^4 }    &\le (\norm{\p^{\alpha}  \eta}_{-1/2}+\sigma\norm{\p^{\alpha}  \eta}_{3/2} )  \norm{\bar{\nab}_{\ast}^{\, 4N-2} \dt G^4 }_{1/2}
\\&\ls \sqrt{\sd{2N}^\sigma} \sqrt{ \se{2N}^\sigma  \sd{2N}^\sigma + \k \f}.
\end{split}
\end{equation}
In the latter case, $\pa$ involves only spatial derivatives; in this case we claim that
\begin{equation}\label{i_de_8}
 \abs{ \int_\Sigma    ( \pa \eta-\sigma\Delta_\ast\pa \eta )    \pa G^4 } \ls \sqrt{\se{2N}^\sigma} \sd{2N}^\sigma + \sqrt{\sd{2N}^\sigma \k \f}.
\end{equation}
Indeed, by the Leibniz rule, we have
\begin{equation}
 \pa G^4 = \pa (\nab_\ast \eta \cdot u) =   \nab_\ast \pa \eta \cdot   u+\sum_{ 0<\beta \le \alpha  } C_{\alpha}^{\beta} \nab_\ast\p^{\alpha-\beta} \eta \cdot\p^\beta u
\end{equation}
For the first term, we use the integration by parts to see that
\begin{equation}\label{e23}
\begin{split}
 &\int_\Sigma ( \pa \eta-\sigma\Delta_\ast\pa \eta )  \nab_\ast \pa \eta \cdot u
 = \int_\Sigma   \pa \eta    \nab_\ast \pa \eta \cdot u +\int_\Sigma  \sigma\nab_\ast\pa \eta   \nab_\ast (\nab_\ast \pa \eta \cdot u )
 \\&\quad= \hal\int_\Sigma     \nab_\ast \abs{\pa \eta}^2 \cdot u+\hal \sigma\int_\Sigma \nab_\ast \abs{\pa \nab_\ast\eta}^2 \cdot u +\int_\Sigma \sigma\pa \nab_\ast\eta  (\nab_\ast \pa \eta \cdot  \nab_\ast u )
 \\&\quad
= \hal\int_\Sigma  \abs{\pa  \eta}^2 \diverge_\ast u -\hal \sigma\int_\Sigma   \abs{\pa \nab_\ast\eta}^2 \diverge_\ast u+\int_\Sigma \sigma\pa \nab_\ast\eta  (\nab_\ast \pa \eta \cdot  \nab_\ast u )
 \\& \quad\ls  (\norm{\eta}_{4N-1/2}+\sigma\norm{\eta}_{4N+3/2})\norm{\eta}_{4N+1/2}\norm{\nab_\ast  u}_{H^2(\Sigma)}\le \sqrt{\sd{2N}^\sigma\f \k}.
\end{split}
\end{equation}
For the second term, we estimate for $|\beta|\ge 1$, similarly as Lemma \ref{p_G_estimates},
\begin{equation}
 \norm{\nab_\ast \p^{\alpha-\beta} \eta \cdot\p^\beta u}_{1/2}
 \ls  \sqrt{ \se{2N}^\sigma  \sd{2N}^\sigma + \k \f}.
\end{equation}
Hence, we have
\begin{equation}\label{e2312}
\begin{split}
 \int_\Sigma   ( \pa \eta-\sigma\Delta_\ast\pa \eta )  \nab_\ast\p^{\alpha-\beta} \eta \cdot\p^\beta u
 &\ls \norm{\pa \eta-\sigma\Delta_\ast\pa \eta}_{-1/2}   \norm{\nab_\ast\p^{\alpha-\beta} \eta \cdot\p^\beta u}_{1/2}
 \\&\ls \sqrt{\sd{2N}^\sigma}\sqrt{ \se{2N}^\sigma  \sd{2N}^\sigma + \k \f}.
 \end{split}
\end{equation}
In light of \eqref{e23} and \eqref{e2312}, we conclude the claim \eqref{i_de_8}.

Now, by \eqref{m8}--\eqref{i_de_8}, we deduce from \eqref{p_u_e_111} that for all $\abs{\alpha} \le 4N$ with $\alpha_0 \le 2N-1$,
\begin{equation}  \label{p_u_e_9}
\begin{split}
 &\hal \frac{d}{dt}\left(  \int_\Omega \abs{\pa (\chi u)}^2  +    \int_\Sigma \abs{\pa \eta}^2 +  \sigma \abs{\pa\nab_\ast \eta}^2\right)
+ \hal \int_\Omega \abs{\sg \pa (\chi u)}^2
 \\&\quad\ls  \sqrt{\sd{2N}^\sigma} \sqrt{ \se{2N}^\sigma  \sd{2N}^\sigma + \k \f}
+ \varepsilon \mathcal{D}_{2N}^\sigma + C(\varepsilon) \bar{\mathcal{D}}_{2N}^t
 \\&\quad\ls   \sqrt{\se{2N}^\sigma} \sd{2N}^\sigma + \sqrt{ \sd{2N}^\sigma \k \f } + \ep \sd{2N}^\sigma +  C(\varepsilon) \sdb{2N}^t.
\end{split}
\end{equation}
The estimate \eqref{p_u_e_00} then follows from \eqref{p_u_e_9} by summing over such $\alpha$.
\end{proof}

\begin{prop}\label{p_upper_evolution N+2}
For any $\ep \in (0,1)$ there exists a constant $C(\varepsilon)>0$ such that
\begin{equation}\label{p_u_e_11}
 \frac{d}{dt}\seb{N+2}^{\sigma,+} +  \sdb{N+2}^{+} \ls   \sqrt{\se{2N}^\sigma} \sd{N+2}^\sigma  + \ep \sd{N+2}^\sigma +  C(\varepsilon) \sdb{N+2}^t.
\end{equation}
\end{prop}
\begin{proof}
The proof is similar to Proposition \ref{i_temporal_evolution 2N}, except using \eqref{p_G_e_h_00} in place of \eqref{p_G_e_00} and estimating in a slight different way the $G^4$ term when $\al=2(N+2)$, namely,
\begin{equation}
\int_\Sigma    ( \pa \eta-\sigma\Delta_\ast\pa \eta )   \pa G^4 .
\end{equation}
But by \eqref{p_G_e_h_00}, we have
\begin{equation}
\int_\Sigma    ( \pa \eta-\sigma\Delta_\ast\pa \eta )   \pa G^4
\le (\norm{\p^{\alpha}  \eta}_{-1/2}+\sigma\norm{\p^{\alpha}  \eta}_{3/2} )  \norm{ \bar{\nab}_\ast^{\,2(N+2)}  G^4 }_{1/2}\ls\sqrt{\se{2N}^\sigma} \sd{N+2}^\sigma.
\end{equation}
Using this, we can conclude the proposition.
\end{proof}

\subsection{Comparison results}
In this subsection, we shall now show  that, up to some error terms, ${\mathcal{E}}_{n}^\sigma$ is comparable to $ \bar{\mathcal{E}}_{n}^\sigma$ and that ${\mathcal{D}} _{n}^\sigma$ is comparable to $ \bar{\mathcal{D}}_{n}$, for both $n=2N$ and $n=N+2$.  We begin with the result for the instantaneous energy.

\begin{thm}\label{eth}
It holds that
\begin{equation}\label{e2n}
{\mathcal{E}}_{2N}^\sigma\lesssim  \bar{\mathcal{E}}_{2N}^\sigma +({\mathcal{E}}_{2N}^\sigma)^{2}+\k \f
\end{equation}
and
\begin{equation}\label{en+2}
{\mathcal{E}}_{N+2}^\sigma\lesssim  \bar{\mathcal{E}}_{N+2}^\sigma + \mathcal{E}_{2N}^\sigma \mathcal{E}_{N+2}^\sigma.
\end{equation}
\end{thm}
\begin{proof}
We first let $n$ denote either $2N$ or $N+2$ throughout the proof, and we compactly write
\begin{equation}\label{p_E_b_0}
 \z_n =  \ns{ \bar{\nab}_0^{2n-2} G^1}_{0}  +  \ns{ \bar{\nab}_0^{2n-2}  G^2}_{1} +
 \ns{ \bar{\nab}_{\ast 0}^{\  2n-2} G^3}_{1/2} +
 \ns{\bar{\nab}_{\ast 0}^{\  2n-1} G^4}_{1/2} .
\end{equation}
Note that  the definitions of $\seb{n}^\sigma  = \seb{n}^{\sigma,t}+\seb{n}^{\sigma,+}$ guarantee that
\begin{equation}\label{p_E_b_4}
  \ns{\dt^{n} u}_{0}+\sum_{j=0}^{n} \ns{\dt^j \eta}_{2n-2j}+\sigma\sum_{j=0}^{n} \ns{\dt^j \eta}_{2n-2j+1} \ls   \seb{n}^\sigma.
\end{equation}

Now we let $j=0,\dots,n-1$ and then apply $\partial_t^j$ to the equations in \eqref{perturb f} to find
\begin{equation}\label{jellip}
\begin{cases}
- \Delta \partial_t^j u+\nabla\partial_t^j p=-\partial_t^{j+1} u+\partial_t^j G^1
 &\text{in }\Omega
\\ \diverge\partial_t^j u=\partial_t^j G^2&\text{in }\Omega
\\   (\partial_t^j p I- \mathbb{D}(\partial_t^j u ) ) e_3= (\partial_t^j\eta-\sigma \Delta_\ast\partial_t^j\eta)  e_3+\partial_t^jG^3&\text{on }\Sigma
\\ \partial_t^j u=0 &\text{on }\Sigma_b.
\end{cases}
\end{equation}
Applying the elliptic estimates of Lemma \ref{i_linear_elliptic} with $r=2n-2j\ge 2$ to the problem \eqref{jellip} and using \eqref{p_E_b_0}--\eqref{p_E_b_4}, we obtain
\begin{equation}\label{p_E_b_2}
\begin{split}
 &\norm{\dt^j  u  }_{2n-2j}^2 + \norm{\dt^j  p  }_{2n-2j-1}^2
 \\&\quad \ls
\norm{\dt^{j+1} u   }_{2n-2j-2 }^2 + \norm{ \dt^j G^1   }_{2n-2j-2}^2
+ \norm{\dt^j  G^2  }_{2n-2j-1}^2 \\ &\qquad
 + \norm{\dt^j  \eta }_{2n-2j-3/2}^2+\sigma^2 \norm{\dt^j  \eta }_{2n-2j+1/2}^2 + \norm{\dt^j  G^3  }_{2n-2j-3/2}^2
 \\&\quad \ls   \norm{\dt^{j+1} u   }_{2n-2(j+1) }^2 + \seb{n}^\sigma+  \z_n.
\end{split}
\end{equation}
A simple induction on \eqref{p_E_b_2} yields, by \eqref{p_E_b_4} again,
\begin{equation}\label{claim22}
\sum_{j=0}^{n } \norm{\dt^{j}  u  }_{2n-2j }^2 + \sum_{j=0}^{n-1}\norm{\dt^{j} p  }_{2n-2j-1}^2 \ls \ns{\dt^{n} u}_{0}+
\seb{n}^\sigma  +\z_n\ls
\seb{n}^\sigma  +\z_n.
\end{equation}
On the other hand, we use the boundary condition
\begin{equation}
\partial_t\eta=u_3+G^4\text{ on }\Sigma
\end{equation}
to have that for $j=1,\dots,n$, by the trace theory, \eqref{p_E_b_0} and \eqref{claim22},
\begin{equation}\label{lj}
\begin{split}
\ns{\partial_t^j\eta}_{2n-2j+3/2} &\le \ns{\dt^{j-1}u_3}_{H^{2n-2j+3/2}(\Sigma)}+\ns{\dt^{j-1} G^4}_{2n-2j+3/2}
\\&\le \ns{\dt^{j-1}u }_{2n-2(j-1)}+\ns{\dt^{j-1} G^4}_{2n-2(j-1)-1/2}\ls  \seb{n}^\sigma  +\z_n.
\end{split}
\end{equation}
For $j=n+1$, we need to use the geometric facts: $\dt^{n+1}\eta=D_t^{n}u\cdot \n$ and $\diverge_\a (D_t^{n}u)=0$. Then by Lemma \ref{l_boundary_dual_estimate} and \eqref{claim22}, we have
\begin{equation}\label{lj23}
 \ns{\dt^{n+1}\eta}_{ {-1/2}}=\snormspace{D_t^{n}u \cdot \n}{-1/2}{\Sigma} \ls \hn{D_t^{2N}u}{0}\ls \seb{n}^\sigma  +\z_n.
\end{equation}

Consequently, summing \eqref{p_E_b_4}, \eqref{claim22} and \eqref{lj}--\eqref{lj23}, we obtain
\begin{equation}\label{claim}
 {\mathcal{E}}_{n}^\sigma\ls
\seb{n}^\sigma  +\z_n.
\end{equation}
Setting $n=2N$ in \eqref{claim}, and using \eqref{p_G_e_0} of Lemma \ref{p_G_estimates} to estimate $\mathcal{Z}_{2N} \lesssim ({\mathcal{E}}_{2N}^\sigma)^{2}+\k \f$, we then obtain \eqref{e2n}; setting $n=N+2$ in \eqref{claim}, and using \eqref{p_G_e_h_0} of Lemma \ref{p_G_estimates} to estimate $\mathcal{Z}_{N+2}\lesssim  {\mathcal{E}} _{2N}^\sigma {\mathcal{E}}_{N+2}^\sigma$, we then obtain \eqref{en+2}.
\end{proof}

Now we consider a similar result for the dissipation.

\begin{thm}\label{dth}
It holds that
\begin{equation}\label{d2n}
{\mathcal{D}}_{2N}^\sigma\lesssim  \bar{\mathcal{D}}_{2N} + \mathcal{E}_{2N}^\sigma \mathcal{D}_{2N}^\sigma+ \mathcal{K}{\mathcal{F}}_{2N}
\end{equation}
and
\begin{equation}\label{dn+2}
{\mathcal{D}}_{N+2}^\sigma\lesssim  \bar{\mathcal{D}}_{N+2}
+\mathcal{E}_{2N}^\sigma \mathcal{D}_{N+2}^\sigma.
\end{equation}
\end{thm}

\begin{proof}
We again let $n$ denote either $2N$ or $N+2$ throughout the proof, and we compactly write
\begin{equation}\label{p_D_b_4}
\begin{split}
 \y_{n} = & \ns{ \bar{\nab}_0^{2n-1} G^1}_{0} +  \ns{ \bar{\nab}_0^{2n-1}  G^2}_{1}  +
 \ns{\bar{\nab}_{\ast 0}^{\  2n-1} G^3}_{1/2} \\&+ \ns{\bar{\nab}_{\ast 0}^{\  2n-1} G^4}_{1/2}
+ \ns{\bar{\nab}_{\ast 0}^{\  2n-2} \dt G^4}_{1/2}+ \sigma^2\ns{\nab_\ast^{\,2n}   G^4}_{1/2}.
\end{split}
\end{equation}

Notice that we have not yet derived an estimate of $\eta$ in terms of the dissipation, so we can not apply the elliptic estimates of Lemma \ref{i_linear_elliptic} as in Theorem \ref{eth}.  It is crucial to observe that we can get higher regularity estimates of $u$ on the boundary $\Sigma$ from $\bar{\mathcal{D}}_{n}$. Indeed, since $\Sigma$ is flat, we may use the definition of Sobolev norm on $\Sigma$, the trace
theorem and Korn's inequality of Lemma \ref{i_korn} to obtain, by the definitions of $\sdb{n}  = \sdb{n}^{ t}+\sdb{n}^{ +}$,
\begin{equation}\label{n21}
\begin{split}
\ns{\partial_t^{j} u}_{H^{2n-2j+1/2}(\Sigma)}& =\ns{\partial_t^{j}(\chi u)}_{H^{2n-2j+1/2}(\Sigma)} \lesssim \ns{\partial_t^{j} (\chi u)}_{L^2(\Sigma)}
+\ns{\nab_\ast^{\,2n-2j}\partial_t^{j}(\chi u)}_{H^{1/2}(\Sigma )}
\\&
\lesssim \ns{ \partial_t^{j}   (\chi u) }_{1} +\ns{\nab_\ast^{\,2n-2j}\partial_t^{j}  (\chi u) }_{1}
 \lesssim \bar{\mathcal{D}}_{n}.
 \end{split}
\end{equation}
This motivates us to use the elliptic estimates of Lemma \ref{i_linear_elliptic2}.

Now we let $j=0,\dots,n-1$ and then apply $\partial_t^j$ to the equations in \eqref{perturb f} to find
\begin{equation}\label{jellip2}
\begin{cases}
- \Delta \partial_t^j u+\nabla\partial_t^j p=-\partial_t^{j+1} u+\partial_t^j G^1
 &\text{in }\Omega
\\ \diverge\partial_t^j u=\partial_t^j G^2&\text{in }\Omega
\\  \partial_t^j u=\partial_t^j u &\text{on }\Sigma
\\ \partial_t^j u=0 &\text{on }\Sigma_b.
\end{cases}
\end{equation}
Applying the elliptic estimates of Lemma \ref{i_linear_elliptic2} with $r=2n-2j+1\ge 3$ to the problem \eqref{jellip2} and using \eqref{p_D_b_4}--\eqref{n21}, we obtain
\begin{equation}\label{p_D_b_2}
\begin{split}
 &\norm{\dt^j  u  }_{2n-2j+1}^2 + \norm{\nab \dt^j  p  }_{2n-2j-1}^2
 \\  &\quad\ls
\norm{\dt^{j+1} u   }_{2n-2j-1 }^2 + \norm{ \dt^j G^1   }_{2n-2j-1}^2
+ \norm{\dt^j  G^2  }_{2n-2j}^2 + \norm{\dt^j  u  }_{2n-2j+1/2}^2
\\&\quad\ls   \norm{\dt^{j+1} u   }_{2n-2(j+1)+1 }^2 + \sdb{n} +  \y_n.
\end{split}
\end{equation}
A simple induction on \eqref{p_D_b_2} yields, by Korn's inequality,
\begin{equation}\label{claim2}
\sum_{j=0}^{n} \norm{\dt^{j}  u  }_{2n-2j+1}^2 + \norm{\nab \dt^{j} p  }_{2n-2j-1}^2  \ls \ns{\dt^{n} u}_{1}+
\sdb{n}   +\y_n\ls
\sdb{n}   +\y_n.
\end{equation}

Now that we have obtained \eqref{claim2}, we estimate the rest parts in $\mathcal{D}_{n}$. We will turn to the boundary conditions in \eqref{perturb f}. First we derive estimates for $\eta$. For the term $\dt^j \eta$ for $j\ge 2$ we use the boundary condition
\begin{equation}\label{n61}
\partial_t\eta=u_3+G^4\text{ on }\Sigma.
\end{equation}
Indeed, for $j=2,\dots,n+1$ we apply $\partial_t^{j-1}$ to \eqref{n61} to see, by \eqref{claim2} and \eqref{p_D_b_4}, that
\begin{equation}\label{eta1}
\begin{split}
\ns{\partial_t^j\eta}_{2n-2j+5/2} & \le  \ns{\partial_t^{j-1}u_3}_{H^{2n-2j+5/2}(\Sigma)} +\ns{\partial_t^{j-1}G^4}_{2n-2j+5/2}
\\
&\lesssim \ns{\partial_t^{j-1}u }_{{2n-2(j-1)+1}} +\ns{\partial_t^{j-1}G^4}_{2n-2(j-1)+1/2}
 \lesssim \sdb{n}   +\y_n.
\end{split}
\end{equation}
For the term $\partial_t\eta$, we again use \eqref{n61}, \eqref{claim2} and \eqref{p_D_b_4} to find
\begin{equation}\label{eta2}
\begin{split}
\sigma^2\ns{\partial_t \eta}_{2n+1/2}+\ns{\partial_t \eta}_{2n-1/2} &\lesssim \ns{u_3}_{H^{2n+1/2}(\Sigma)}+\sigma^2\ns{ G^4}_{2n+1/2}+\ns{ G^4}_{2n-1/2}
\\&\lesssim \ns{ u }_{2n+1}+\y_n\lesssim\sdb{n}   +\y_n.
\end{split}
\end{equation}
For the remaining $\eta$ term, i.e. those without temporal derivatives, we use the boundary condition
\begin{equation}\label{pb1}
 -\sigma\Delta_\ast\eta+\eta= p- 2\partial_3u_{3} -G_{3}^3\text{ on }\Sigma.
\end{equation}
Notice that at this point we do not have any bound of $p$ on the boundary $\Sigma$, but we have bounded  $\nabla p$ in $\Omega$.  Applying $\nab_\ast$  to \eqref{pb1} and then applying the standard elliptic theory,  by the trace theory, \eqref{claim2} and \eqref{p_D_b_4}, we obtain
\begin{equation}\label{n51}
\begin{split}
&\sigma^2\ns{ \nab_\ast \eta}_{2n+1/2}+\ns{ \nab_\ast \eta}_{2n-3/2}
\\&\quad\lesssim \ns{  \nab_\ast p }_{H^{2n-3/2}(\Sigma)}  + \ns{  \nab_\ast \partial_3u_3 }_{H^{2n-3/2}(\Sigma)} +\ns{\nab_\ast G^3_3}_{ 2n-3/2 }  \\
&\quad\lesssim \ns{\nabla p}_{2n-1} + \ns{u }_{2n+1}  + \ns{G^3}_{2n-1/2} \lesssim \sdb{n}   +\y_n.
\end{split}
\end{equation}
We may then use Poincar\'e's inequality on $\Sigma$ (since \eqref{avege})  to obtain from \eqref{n51} that
\begin{equation} \label{eta3}
\begin{split}
\sigma^2\ns{\eta}_{2n+3/2}+\ns{\eta}_{2n-1/2} &\lesssim \ns{\eta}_{0}+\sigma^2\ns{ \nab_\ast \eta}_{2n+1/2} + \ns{\nab_\ast \eta}_{2n-3/2}
\\&\lesssim \sigma^2\ns{ \nab_\ast \eta}_{2n+1/2} + \ns{\nab_\ast \eta}_{2n-3/2} \lesssim \sdb{n}   +\y_n.
\end{split}
\end{equation}
Summing  \eqref{eta1}, \eqref{eta2}  and \eqref{eta3}, we complete the estimates for $\eta$:
\begin{equation} \label{eta0}
\sigma^2\ns{\eta}_{2n+3/2}+\sigma^2\ns{\dt\eta}_{2n+1/2}+\ns{\eta}_{2n-1/2}+\ns{\partial_t \eta}_{2n-1/2}  + \sum_{j=2}^{n+1} \ns{\partial_t^j\eta}_{2n-2j+5/2} \lesssim\sdb{n}   +\y_n.
\end{equation}

The $\eta$ estimates \eqref{eta0} allows us to further bound $\ns{\partial_t^jp}_0$. Applying $\partial_t^j,\ j=0,\dots,n-1$ to \eqref{pb1} and employing the trace theory, \eqref{claim2}, \eqref{eta0} and \eqref{p_D_b_4}, we find
\begin{equation}\label{pp1}
\begin{split}
\ns{\partial_t^j p}_{L^2(\Sigma )}
&\lesssim \ns{\partial_t^j \eta}_{0}+\sigma^2\ns{\Delta_\ast\partial_t^j \eta}_{0}  + \ns{\partial_3\partial_t^j u_3}_{L^2(\Sigma)}  + \ns{\partial_t^j G^3_3}_{0}
\\
&\lesssim \ns{\partial_t^j \eta}_{2}+  \ns{ \partial_t^j u }_{2}+\ns{\partial_t^j G^3}_{0} \lesssim \sdb{n}  +\y_n.
\end{split}
\end{equation}
By Poincar\'e's inequality on $\Omega$ (Lemma \ref{poincare_b}) and \eqref{claim2} and \eqref{pp1}, we have
\begin{equation}\label{pp2}
\ns{\partial_t^j p}_{1}  = \ns{\partial_t^j p}_{0}+\ns{\nab\partial_t^j p}_{0}
\lesssim \ns{\partial_t^j p}_{L^2(\Sigma )} +\ns{\nab\partial_t^j p}_{0}
\lesssim \sdb{n}  +\y_n.
\end{equation}
In light of \eqref{pp2}, we may improve the estimate \eqref{claim2} to be
\begin{equation}\label{claim00}
\sum_{j=0}^{n} \norm{\dt^{j}  u  }_{2n-2j+1}^2 + \sum_{j=0}^{n-1}\norm{  \dt^{j} p  }_{2n-2j }^2   \ls
\sdb{n}  +\y_n.
\end{equation}

Consequently, summing \eqref{eta0} and \eqref{claim00}, we obtain
\begin{equation}\label{upeta0}
 \sd{n}^\sigma\ls
\sdb{n}   +\y_n.
\end{equation}
Setting $n=2N$ in \eqref{upeta0} and using \eqref{p_G_e_00} of Lemma \ref{p_G_estimates} to estimate
$\mathcal{Y}_{2N}\lesssim  {\mathcal{E}}_{2N}^\sigma {\mathcal{D}}_{2N}^\sigma + \mathcal{K} {\mathcal{F}}_{2N}$,
we then obtain \eqref{d2n};  setting $n=N+2$ in \eqref{upeta0}, and using \eqref{p_G_e_h_00} of Lemma \ref{p_G_estimates} to estimate $\mathcal{Y}_{N+2}\lesssim  {\mathcal{E}}_{2N}^\sigma {\mathcal{D}}_{N+2}^\sigma$, we then obtain \eqref{dn+2}.
\end{proof}

\subsection{Global energy estimates}

In this subsection, we shall now conclude our global energy estimates of the solution to \eqref{geometric}. We begin with the estimate of $\f$.

\begin{prop}\label{p_f_bound}
There exists a universal constant $0< \delta < 1$ so that if $\g^\sigma(T) \le \delta$, then
\begin{equation}\label{p_f_b_0}
\sup_{0\le r \le t}  \f(r) \ls
\f(0) +   t \int_0^t \sd{2N}^\sigma\text{ for all }0 \le t \le T.
\end{equation}
\end{prop}
\begin{proof}
Based on the transport estimate of Lemma \ref{i_sobolev_transport} on the kinematic boundary condition, we may show as in Lemma 7.1 of \cite{GT_per} that
\begin{equation}\label{l0}
\begin{split}
\sup_{0\le r \le t}  \f(r) \ls& \exp\left(C \int_0^t \sqrt{\k(r)} dr \right) \\
&\times \left[ \f(0)   +   t \int_0^t (1+\se{2N}^\sigma(r)) \sd{2N}^\sigma(r)dr  + \left( \int_0^t \sqrt{\k(r) \f(r)} dr\right)^2 \right].
\end{split}
\end{equation}
The Sobolev and trace embeddings allow us to estimate $\k \ls \se{N+2}^\sigma$, and hence
\begin{equation}\label{l1}
 \int_0^t \sqrt{\k(r)} dr \ls \int_0^t \sqrt{\se{N+2}^\sigma(r)} dr \le  \sqrt{\delta} \int_0^\infty  \frac{1}{(1+r)^{2N-4}}dr  \ls \sqrt{\delta}.
\end{equation}
Then by \eqref{l1}, we deduce from \eqref{l0} that
\begin{equation}\label{l2}
\begin{split}
\sup_{0\le r\le t}\mathcal{F}_{2N}(r) & \lesssim \mathcal{F}_{2N}(0)+t\int_0^t\mathcal{D}_{2N}^\sigma+\sup_{0\le r\le t}{{\mathcal{F}}_{2N}}(r)\left(
  \int_0^t\sqrt{ \mathcal{K}(r)}dr\right)^2
  \\
&\lesssim \mathcal{F}_{2N}(0) + t\int_0^t\mathcal{D}_{2N}^\sigma+\delta\sup_{0\le r\le t}{{\mathcal{F}}_{2N}}(r).
\end{split}
\end{equation}
By taking $\delta$ small enough, we may absorb the right-hand $\f$ term of \eqref{l2} onto the left and deduce \eqref{p_f_b_0}.
\end{proof}

Now we show the boundedness of $\mathcal{E}_{2N}^\sigma+\int_0^t\mathcal{D}_{2N}^\sigma$.

\begin{prop} \label{Dgle}
There exists $\delta>0$ so that if $\mathcal{G}_{2N}^\sigma(T)\le\delta$, then
\begin{equation}\label{Dg}
\mathcal{E}_{2N}^\sigma(t)+\int_0^t\mathcal{D}_{2N}^\sigma\lesssim
\mathcal{E}_{2N}^\sigma(0) + \mathcal{F}_{2N}(0)  \text{ for all
}0\le t\le T.
\end{equation}
\end{prop}

\begin{proof}
Note first that since ${\mathcal{E}}_{2N}^\sigma(t)\le
{\mathcal{G}}_{2N}^\sigma(T)\le \delta$, by taking $\delta$ small, we
obtain from \eqref{e2n} of Theorem \ref{eth} and \eqref{d2n} of
Theorem \ref{dth} that
\begin{equation}\label{q0}
\bar{\mathcal{E}}_{2N}^\sigma \lesssim{\mathcal{E}}_{2N}^\sigma
\lesssim \bar{\mathcal{E}}_{2N}^\sigma + \mathcal{K}
\mathcal{F}_{2N}, \text{ and } \bar{\mathcal{D}}_{2N}
\lesssim {\mathcal{D}}_{2N}^\sigma \lesssim
\bar{\mathcal{D}}_{2N} + \mathcal{K} \mathcal{F}_{2N}.
\end{equation}

Now we multiply the integration in time version of \eqref{tem en 2N}
by a constant $1+\beta\ge 2$ and add then this to  the integration in
time version of \eqref{p_u_e_00} to find, by \eqref{tem
en 2N00},
\begin{equation} \label{q1}
\begin{split}
&(1+\beta)\bar{\mathcal{E}}^{\sigma,t}_{2N}(t) +
\bar{\mathcal{E}}^{\sigma,+}_{2N}(t) + \int_0^t
(1+\beta)\bar{\mathcal{D}}^{t}_{2N}+ \bar{\mathcal{D}}^{+}_{2N}
\\&\quad
\lesssim  (1+\beta)(
\bar{\mathcal{E}}^{\sigma,t}_{2N}(0)+(\mathcal{E}_{2N}^\sigma(0))^{
{3}/{2}})+\bar{\mathcal{E}}^{\sigma,+}_{2N}(0) + (1+\beta)
(\mathcal{E}_{2N}^\sigma(t))^{ {3}/{2}}
\\&\qquad
+ \int_0^t(2+\beta) \sqrt{\mathcal{E}_{2N}^\sigma }
\mathcal{D}_{2N}^\sigma+
\sqrt{\mathcal{D}_{2N}^\sigma\mathcal{K}\mathcal{F}_{2N}}+\varepsilon\mathcal{D}_{2N}^\sigma
+ C(\varepsilon)\bar{\mathcal{D}}^{t}_{2N}.
\end{split}
\end{equation}
Then by \eqref{q0} we may improve \eqref{q1} to be
\begin{multline} \label{q20}
\mathcal{E}_{2N}^\sigma(t)+\int_0^t \mathcal{D}_{2N}^\sigma+ \beta
\bar{\mathcal{D}}^{t}_{2N} \lesssim  (1+\beta)
\mathcal{E}_{2N}^\sigma(0) + (1+\beta)
(\mathcal{E}_{2N}^\sigma(t))^{ 3/2}+\k(t)\f(t)
\\
+ \int_0^t  (2+\beta) \sqrt{\mathcal{E}_{2N}^\sigma }
\mathcal{D}_{2N}^\sigma+\sqrt{\mathcal{D}_{2N}^\sigma\mathcal{K}\mathcal{F}_{2N}}+\mathcal{K}\mathcal{F}_{2N}
+ \varepsilon \mathcal{D}_{2N}^\sigma
+C(\varepsilon)\bar{\mathcal{D}}^{t}_{2N}.
\end{multline}
Taking $\varepsilon$ sufficiently small first, $\beta$ sufficiently
large second, and $\delta$ sufficiently small third, we may deduce
from \eqref{q20} that
\begin{equation} \label{q2}
\mathcal{E}_{2N}^\sigma(t)+\int_0^t \mathcal{D}_{2N}^\sigma \lesssim
 \mathcal{E}_{2N}^\sigma(0)  +\k\f
+ \int_0^t
 \sqrt{\mathcal{D}_{2N}^\sigma\mathcal{K}\mathcal{F}_{2N}}+\mathcal{K}\mathcal{F}_{2N}.
\end{equation}

Since $\mathcal{G}_{2N}^\sigma(T) \le \delta$, it is easy to use Proposition
\ref{p_f_bound} to verify that
\begin{equation}\label{k0}
 \k(t)\f(t) \lesssim \delta(1+t)^{-4N+8}\f\lesssim \delta
\mathcal{F}_{2N}(0)+\delta\int_0^t\mathcal{D}_{2N}^\sigma(r)dr,
\end{equation}
\begin{equation}\label{k1}
\int_0^t   \mathcal{K}(r)\mathcal{F}_{2N}(r)dr\lesssim \delta
\mathcal{F}_{2N}(0)+\delta\int_0^t\mathcal{D}_{2N}^\sigma(r)dr,
\end{equation}
and
\begin{equation}\label{k2}
\int_0^t
\sqrt{\mathcal{D}_{2N}^\sigma(r)\mathcal{K}(r)\mathcal{F}_{2N}(r)}\lesssim
\mathcal{F}_{2N}(0)
+\sqrt{\delta}\int_0^t\mathcal{D}_{2N}^\sigma(r)dr.
\end{equation}
We then plug \eqref{k0}--\eqref{k2} into  \eqref{q2} to have
\begin{equation}\label{k3}
 \mathcal{E}_{2N}^\sigma(t) + \int_0^t\mathcal{D}_{2N}^\sigma
\lesssim \mathcal{E}_{2N}^\sigma(0) +
\mathcal{F}_{2N}(0)+\sqrt{\delta}\int_0^t\mathcal{D}_{2N}^\sigma(r)dr.
\end{equation}
Taking further $\delta$ sufficiently smaller, we may then conclude \eqref{Dg}.
\end{proof}

It remains to show the decay estimates of
$\mathcal{E}_{N+2}^\sigma$.

\begin{prop} \label{decaylm}
There exists $\delta>0$ so that if
$\mathcal{G}_{2N}^\sigma(T)\le\delta$, then
\begin{equation}\label{n+2}
(1+t)^{4N-8} \mathcal{E}_{N+2}^\sigma(t)\lesssim
\mathcal{E}_{2N}^\sigma(0)+ \mathcal{F}_{2N}(0) \ \text{for all
}0\le t\le T.
\end{equation}
\end{prop}
\begin{proof}
Since ${\mathcal{E}}_{2N}^\sigma(t)\le
{\mathcal{G}}_{2N}^\sigma(T)\le \delta$, by taking $\delta$ small,
we obtain from \eqref{en+2} and \eqref{dn+2} that
\begin{equation}
{\mathcal{E}}_{N+2}^\sigma\lesssim\bar{\mathcal{E}}_{N+2}^\sigma
\lesssim{\mathcal{E}}_{N+2}^\sigma, \text{ and }
{\mathcal{D}}_{N+2}^\sigma\lesssim\bar{\mathcal{D}}_{N+2}^\sigma
\lesssim{\mathcal{D}}_{N+2}^\sigma.
\end{equation}
By these estimates and the smallness of $\delta$, we may deduce from
Proposition \ref{i_temporal_evolution N+2} and Proposition
\ref{p_upper_evolution N+2} that there exists an instantaneous
energy, which is equivalent to ${\mathcal{E}}_{N+2}^\sigma$, but for
simplicity is still denoted by ${\mathcal{E}}_{N+2}^\sigma$, such
that
\begin{equation} \label{u1}
\frac{d}{dt} \mathcal{E}_{N+2}^\sigma +  \mathcal{D}_{N+2}^\sigma\le
0.
\end{equation}

In order to get decay from \eqref{u1}, we shall now estimate $\mathcal{E}_{N+2}^\sigma$ in terms of $\mathcal{D}_{N+2}^\sigma$. Notice that ${\mathcal{D}}_{N+2}^\sigma$ can control every term in ${\mathcal{E}}_{N+2}^\sigma$ except $\sigma\ns{\eta}_{2(N+2)+1}$ and $\ns{\eta}_{2(N+2)}$. The key point is to use the Sobolev interpolation. Indeed, we first have that
\begin{equation} \label{intep0}
\ns{\eta}_{2(N+2)}\le  \norm{\eta}_{2(N+2)-1/2} ^{2\theta} \norm{\eta}_{4N}^{2(1-\theta)}\lesssim( {\mathcal{D}_{N+2}^\sigma})^\theta({\mathcal{E}_{2N}^\sigma})^{1-\theta},\text{
where }\theta=\frac{4N-8}{4N-7}.
\end{equation}
Similarly, we have
\begin{equation} \label{intep1}
\begin{split}
\sigma\ns{\eta}_{2(N+2)+1}
&\le \sigma\norm{\eta}_{2(N+2)+3/2}  \norm{\eta}_{2(N+2)+1/2}
\\& \le \sigma\norm{\eta}_{2(N+2)+3/2}    \norm{\eta}_{2(N+2)-1/2}^\frac{4N-9}{4N-7} \norm{\eta}_{4N}^\frac{2}{4N-7}\lesssim( {\mathcal{D}_{N+2}^\sigma})^\theta({\mathcal{E}_{2N}^\sigma})^{1-\theta}.
\end{split}
\end{equation}
Hence, in light of \eqref{intep0}--\eqref{intep1} we may deduce
\begin{equation} \label{intep}
{\mathcal{E}}_{N+2}^\sigma\lesssim({\mathcal{D}}_{N+2}^\sigma)^\theta({\mathcal{E}}_{2N}^\sigma)^{1-\theta}.
\end{equation}

Now since by Proposition \ref{Dgle},
\begin{equation}
\sup_{0\le r\le t}\mathcal{E}_{2N}^\sigma(r)\lesssim
\mathcal{E}_{2N}^\sigma(0)+ \mathcal{F}_{2N}(0):=\mathcal{M}_0,
\end{equation}
we obtain from  \eqref{intep} that
\begin{equation} \label{u2}
\mathcal{E}_{N+2}^\sigma\lesssim\mathcal{M}_0^{1-\theta}
(\mathcal{D}_{N+2}^\sigma)^\theta.
\end{equation}
Hence by \eqref{u1} and \eqref{u2}, there exists some constant $C>0$
such that
\begin{equation}
\frac{d}{dt} \mathcal{E}_{N+2}^\sigma+\frac{C}{\mathcal{M}_0^s}
(\mathcal{E}_{N+2}^\sigma)^{1+s}\le 0,\ \text{ where } s =
\frac{1}{\theta}-1 = \frac{1}{4N-8}.
\end{equation}
Solving this differential inequality directly, we obtain
\begin{equation} \label{u3}
\mathcal{E}_{N+2}^\sigma(t)\le \frac{\mathcal{M}_0}{(\mathcal{M}_0^s
+ s C( \mathcal{E}_{N+2}^\sigma(0))^s t)^{1/s} }
{\mathcal{E}}_{N+2}^\sigma(0).
\end{equation}
Using that ${\mathcal{E}}_{N+2}^\sigma(0)\lesssim\mathcal{M}_0 $ and
the fact $1/s=4n-8>1$, we obtain from \eqref{u3} that
\begin{equation}
{\mathcal{E}}_{N+2}^\sigma(t)\lesssim
\frac{\mathcal{M}_0}{(1+sCt)^{1/s} }\lesssim
\frac{\mathcal{M}_0}{(1+t^{1/s}) } =
\frac{\mathcal{M}_0}{(1+t^{4N-8}) }.
\end{equation}
This directly implies \eqref{n+2}.
\end{proof}

Now we can arrive at our ultimate energy estimates for
$\mathcal{G}_{2N}^\sigma$.
\begin{thm}\label{Ap}
There exists a universal $0 < \delta < 1$ so that if $
\mathcal{G}_{2N}^\sigma(T) \le \delta$, then
\begin{equation}\label{Apriori}
 \mathcal{G}_{2N}^\sigma(t) \ls\mathcal{G}_{2N}^\sigma(0) \text{ for all }0 \le t \le
 T.
\end{equation}
\end{thm}
\begin{proof}
The conclusion follows directly from the definition of
$\mathcal{G}_{2N}^\sigma$ and Propositions
\ref{p_f_bound}--\ref{decaylm}.
\end{proof}

\subsection{The limit as $\sigma\rightarrow0$: the proof of Theorem \ref{gwp}}

We can now prove a $\sigma$-independent global well-posedness of \eqref{geometric} which allow us to pass to the limit as $\sigma\rightarrow0$ on the whole time interval $[0,\infty)$.
 \begin{proof}[Proof of Theorem \ref{gwp}]
 Assume that the initial data $ u_0^\sigma$ and $\eta_0^\sigma$ satisfy the assumptions of Theorem \ref{gwp} for $0<\sigma\le 1$. For each $\sigma$, we let $(u^\sigma,p^\sigma,\eta^\sigma)$ be the local solution to \eqref{geometric} on $[0,T_0]$ produced by Theorem \ref{lwp}. By the estimate \eqref{Apriori} of Proposition \ref{Ap} and \eqref{2l_Fe211f}, we have
 \begin{equation}
 \mathcal{G}_{2N}^\sigma(t) \ls\mathcal{G}_{2N}^\sigma(0)\ls P(\mathcal{E}_0^\sigma)\ls  \mathcal{E}_0^\sigma \text{ for all }0 \le t \le
 T_0.
\end{equation}
This allows us to employ a continuity argument as in Section 9 of \cite{GT_per} to prove that there exists a universal $\delta_0>0$ such that if $\mathcal{E}_0^\sigma\le \delta_0$, then the solution $(u^\sigma,p^\sigma,\eta^\sigma)$ is a unique global solution to \eqref{geometric} on $[0,\infty)$ and satisfies the $\sigma$-independent global estimates \eqref{zz22}.

As for Theorem \ref{lwp}, the $\sigma$-independent estimates \eqref{zz22} yield the strong convergence of the whole family of $(u^\sigma,p^\sigma,\eta^\sigma)$ to the unique limit $(u,p,\eta)$ which is the unique strong solution of \eqref{geometric} for $\sigma=0$ on $[0,\infty)$ that satisfies the estimates \eqref{zz33}. The strong convergence is in any functional space that $\mathcal{G}_{2N}^0$ can compactly embed into, which is more than sufficient for us to pass to the limit as $\sigma\rightarrow0$ in \eqref{geometric} for each $t\in[0,\infty)$. We thus conclude the global-in-time zero surface tension limit  and complete the proof of Theorem \ref{gwp}.
\end{proof}

\subsection{Byproduct: a global well-posedness of \eqref{geometric} with surface tension}
For each fixed $\sigma>0$, we can establish from the proof of Theorem \ref{gwp} a global well-posedness of \eqref{geometric} with exponential decay rate. In this subsection, we will allow the constants $C$, $\delta_0$, etc. depend on $\sigma$.
\begin{thm}\label{globall}
Let $\sigma>0$.
Assume that the initial data $ u_0$ and $\eta_0$ satisfy the $(2N)^{th}$ compatibility conditions  \eqref{l_comp_cond_2N}, and that $\eta_0$ satisfies the zero average condition \eqref{z_avg}. There exists a $\delta_0>0$ such that if $\ns{u_0}_{4N}+\ns{\eta_0}_{4N+1}\le \delta_0$, then there exists a unique global solution $(u,p,\eta)$ to \eqref{geometric} on  $[0,\infty)$.  The solution $(u,p,\eta)$ obeys the estimate
\begin{equation}\label{esti}
\sup_{t\ge 0} \se{2N}^1(t) +\int_0^t \sd{2N}^1(s)ds\le C\left(\ns{u_0}_{4N}+\ns{\eta_0}_{4N+1}\right).
\end{equation}
Moreover, there exist constants $C_1,C_2>0$ such that
\begin{equation} \label{esti2}
\se{2N}^1(t)
\le  C_1 \se{2N}^1(0)e^{-C_2 t}.
\end{equation}
\end{thm}
\begin{proof}
Since $\f \le C {\mathcal{E}}_{2N}^1$, we may derive from \eqref{q0} that
\begin{equation}
\bar{\mathcal{E}}_{2N}^1 \lesssim{\mathcal{E}}_{2N}^1 \lesssim
\bar{\mathcal{E}}_{2N}^1, \text{ and } \bar{\mathcal{D}}_{2N}
\lesssim {\mathcal{D}}_{2N}^1 \lesssim \bar{\mathcal{D}}_{2N} .
\end{equation}
We may then deduce from
Proposition \ref{i_temporal_evolution 2N} and Proposition
\ref{p_upper_evolution 2N} that there exists an instantaneous
energy, which is equivalent to ${\mathcal{E}}_{2N}^1$, but for
simplicity is still denoted by ${\mathcal{E}}_{2N}^1$, such
that
\begin{equation} \label{2N2N}
\frac{d}{dt} {\mathcal{E}}_{2N}^1 +  {\mathcal{D}}_{2N}^1\le
0.
\end{equation}
Integrating the inequality \eqref{2N2N} directly in time, we get \eqref{esti}. But we see that ${\mathcal{E}}_{2N}^1\ls  {\mathcal{D}}_{2N}^1$, so the decay estimate \eqref{esti} follows by \eqref{2N2N} and an application of Gronwall's inequality.
\end{proof}

\end{document}